\documentclass[11pt]{amsbook}

\usepackage{amssymb}

\def\Re{\textrm{Re}} 
\def\11{{\rm 1~\hspace{-1.4ex}l} }
\def\R{\mathbb R}
\def\C{\mathbb C}
\def\Z{\mathbb Z}
\def\N{\mathbb N}
\def\T{\mathbb T}
\def\E{\mathbb E}

\newcommand{\ft}{\widehat}

\def\beq{\begin{equation}}
\def\eeq{\end{equation}}

\newtheorem{prop}{{ Proposition}}
\newtheorem{theo}[prop]{{ Theorem}}
\newtheorem{cor}[prop]{{ Corollary}}
\newtheorem{lemm}[prop]{{ Lemma}}
\newtheorem{defi}[prop]{{ Definition}}
\newtheorem{rema}[prop]{{ Remark}}

\numberwithin{equation}{chapter}
\numberwithin{prop}{chapter}

\begin{document}

\title[Random data wave equations ]
{
Random data wave equations
}

\author{Nikolay Tzvetkov}

\address{Universit{\'e} de Cergy Pontoise} 
\email{nikolay.tzvetkov@u-cergy.fr}


\begin{abstract} 
{\large
Nowadays we have many methods allowing to exploit the regularising properties of the linear part of a nonlinear dispersive equation (such as the KdV equation, the nonlinear wave or the nonlinear Schr\"odinger equations) in order to prove well-posedness in low regularity Sobolev spaces.  By  well-posedness in low regularity Sobolev spaces we mean that less regularity than the one imposed by the energy methods is required (the energy methods do not exploit the dispersive properties of the linear part of the equation). In many cases these methods 
to prove well-posedness in low regularity Sobolev spaces
lead to optimal results in terms of the regularity of the initial data.  By optimal we mean that if one requires slightly less regularity then the corresponding Cauchy problem becomes ill-posed in the Hadamard sense. 
We call the Sobolev spaces in which these ill-posedness results hold {\it spaces of supercritical regularity}.
More recently, methods to prove  {\it probabilistic well-posedness}  in Sobolev spaces of supercritical regularity were developed. More precisely, 
by probabilistic well-posedness we mean that one endows the corresponding Sobolev space of supercritical regularity with a non degenerate probability measure and then one shows that almost surely with respect to this measure one can define a (unique) global flow. However, in most of the cases when the methods to prove  
 probabilistic well-posedness apply, there is no information about the measure transported by the flow. 
 Very recently, a method to prove that the transported measure is absolutely continuous with respect to the initial measure was developed. In such a situation, we have a measure which is quasi-invariant under the corresponding flow.
 \\
 
 The aim of these lectures is to present all of the above described developments in the context of the nonlinear wave equation.  
}

\end{abstract}

\maketitle

\tableofcontents

\chapter
[The deterministic Cauchy problem] 
{Deterministic Cauchy theory for the  $3d$ cubic wave equation }\label{Chapter1}
\section{Introduction}
In this chapter, we consider the cubic defocusing wave equation
\begin{equation}\label{3dNLW}
(\partial_t^2 -\Delta) u+u^3=0,
\end{equation}
where $u=u(t,x)$ is real valued, $t\in\R$, $x\in\T^3=(\R\slash (2\pi\Z))^3$ (the $3d$ torus).
In \eqref{3dNLW}, $\Delta$ denotes the Laplace operator, namely
$$
\Delta=\partial_{x_1}^2+\partial_{x_2}^2+\partial_{x_3}^2\,.
$$
Since \eqref{3dNLW} is  of second order in time, it is natural to complement it with two initial conditions 
\begin{equation}\label{data}
u(0,x)=u_{0}(x),\quad \partial_t u(0,x)=u_1(x)\,.
\end{equation}
In this chapter,  we will be studying the local and global well-posedness of the initial value problem \eqref{3dNLW}-\eqref{data} in Sobolev spaces via deterministic methods. 

The Sobolev spaces $H^s(\T^3)$ are defined as follows. For a function $f$ on $\T^3$ given by its Fourier series
$$
f(x)=\sum_{n\in\Z^3} \hat{f}(n)\, e^{in\cdot x},
$$
we define the Sobolev norm $H^s(\T^3)$ of $f$ as
$$
\| f \|_{H^s}^2=\sum_{n\in \Z^3}\langle n\rangle^{2s}\ |\hat{f}(n)|^2,
$$
where $\langle n\rangle=(1+|n|^2)^{1/2}$.  On has that
$$
\| f \|_{H^s}\approx \|D^s f\|_{L^2},\quad D\equiv (1-\Delta)^{1/2}\,.
$$
For integer values of $s$ one can also give an equivalent norm in the physical space as follows
$$
\|f\|_{H^s(\T^3)}\approx \sum_{|\alpha|\leq s}\| \partial_{x_1}^{\alpha_1}\partial_{x_2}^{\alpha_2}\partial_{x_3}^{\alpha_3}f\|_{L^2(\T^3)}\,,
$$
where the summation is taken over all multi-indexes $\alpha=(\alpha_1,\alpha_2,\alpha_3)\in \N^3$.

As we shall see, it will be of importance to understand the interplay between the linear and the nonlinear part of \eqref{3dNLW}.
Indeed, let us first consider the Cauchy problem 
$$
\partial_t^2 u+u^3=0,\quad u(0,x)=u_{0}(x),\quad \partial_t u(0,x)=u_1(x)
$$
which is obtained from \eqref{3dNLW} by neglecting the Laplcian. If we set 
$$
U=(U_1,U_2)\equiv (u,\partial_t u)^t
$$ 
than the last problem can be written as 
$$
\partial_t U=F(U), \quad F(U)=(U_2,-U_1^3)^t\,.
$$
On may wish to solve, at least locally, the last problem via the Cauchy-Lipschitz argument in the spaces  $H^{s_1}(\T^3)\times H^{s_2}(\T^3)$.
For such a purpose one should check that the vector field $F(U)$ is locally Lipschitz on these spaces.
Thanks to the Sobolev embedding $H^s(\T^3)\subset L^\infty(\T^3)$, $s>3/2$ we can see that the map $U_1\mapsto U_1^3$  is locally Lipschitz on $H^s(\T^3)$, $s>3/2$. It is also easy to check that the map
 $U_1\mapsto U_1^3$ is not continuous on $H^s(\T^3)$, $s<3/2$. A  more delicate argument shows that it is not continuous on $H^{3/2}(\T^3)$ either. Therefore, if we impose that $F(u)$ is locally Lipschitz on 
 $H^{s_1}(\T^3)\times H^{s_2}(\T^3)$ than we necessarily need to impose a regularity assumption $s_1>3/2$. As we shall see bellow the term containing the Laplacian in \eqref{3dNLW} will allow as to significantly relax this regularity assumption. 

On the other hand if we neglect the nonlinear term $u^3$ in \eqref{3dNLW}, we get the linear wave equation which is well-posed in $H^s(\T^3)\times H^{s-1}(\T^3)$ for {\it any} $s\in \R$, as it can be easily seen by the Fourier series description of the solutions of the linear wave equation (see the next section).  In other words the absence of a nonlinearity allows us to solve the problem in arbitrary singular Sobolev spaces. 

In summary, we expect that the Laplacian term in \eqref{3dNLW} will help us to prove the well-posedness of the problem \eqref{3dNLW} in singular Sobolev spaces while the nonlinear term $u^3$ will be responsible for the lack of well-posedness in singular spaces.
\section{Local and global well-posedness in $H^1\times L^2$}
\subsection{The free evolution}
We first define the free evolution, i.e. the map defining the solutions of the linear wave equation 
\begin{equation}\label{linear_wave}
(\partial_t^2-\Delta)u=0,\quad  u(0,x)=u_0(x), \quad \partial_t u(0,x)=u_1(x).
\end{equation}
Using the Fourier transform and solving the corresponding second order linear ODE's, we obtain that the solutions of \eqref{linear_wave} are generated by the map $S(t)$, defined as follows 
\begin{equation*}
S(t)(u_0,u_1)\equiv
\cos(t\sqrt{-\Delta})(u_0)+\frac{\sin(t\sqrt{-\Delta})}{\sqrt{-\Delta}}(u_1),
\end{equation*}
where
$$
\cos(t\sqrt{-\Delta})(u_0)\equiv
\sum_{n\in\Z^3} \cos(t|n|) \widehat{u_0}(n)\, e^{in\cdot x}
$$
and
$$
\frac{\sin(t\sqrt{-\Delta})}{\sqrt{-\Delta}}(u_1)
\equiv t\widehat{u_1}(0)+
\sum_{n\in\Z^3_{\star}} 
\frac{\sin(t|n|)}{|n|}\widehat{u_1}(n)\, e^{in\cdot x}\,,\quad  \Z^3_{\star}=\Z^3\backslash\{0\}\,.
$$
We have that $S(t)(u_0,u_1)$ solves \eqref{linear_wave} and if $(u_0,u_1)\in H^s\times H^{s-1}$, $s\in \R$ then $S(t)(u_0,u_1)$ is the unique solution of \eqref{linear_wave} in 
$C(\R;H^s(\T^3))$ such that its time derivative is in  $C(\R;H^{s-1}(\T^3))$. 
It follows directly from the definition that the operator 
$\bar{S}(t)\equiv (S(t),\partial_t S(t))$ is bounded on $H^s\times H^{s-1}$, $\bar{S}(0)={\rm Id}$ and $\bar{S}(t+\tau)=\bar{S}(t) \circ \bar{S}(\tau)$, for every real numbers $t$ and $\tau$.
In the proof of the boundedness on $H^s\times H^{s-1}$, we only use the boundedness of $\cos(t|n|)$ and $\sin(t|n|)$. As we shall see below one may use the oscillations of 
$\cos(t|n|)$ and $\sin(t|n|)$ for $|n|\gg 1$ in order to get more involved $L^p$, $p>2$ properties of the map $S(t)$.

Let us next consider the non homogeneous problem 
\begin{equation}\label{linear_wave_inhom}
(\partial_t^2-\Delta)u=F(t,x),\quad  u(0,x)=0, \quad \partial_t u(0,x)=0.
\end{equation}
Using the variation of the constants method, we obtain that the solutions of \eqref{linear_wave_inhom} are given by 
\begin{equation*}
u(t)=\int_{0}^t \frac{\sin((t-\tau)\sqrt{-\Delta})}{\sqrt{-\Delta}}((F(\tau))d\tau\,.
\end{equation*}
As a consequence, we obtain that the solution of the non homogeneous problem  \eqref{linear_wave_inhom} is one derivative smoother than the source term $F$. More precisely,  for every $s\in\R$, the solution of 
\eqref{linear_wave_inhom} satisfies the bound
\begin{equation}\label{wave_regularity}
\|u\|_{L^\infty([0,1];H^{s+1}(\T^3))}\leq C\|F\|_{L^1([0,1];H^s(\T^3))}\,.
\end{equation}
\subsection{The local well-posedness}
We state the local well-posedness result. 
\begin{prop}[local well-posedness]\label{prop.local}
Consider the cubic defocusing wave equation
\begin{equation}\label{model}
(\partial_t^2-\Delta)u+u^3=0\,,
\end{equation}
posed on $\T^3$.
There exist constants  $c$ and $C$ such that for every  $a\in\R$, 
every $\Lambda\geq 1$, every 
$$
(u_0,u_1)\in H^1(\T^3)\times L^2(\T^3)
$$
satisfying
\begin{equation}\label{uppp}
\|u_0\|_{H^1}+\|u_1\|_{L^2}\leq \Lambda
\end{equation}
there exists a unique solution of \eqref{model} on the time interval $[a,a+c \Lambda^{-2}]$ of \eqref{model} with initial data
$$
u(a,x)=u_0(x), \quad \partial_t u(a,x)=u_1(x)\,.
$$
Moreover the solution satisfies
$$
\|(u,\partial_t u)\|_{L^\infty([a,a+c \Lambda^{-2}],H^1(\T^3)\times L^2(\T^3))}\leq C\Lambda,
$$
$(u,\partial_t u)$ is unique in the class $L^\infty([a,a+c \Lambda^{-2}],H^1(\T^3)\times L^2(\T^3))$ and the dependence with respect to the initial data and with respect to the time is continuous.
Finally, if 
$$
(u_0,u_1)\in H^s(\T^3)\times H^{s-1}(\T^3)
$$
for some $s\geq 1$ then there exists $c_s>0$ such that
$$
(u,\partial_t u)\in
C([a,a+c_s \Lambda^{-2}];H^s(\T^3)\times H^{s-1}(\T^3))\,.
$$
\end{prop}
\begin{proof}
If $u(t,x)$ is a solution of \eqref{model} then so is $u(t+a,x)$.
Therefore, it suffices to consider the case $a=0$.

Thanks to the analysis of the previous section, we obtain that we should solve the integral equation
\begin{equation}\label{Duhamel}
u(t)=S(t)(u_0,u_1)-\int_{0}^t \frac{\sin((t-\tau)\sqrt{-\Delta})}{\sqrt{-\Delta}}((u^3(\tau))d\tau\,.
\end{equation}
Set
$$
\Phi_{u_0,u_1}(u)\equiv  S(t)(u_0,u_1)-\int_{0}^t \frac{\sin((t-\tau)\sqrt{-\Delta})}{\sqrt{-\Delta}}((u^3(\tau))d\tau.
$$
Then for $T\in (0,1]$,  we define $X_T$ as
$$
X_{T}\equiv C([0,T];H^1(\T^3)),
$$
endowed with the natural norm
$$
\|u\|_{X_T}=\sup_{0\leq t\leq T}\|u(t)\|_{H^1(\T^3)}\,.
$$
Using  the boundedness properties of $\bar{S}$ on $H^s\times H^{s-1}$ explained in the previous section and using  
the Sobolev embedding $H^1(\T^3)\subset L^6(\T^3)$, we get
\begin{eqnarray*}
\|\Phi_{u_0,u_1}(u)\|_{X_T} & \leq &  C\big(\|u_0\|_{H^1}+\|u_1\|_{L^2}+T\sup_{\tau\in[0,T]}\|u(\tau)\|_{L^6}^3\big)
\\
& \leq &
C\big(\|u_0\|_{H^1}+\|u_1\|_{L^2}+CT\|u\|^3_{X_T}\big).
\end{eqnarray*}
It is now clear that for $T= c \Lambda^{-2}$ , $c\ll 1$ the map $\Phi_{u_0,u_1}$ sends the ball
$$
B\equiv (u:\|u\|_{X_T)}
\leq 2C\Lambda)
$$
into itself.  Moreover, by a similar arguments involving the the Sobolev embedding $H^1(\T^3)\subset L^6(\T^3)$ and the H\"older inequality, we obtain the estimate
 \begin{equation}\label{difference}
 \|\Phi_{u_0,u_1}(u)-\Phi_{u_0,u_1}(\tilde{u})\|_{X_T}  \leq    CT\|u-\tilde{u}\|_{X_T} \big(\|u\|^2_{X_T} +\|\tilde{u}\|^2_{X_T}\big).
 \end{equation}
 Therefore, with our choice of $T$, we get that
 $$
 \|\Phi_{u_0,u_1}(u)-\Phi_{u_0,u_1}(\tilde{u})\|_{X_T}  \leq    \frac{1}{2}\|u-\tilde{u}\|_{X_T}\,,\quad u,\tilde{u}\in B\,.
 $$
Consequently the map $\Phi_{u_0,u_1}$ is a contraction on $B$. The fixed point of this contraction defines the solution $u$ on $[0,T]$ we are looking for.
The estimate of $\|\partial_t u\|_{L^2}$ follows by differentiating
in $t$ the Duhamel formula \eqref{Duhamel}.  
Let us now turn to the uniqueness. 
Let $u,\tilde{u}$ be two solutions of \eqref{model} with the same initial data in the space $X_T$ for some $T>0$. Then for $\tau\leq T $, we can write similarly to  \eqref{difference}
\begin{equation}\label{diff-bis}
 \|\Phi_{u_0,u_1}(u)-\Phi_{u_0,u_1}(\tilde{u})\|_{X_\tau}  \leq    C\tau \|u-\tilde{u}\|_{X_\tau} \big(\|u\|^2_{X_T} +\|\tilde{u}\|^2_{X_T}\big).
\end{equation}
Let us take $\tau$ such that
$$
 C\tau \big(\|u\|^2_{X_T} +\|\tilde{u}\|^2_{X_T}\big)<\frac{1}{2}.
 $$
 This fixes the value of $\tau$. 
 Thanks to \eqref{diff-bis}, we obtain that $u$ and $\tilde{u}$ are the same on $[0,\tau]$.  Next, we cover the interval $[0,T]$ by intervals of size $\tau$ and we inductively obtain that $u$ and $\tilde{u}$ are the same on each interval of size $\tau$.
 This yields the uniqueness statement.  
 
 The continuous dependence with respect to time follows from the Duhamel formula representation of the solution of \eqref{Duhamel}. 
 The continuity with respect to the initial data follows from the estimates on the difference of two solutions we have just performed. Notice that we also obtain uniform continuity of the map data-solution on bounded subspaces of $H^1\times L^2$.
 
 Let us finally turn to the propagation of higher regularity. 
 Let $(u_0,u_1)\in H^1\times L^2$ such that \eqref{uppp} holds satisfy the additional regularity property 
 $(u_0,u_1)\in H^s\times H^{s-1}$ for some $s>1$.
 We will show that the corresponding solution remains in $H^s\times H^{s-1}$ {\it in the (essentially) whole} time of existence. For $s\geq 1$, we define 
 $X^s_T$ as
$$
X^s_{T}\equiv C([0,T];H^s(\T^3)),
$$
endowed with the norm
$$
\|u\|_{X^s_T}=\sup_{0\leq t\leq T}\|u(t)\|_{H^s(\T^3)}\,.
$$
We have that the solution with data $(u_0,u_1)\in H^s\times H^{s-1}$  remains in this space for time intervals of order 
$
 (1+\|u_0\|_{H^s}+\|u_1\|_{H^{s-1}})^{-2}
$
by a fixed point argument,  similar to the one we performed for data in $H^{1}\times L^2$.
We now show that the regularity is preserved for (the longer) time intervals of order 
$
(1+\|u_0\|_{H^1}+\|u_1\|_{L^{2}})^{-2}\,.
$
Coming back to \eqref{Duhamel}, we  can write 
 $$
 \|\Phi_{u_0,u_1}(u)\|_{X^s_T} \leq   C\big(\|u_0\|_{H^s}+\|u_1\|_{H^{s-1}}+
 T\sup_{\tau\in[0,T]}\|u^3(\tau)\|_{H^{s-1}}\big).
 $$
 Now using  the Kato-Ponce product inequality, we can obtain that for $\sigma\geq 0$, one has the bound 
\begin{equation}\label{L6}
\|v^3\|_{H^\sigma(\T^3)}\leq C\|D^\sigma v\|_{L^6(\T^3)}\, \|v\|_{L^6(\T^3)}^2\,.
\end{equation}
Using \eqref{L6} and applying the Sobolev embedding $H^1(\T^3)\subset L^6(\T^3)$, we infer that
$$
\|u^3(\tau)\|_{H^{s-1}}\lesssim \|D^{s-1} u(\tau)\|_{L^6}\|u(\tau)\|_{L^6}^2\lesssim
\|D^s u(\tau)\|_{L^2}\|u(\tau)\|_{H^1}^2\,.
$$
Therefore, we arrive at the bound 
$$
 \|\Phi_{u_0,u_1}(u)\|_{X^s_T} \leq   C\big(\|u_0\|_{H^s}+\|u_1\|_{H^{s-1}}+
 C_{s} T\sup_{\tau\in[0,T]} \|D^s u(\tau)\|_{L^2}\|u(\tau)\|_{H^1}^2\big)\,.
 $$
By construction of the solution we infer that if $T\leq c_s\Lambda^{-2}$ with $c_s$ small enough, we have that 
$$
\|u\|_{X^s_T}=
\|\Phi_{u_0,u_1}(u)\|_{X^s_T} \leq   C\big(\|u_0\|_{H^s}+\|u_1\|_{H^{s-1}}\big)+
 \frac{1}{2}\|u\|_{X^s_T}
 $$
which implies the propagation of the regularity statement for $u$. 
Strictly speaking, one should apply a bootstrap argument starting from the propagation of the regularity on times of order 
$
(1+\|u_0\|_{H^s}+\|u_1\|_{H^{s-1}})^{-2}
$
and then extend the regularity propagation to the longer interval $[0,c_s\Lambda^{-2}]$.
One estimates similarly $\partial_t u$ in $H^{s-1}$ by differentiating the Duhamel formula with respect to $t$. 
The continuous dependence with respect to time in $H^s\times H^{s-1}$ follows once again from the Duhamel formula \eqref{Duhamel}. 
This completes the proof of Proposition~\ref{prop.local}.
\end{proof}
\begin{theo}[global well-posedness]\label{prop.global}
For every 
$
(u_0,u_1)\in H^1(\T^3)\times L^2(\T^3)
$
the local solution of  the cubic defocusing wave equation 
\begin{equation*}
(\partial_t^2-\Delta)u+u^3=0\,,\quad
u(0,x)=u_0(x), \quad \partial_t u(0,x)=u_1(x)
\end{equation*}
can be extended globally in time. 
It is unique in the class $C(\R ;H^1(\T^3)\times L^{2}(\T^3))$ and there exists a constant $C$  depending only on $\|u_0\|_{H^1}$ and $\|u_1\|_{L^2}$  such that for every $t\in\R$,
$$
\|u(t)\|_{H^1(\R)}\leq C.
$$
If in addition 
$
(u_0,u_1)\in H^s(\T^3)\times H^{s-1}(\T^3)
$
for some $s\geq 1$ then 
$$
(u,\partial_t u)\in
C(\R ;H^s(\T^3)\times H^{s-1}(\T^3))\,.
$$
\end{theo}
\begin{rema}
{\rm One may obtain global weak solutions of the cubic defocusing wave equation for data in $H^1\times L^2$ via compactness arguments.
The uniqueness and the propagation of regularity statements of Theorem~\ref{prop.global} are the major differences with respect to the weak solutions.
}
\end{rema}
\begin{proof}[Proof of  Theorem~\ref{prop.global}]
The key point is the conservation of the energy displayed in the following lemma.
\begin{lemm}\label{vg}
There exist $c>0$ and $C>0$ such that for every $(u_0,u_1)\in H^1(\T^3)\times L^2(\T^3)$
the local solution of  the  cubic defocusing wave equation, with data $(u_0,u_1)$, constructed in Proposition~\ref{prop.local}
is defined on $[0,T]$ with 
$$
T=c(1+\|u_0\|_{H^1(\T^3)}+\|u_1\|_{L^2(\T^3)})^{-2}
$$
and
\begin{multline}\label{island}
\int_{\T^3}  \big((\partial_t u(t,x))^2+|\nabla_x u(t,x)|^2+\frac{1}{2}u^4(t,x)\big)dx
\\
=
\int_{\T^3}  \big((u_1(x))^2+|\nabla_x u_0(x)|^2+\frac{1}{2}u_0^4(x)\big)dx,\quad t\in [0,T].
\end{multline}
As a consequence, for $t\in[0,T]$,
$$
\|u(t)\|_{H^1(\T^3)}+\|\partial_t u(t)\|_{L^2(\T^3)}\leq C
\big(1+\|u_0\|_{H^1(\T^3)}^2+\|u_1\|_{L^2(\T^3)}
\big).
$$
\end{lemm}
\begin{rema}
{\rm
Using the invariance with respect to translations in time, we can state  Lemma~\ref{vg} with initial data at an arbitrary initial time.
}
\end{rema}
\begin{proof}[Proof of Lemma~\ref{vg}]
We apply  Proposition~\ref{prop.local} with $\Lambda= \|u_0\|_{H^1}+\|u_1\|_{L^2}$ and we take $T=c_{10}\Lambda^{-2}$, where $c_{10}$ is the small constant involved in the propagation of the $H^{10}\times H^9$ regularity.
Let $(u_{0,n},u_{1,n})$ be a sequence in $H^{10}\times H^{9}$ which converges to $(u_0,u_1)$ in $H^1\times L^2$ and such that 
$$
  \|u_{0,n}\|_{H^1}+\|u_{1,n}\|_{L^2}\leq 
 \|u_0\|_{H^1}+\|u_1\|_{L^2}\,.
 $$
Let $u_n(t)$ be the solution of  the  cubic defocusing wave equation, with data $(u_{0,n},u_{1,n})$.
By  Proposition~\ref{prop.local} these solutions are defined on $[0,T]$ and they keep their  $H^{10}\times H^{9}$ regularity {\it on the same time interval}. 
We multiply the equation 
$$
(\partial_t^2-\Delta)u_n+u_n^3=0
$$
by $\partial_t u_n$. Using the  regularity properties of $u_n(t)$, after integrations by parts, we arrive at
$$
\frac{d}{dt}\Big[\int_{\T^3}  \big((\partial_t u_n(t,x))^2+|\nabla_x u_n(t,x)|^2+\frac{1}{2}u_n^4(t,x)\big)dx\Big]=0
$$
which implies the identity 
\begin{multline}\label{unn}
\int_{\T^3}  \big((\partial_t u_n(t,x))^2+|\nabla_x u_n(t,x)|^2+\frac{1}{2}u_n^4(t,x)\big)dx
\\
=
\int_{\T^3}  \big((u_{1,n}(x))^2+|\nabla_x u_{0,n}(x)|^2+\frac{1}{2}u_{0,n}^4(x)\big)dx,\quad t\in [0,T].
\end{multline}
We now pass to the limit $n\longrightarrow \infty$ in \eqref{unn}. 
The right hand-side converges to 
$$
\int_{\T^3}  \big((u_1(x))^2+|\nabla_x u_0(x)|^2+\frac{1}{2}u_0^4(x)\big)dx
$$
by the definition of $(u_{0,n},u_{1,n})$ (we invoke the Sobolev embedding for the convergence of the $L^4$ norms) . 
The right hand-side of \eqref{unn} converges to
$$
\int_{\T^3}  \big((\partial_t u(t,x))^2+|\nabla_x u(t,x)|^2+\frac{1}{2}u^4(t,x)\big)dx
$$
by the continuity of the flow map established in Proposition~\ref{prop.local}.
Using the compactness of $\T^3$ and the H\"older inequality, we have that
$$
\|u\|_{L^2(\T^3)}\leq C\|u\|_{L^4(\T^3)}\leq C(1+\|u\|^2_{L^4(\T^3)})
$$
and therefore 
$$
\|u(t)\|^2_{H^1(\T^3)}+\|\partial_t u(t)\|^2_{L^2(\T^3)}
$$
is bounded by
$$
C\,\int_{\T^3}  \big(1+
(\partial_t u(t,x))^2+|\nabla_x u(t,x)|^2+\frac{1}{2}u^4(t,x)\big)dx\,.
$$
Now, using \eqref{island} and the Sobolev inequality 
$$
\|u\|_{L^4(\T^3)}\leq C\|u\|_{H^1(\T^3)}\, ,
$$
we obtain that for $t\in [0,T]$,
$$
\|u(t)\|^2_{H^1(\T^3)}+\|\partial_t u(t)\|^2_{L^2(\T^3)}\leq C
\big(
1+\|u_0\|_{H^1(\T^3)}^4+\|u_1\|^2_{L^2(\T^3)}
\big).
$$
This completes the proof of Lemma~\ref{vg}.
\end{proof}
Let us now complete the proof of Theorem~\ref{prop.global}.
Let  $(u_0,u_1)\in H^1(\T^3)\times L^2(\T^3)$.
Set
$$
T=
c\big(C\big(1+\|u_0\|_{H^1(\T^3)}^2+\|u_1\|_{L^2(\T^3)}
\big)
\big)^{-2},
$$
where the constants $c$ and $C$ are defined in Lemma~\ref{vg}.
We now observe that we can use Proposition~\ref{prop.local} and  Lemma~\ref{vg} on the intervals $[0,T]$, $[T,2T]$, $[2T,3T]$, and so on and therefore we extend the solution with data $(u_0,u_1)$ on $[0,\infty)$.
By the time reversibility of the wave equation we similarly can construct the solution for negative times. 
More precisely, the free evolution $S(t)(u_0,u_1)$  well-defined for all $t\in\R$ and one can prove in the same way the natural counterparts of Proposition~\ref{prop.local} and  Lemma~\ref{vg} for negative times. 
The propagation of higher Sobolev regularity globally in time follows from  Proposition~\ref{prop.local} while the $H^1$ a priori bound on the solutions follows from  Lemma~\ref{vg}.
This completes the proof of Theorem~\ref{prop.global}.
\end{proof}
\begin{rema}
{\rm
One may proceed slightly differently in the proof of  Theorem~\ref{prop.global} by observing that as a consequence of Proposition~\ref{prop.local}, if a local solution with $H^1\times L^2$ data blows-up at time $T^\star<\infty$
then 
\begin{equation}\label{b-upp}
\lim_{t\rightarrow T^{\star}}
\|(u(t),\partial_t u(t))\|_{H^1(\T^3)\times L^{2}(\T^3)}
=\infty.
\end{equation}
The statement \eqref{b-upp} is in contradiction with the energy conservation law. 
}
\end{rema}
\begin{rema}
{\rm Observe that the nonlinear problem 
\begin{equation}\label{eng1}
(\partial_t^2-\Delta)u+u^3=0
\end{equation}
behaves better than the linear problem 
\begin{equation}\label{eng2}
(\partial_t^2-\Delta)u=0
\end{equation}
with respect to the $H^1$ global in time bounds. 
Indeed, Theorem~\ref{prop.global} establishes that the solutions of \eqref{eng1} are bounded in $H^1$ as far as the initial data is in $H^1\times L^2$.
On the other hand one can consider 
$
u(t,x)=t
$ 
which is a solution of the linear wave equation \eqref{eng2} on $\T^3$ with data in  $H^1\times L^2$ and its $H^1$ norm is  clearly growing in time.
}
\end{rema}
\begin{rema}
{\rm
The sign in front of the nonlinearity is not of importance for Proposition~\ref{prop.local}.
One can therefore obtain the local well-posedness of the cubic focusing wave equation
\begin{equation}\label{focusing}
(\partial_t^2-\Delta)u-u^3=0,
\end{equation}
posed on $\T^3$, with data in $H^1(\T^3)\times L^2(\T^3)$. 
However, the sign in front of the nonlinearity is of crucial importance in the proof of Theorem~\ref{prop.global}.
Indeed, one has that
$$
u(t,x)=\frac{\sqrt{2}}{1-t}
$$
is a solution of \eqref{focusing}, posed on $\T^3$  with data $(\sqrt{2},-\sqrt{2})$ which is not defined globally in time (it blows-up in $H^1\times L^2$ at $t=1$).
}
\end{rema}
\section{The Strichartz estimates}
In the previous section, we solved globally in time the cubic defocusing wave equation in $H^1\times L^2$. 
One may naturally ask whether it is possible to extend these results to the more singular Sobolev spaces $H^s\times H^{s-1}$ for some $s<1$. It turns out that this is possible by invoking more refined properties of the map $S(t)$ defining the free evolution. The proof of these properties uses in an essential way the time oscillations in $S(t)$ and can be quantified as the $L^p$, $p>2$ mapping properties of $S(t)$ (cf. \cite{GiVe,LS}).
\begin{theo}[Strichartz inequality for the wave equation]\label{Strichartz}
Let $(p,q)\in \R^2$ be such that $2<p\leq \infty$ and 
$
\frac{1}{p}+\frac{1}{q}=\frac{1}{2}.
$ 
Then we have the estimate 
$$
\|S(t)(u_0,u_1)\|_{L^p([0,1];L^q(\T^3))}\leq C\big(\|u_0\|_{H^{\frac{2}{p}}(\T^3)}+\|u_1\|_{H^{\frac{2}{p}-1}(\T^3)}\big).
$$
\end{theo}
We shall use that the solutions of the wave equation satisfy a finite propagation speed property which will allow us to deduce the result of Theorem~\ref{Strichartz} from the corresponding Strichartz estimate for the wave equation on the euclidean space.
Consider therefore the wave equation
\begin{equation}\label{linear_wave_e}
(\partial_t^2-\Delta)u=0,\quad  u(0,x)=u_0(x), \quad \partial_t u(0,x)=u_1(x),
\end{equation}
where now the spatial variable $x$ belongs to $\R^3$ and the initial data $(u_0,u_1)$ belong to $H^s(\R^3)\times H^{s-1}(\R^3)$.
Using the Fourier transform on $\R^3$, we can solve \eqref{linear_wave_e} and obtain that the solutions are  generated by the map $S_{e}(t)$, defined as  
\begin{equation*}
S_{e}(t)(u_0,u_1)\equiv
\cos(t\sqrt{-\Delta_{\R^3}})(u_0)+\frac{\sin(t\sqrt{-\Delta_{\R^3}})}{\sqrt{-\Delta_{\R^3}}}(u_1),
\end{equation*}
where for $u_0$ and $u_1$ in the Schwartz class,
$$
\cos(t\sqrt{-\Delta_{\R^3}})(u_0)\equiv
\int_{\R^3}\cos(t|\xi|) \widehat{u_0}(\xi)\, e^{i\xi\cdot x}d\xi
$$
and
$$
\frac{\sin(t\sqrt{-\Delta_{\R^3}})}{\sqrt{-\Delta_{\R^3}}}(u_1)
\equiv 
\int_{\R^3} 
\frac{\sin(t|\xi|)}{|\xi|}\widehat{u_1}(\xi)\, e^{i\xi\cdot x}d\xi\,,
$$
where $ \widehat{u_0}$  and $\widehat{u_1}$ 
are the Fourier transforms of $u_0$  and $u_1$ respectively.
By density, one then extends $S_{e}(t)(u_0,u_1)$ to a bounded map from $H^s(\R^3)\times H^{s-1}(\R^3)$  to $H^s(\R^3)$  for any $s\in\R$. 
The next lemma displays the finite propagation speed property of $S_{e}(t)$.
\begin{prop}[finite propagation speed]\label{FPS}
Let $(u_0,u_1)\in H^s(\R^3)\times H^{s-1}(\R^3)$ for some $s\geq 0$ be such that 
$$
{\rm supp}(u_0)\cup {\rm supp}(u_1)\subset \{x\in\R^3\,:\, |x-x_0|\leq R\},
$$
for some $R>0$ and $x_0\in\R^3$. Then for $t\geq 0$,
$$
{\rm supp}(S_{e}(t)(u_0,u_1))\subset \{x\in\R^3\,:\, |x-x_0|\leq t+R\}.
$$
\end{prop} 
\begin{proof}
The statement of Proposition~\ref{FPS} (and even more precise localisation property) follows from the Kirchoff formula representation of the solutions of the three dimensional wave equation. 
Here we will present another proof which has the advantage to extend to an arbitrary dimension and to variable coefficient settings. 
By the invariance of the wave equation with respect to spatial translations, we can assume that $x_0=0$. 
We need to prove Proposition~\ref{FPS} only for (say) $s\geq 100$ which ensures by the Sobolev embedding that the solutions we study are of class $C^2(\R^4)$. 
We than can treat the case of an arbitrary $(u_0,u_1)\in H^s(\R^3)\times H^{s-1}(\R^3)$, $s\geq 0$ by observing that 
\begin{equation}\label{convolution}
\rho_{\varepsilon}\star  S_{e}(t)(u_0,u_1) =S_{e}(t)(\rho_{\varepsilon}\star u_0,\rho_{\varepsilon }\star u_1),
\end{equation}
where $\rho_{\varepsilon}(x)=\varepsilon^{-3}\rho(x/\varepsilon)$, $\rho\in C^\infty_0(\R^3)$, $0\leq \rho\leq 1$, $\int\rho =1$.
It suffices then to pass to the limit $\varepsilon\rightarrow 0$ in \eqref{convolution}.
Indeed, for $\varphi \in C^\infty_0(|x|>t+R)$, $S_{e}(t)(\rho_{\varepsilon}\star u_0,\rho_{\varepsilon }\star u_1)(\varphi)$ is zero for $\varepsilon$ small enough while 
$\rho_{\varepsilon}\star  S_{e}(t)(u_0,u_1)(\varphi)$ converges to $S_{e}(t)(u_0,u_1)(\varphi)$.

Therefore, in the remaining of the proof of  Proposition~\ref{FPS}, we shall assume that $S_{e}(t)(u_0,u_1)$ is a $C^2$ solution of the $3d$ wave equation.
The main point in the proof is the following lemma.
\begin{lemm}\label{cone}
Let $x_0\in\R^3$, $r>0$ and let $S_{e}(t)(u_0,u_1)$ be a $C^2$ solution of the $3d$  linear wave equation.
Suppose that $u_0(x)=u_1(x)=0$ for $|x-x_0|\leq r$.
Then $S_{e}(t)(u_0,u_1)=0$ in the cone $C$ defined by
$$
C=\{(t,x)\in \R^{4}\,:\, 0\leq t\leq r,\, \, |x-x_0|\leq r-t  \}.
$$
\end{lemm}
\begin{proof}
Let $u(t,x)=S_{e}(t)(u_0,u_1)$.  For $t\in [0,r]$, we set
$$
E(t)\equiv \frac{1}{2}\int_{B(x_0,r-t)}\big(
(\partial_t u)^2(t,x)+|\nabla_x u(t,x)|^2
\big)dx,
$$
where $B(x_0,r-t)=\{x\in\R^3\,:\, |x|\leq r-t \}$.  Then using the Gauss-Green theorem and the equation solved by $u$, we obtain that
$$
\dot{E}(t)= -\frac{1}{2}\int_{\partial B}\big(
(\partial_t u)^2(t,y)+|\nabla_x u(t,y)|^2-2\partial_t u(t,y)\nabla_x u(t,y)\cdot \nu(y)
\big)dS(y),
$$
where $\partial B\equiv \{x\in\R^3\,:\, |x|= r-t \}$, $dS(y)$ is the volume element associated with $\partial B$ and $\nu(y)$ is the outer unit normal to  $\partial B$.
We clearly have 
$$
2\partial_t u(t,y)\nabla_x u(t,y)\cdot \nu(y)
\leq 
(\partial_t u)^2(t,y)+|\nabla_x u(t,y)|^2,
$$
which implies that $\dot{E}(t)\leq 0$. Since $E(0)=0$ we obtain that $E(t)=0$ for every $t\in [0,r]$. 
This in turn implies that $u(t,x)$ is a constant in $C$. We also know that $u(0,x)=0$ for $|x-x_0|\leq r$. Therefore $u(t,x)=0$ in $C$.
This completes the proof of Lemma~\ref{cone}. 
\end{proof}
Let us now complete the proof of Proposition~\ref{FPS}.
Let $t_0\in \R$ and $y\in\R^3$ such that $|y|>R+t_0$. We need to show that $u(t_0,y)=0$.
Consider the cone $C$ defined by 
$$
C=\{(t,x)\in \R^{4}\,:\, 0\leq t\leq t_0,\,\, |x-y|\leq t_0-t  \}.
$$
Set $B\equiv C\cap \{(t,x)\in\R^4: t=0 \}$.
We have that 
$$
B=\{(t,x)\in \R^{4}\,:\, t=0,\,\, |x-y|\leq t_0  \}
$$
and therefore by the definition of $t_0$ and $y$ we have that  
\begin{equation}\label{empty}
B\cap \{(t,x)\in \R^{4}\,:\, t=0,\,\, |x|\leq R \}=\emptyset.
\end{equation}
Therefore $u(0,x)=\partial_t u(0,x)$ for $|x-y|\leq t_0$. 
Using Lemma~\ref{cone}, we obtain that $u(t,x)=0$ in $C$. In particular $u(t_0,y)=0$. 
This completes the proof of Proposition~\ref{FPS}.
\end{proof}
Using  Proposition~\ref{FPS} and a decomposition of the initial data associated with
a partition of unity corresponding to a covering of $\T^3$ by sufficiently small balls, we obtain that the result of Theorem~\ref{Strichartz} is a consequence of the following statement. 
\begin{prop}[local in time Strichartz inequality for the wave equation on $\R^3$]
\label{Strichartz_R3}
Let $(p,q)\in \R^2$ be such that $2<p\leq \infty$ and $\frac{1}{p}+\frac{1}{q}=\frac{1}{2}$. Then we have the estimate
$$
\|S_{e}(t)(u_0,u_1)\|_{L^p([0,1];L^q(\R^3))}\leq C\big(\|u_0\|_{H^{\frac{2}{p}}(\R^3)}+\|u_1\|_{H^{\frac{2}{p}-1}(\R^3)}\big).
$$
\end{prop}
\begin{proof}
Let $\chi\in C^\infty_0(\R^3)$ be such that $\chi(x)=1$ for $|x|<1$. We then define the Fourier multiplier $\chi(D_x)$ by
\begin{equation}\label{hagel}
\chi(D_x)(f)=\int_{\R^3}\chi(\xi) \widehat{f}(\xi)\, e^{i\xi\cdot x}d\xi.
\end{equation}
Using a suitable Sobolev embedding in $\R^3$, we obtain that  for every $\sigma\in \R$,
$$
\big\|
\frac{\sin(t\sqrt{-\Delta_{\R^3}})}{\sqrt{-\Delta_{\R^3}}}(\chi(D_x) u_1)
\big\|_{L^p([0,1];L^q(\R^3))}
\leq  C\|u_1\|_{H^{\sigma}(\R^3)}\,.
$$
Therefore, by splitting $u_1$ 
as 
$$
u_1=\chi(D_x)(u_1)+(1-\chi(D_x))(u_1)
$$
and by expressing the $\sin$ and $\cos$ functions as combinations of exponentials, we observe that Proposition~\ref{Strichartz_R3} follows from the following statement.
\begin{prop}\label{simm}
Let $(p,q)\in \R^2$ be such that $2<p\leq \infty$ and 
$
\frac{1}{p}+\frac{1}{q}=\frac{1}{2}.
$
Then we have the estimate
$$
\big\|e^{\pm it\sqrt{-\Delta_{\R^3}}}(f)\big\|_{L^p([0,1];L^q(\R^3))}\leq  C\|f\|_{H^{\frac{2}{p}}(\R^3)}\,.
$$
\end{prop}
\begin{rema}\label{beat_sob}
{\rm
Let us make an important remark.
As a consequence of Proposition~\ref{simm} and a suitable Sobolev embedding, we obtain the estimate
\begin{equation}\label{bla}
\big\|e^{\pm it\sqrt{-\Delta_{\R^3}}}(f)\big\|_{L^2([0,1];L^\infty(\R^3))}\leq  C\|f\|_{H^{s}(\R^3)}\,, \quad s>1.
\end{equation}
Therefore, we obtain that for $f\in H^s(\R^3)$, $s>1$, the function
$e^{ it\sqrt{-\Delta_{\R^3}}}(f)$ which is a priori defined as an element of $C([0,1];H^s(\R^3))$ has the remarkable property that
$$
e^{ it\sqrt{-\Delta_{\R^3}}}(f)\in L^\infty(\R^3)
$$
for almost every $t\in [0,1]$.
Recall that the Sobolev embedding requires the condition $s>3/2$ in order to ensure that an $H^s(\R^3)$ function  is in $L^\infty(\R^3)$. 
Therefore, one may wish to see \eqref{bla} as an almost sure in $t$ improvement  (with $1/2$ derivative) of the Sobolev  embedding $H^{\frac{3}{2}+}(\R^3)\subset L^\infty(\R^3)$, under the evolution of the linear wave equation. 
}
\end{rema}
\begin{proof}[Proof of Proposition~\ref{simm}]
Consider a Littlewood-Paley decomposition of the unity
\begin{equation}\label{lp1}
{\rm Id}= P_{0}+\sum_{N}P_{N},
\end{equation}
where the summation is taken over the dyadic values of $N$, i.e. $N=2^j$, $j=0,1,2,\dots$ and  $P_0$, $P_{N}$ are 
Littlewood-Paley projectors. More precisely they are defined as Fourier multipliers by $\Delta_0=\psi_0(D_x)$ and for $N\geq 1$,
$P_N=\psi(D_x/N)$, where $\psi_0\in C_0^\infty(\R^3)$ and $\psi\in C_0^\infty(\R^3\backslash\{0\})$ are suitable functions such that
\eqref{lp1} holds. The maps $\psi(D_x/N)$ are defined similarly to \eqref{hagel} by
$$
\psi(D_x/N)(f)=\int_{\R^3}\psi(\xi/N) \widehat{f}(\xi)\, e^{i\xi\cdot x}d\xi.
$$
Set 
$$
u(t,x)\equiv e^{\pm it\sqrt{-\Delta_{\R^3}}}(f)\,.
$$
Our goal is to evaluate $\|u\|_{L^p([0,1]L^q(\R^3))}$. Thanks to the  Littlewood-Paley square function theorem, we have that
\begin{equation}\label{square_function}
\|u\|_{L^q(\R^3)}\approx 
\Big\|
\big(|P_0 u|^2+\sum_{N}|P_N u|^2\big)^{\frac{1}{2}}
\Big\|_{L^q(\R^3)}\,.
\end{equation}
The proof of \eqref{square_function} can be obtained as a combination of the Mikhlin-H\"ormander multiplier theorem and the Khinchin inequality for Bernouli variables\footnote{Interestingly, variants of the Khinchin inequality will be essentially used in our probabilistic approach to the cubic defocusing wave equation with data of super-critical regularity.}.  
Using the Minkowski inequality, since $p\geq 2$ and $q\geq 2$, we can write
\begin{equation}\label{debut}
\|u\|_{L^p_{t}L^q_{x}}\lesssim \|P_0 u\|_{L^p_t L^q_{x}}+\|P_N u\|_{L^p_t L^q_x l^2_{N}}\leq 
 \|P_0 u\|_{L^p_t L^q_{x}}+\|P_N u\|_{l^2_N L^p_t L^q_x }
 \end{equation}
Therefore, it suffices to prove that 
for every  $\psi\in C_0^\infty(\R^3\backslash\{0\})$ 
there exists $C>0$ such that for every $N$ and every $f\in L^2(\R^3)$,
\begin{equation}\label{strichartz_localisee}
\| \psi(D_x/N) e^{\pm it\sqrt{-\Delta_{\R^3}}}(f)\|_{L^p([0,1];L^q(\R^3))}\leq CN^{\frac{2}{p}}\|f\|_{L^2(\R^3)}\, .
\end{equation}
Indeed, suppose that \eqref{strichartz_localisee} holds true. Then, we define $\tilde{P}_{N}$ as 
$\tilde{P}_N=\tilde{\psi}(D_x/N)$, where $\tilde{\psi}\in C_0^\infty(\R^3\backslash\{0\})$ is such that $\tilde{\psi}\equiv 1$ on the support of $\psi$. Then $P_N=\tilde{P}_{N} P_{N}$. 
Now, coming back to \eqref{debut}, using the Sobolev inequality to evaluate $ \|P_0 u\|_{L^p_t L^q_{x}}$ and \eqref{strichartz_localisee} to evaluate $\|P_N u\|_{l^2_N L^p_t L^q_x }$, we arrive at the bound
$$
\|u\|_{L^p_{t}L^q_{x}}\lesssim \|f\|_{L^2}+
\|N^{\frac{2}{p}}\|P_{N} f\|_{ L^2_x}\|_{l^2_N}\lesssim \|f\|_{H^{\frac{2}{p}}}\,.
$$
Therefore, it remains to prove  \eqref{strichartz_localisee}.
Set 
$$
T\equiv \psi(D_x/N) e^{\pm it\sqrt{-\Delta_{\R^3}}}\,.
$$
Our goal is to study the mapping properties of $T$ from $L^2_{x}$ to $L^p_tL^q_{x}$.
We can write
\begin{equation}\label{dual}
\|Tf\|_{L^p_tL^q_{x}}=\sup_{\|G\|_{L^{p'}_tL^{q'}_{x}}\leq 1}\big |
\int_{t,x}Tf\overline{G}
\big|,
\end{equation}
where $\frac{1}{p}+\frac{1}{p'}=\frac{1}{q}+\frac{1}{q'}=1$.
Note that in order to write \eqref{dual} the values $1$ and $\infty$ of $p$ and $q$ are allowed.
Next, we can write
\begin{equation}\label{star}
\int_{t,x}Tf\overline{G}=\int_{x}f\overline{T^\star G},
\end{equation}
where $T^{\star}$ is defined by
$$
T^{\star}G\equiv \int_{0}^1
 \psi(D_x/N) e^{\mp i\tau\sqrt{-\Delta_{\R^3}}}G(\tau)d\tau\,.
$$
Indeed, we have
\begin{eqnarray*}
\int_{t,x}Tf\overline{G} & = &\int_{0}^{1}\int_{\R^3} 
 \psi(D_x/N) e^{\pm it\sqrt{-\Delta_{\R^3}}} f \,
\overline{G(t)}dxdt
\\
& = & \int_{0}^{1}\int_{\R^3} f\, 
\overline{ \psi(D_x/N) e^{\mp it\sqrt{-\Delta_{\R^3}}}G(t)}
dxdt
\\
& = &
\int_{\R^3} f\, \int_{0}^1 \overline{ \psi(D_x/N) e^{\mp it\sqrt{-\Delta_{\R^3}}}G(t)}dt\, dx\,.
\end{eqnarray*}
Therefore \eqref{star} follows.
But thanks to the Cauchy-Schwarz inequality we can write 
$$
|\int_{x}f\overline{T^\star G}|\leq \|f\|_{L^2_{x}}\|T^\star G\|_{L^2_{x}}\,.
$$
Therefore, in order to prove  \eqref{strichartz_localisee}, it suffices to prove the bound 
$$
\|T^\star G\|_{L^2_{x}}\lesssim N^{\frac{2}{p}} \|G\|_{L^{p'}_t L^{q'}_{x}}\,.
$$
Next, we can write
\begin{eqnarray*}
\|T^\star G\|^2_{L^2_{x}} & = & \int_{x} T^\star G\,\overline{T^\star G}
\\
& = &
\int_{t,x}T(T^{\star}(G))\overline{G}
\\
& \leq &
\|T(T^{\star}(G))\|_{L^p_t L^{q}_{x}}   \|G\|_{L^{p'}_t L^{q'}_{x}}\,.
\end{eqnarray*}
Therefore, estimate   \eqref{strichartz_localisee} would follow from the estimate
\begin{equation}\label{stri-1-bis}
\|T(T^{\star}(G))\|_{L^p_t L^{q}_{x}}  \lesssim  N^{\frac{4}{p}}\|G\|_{L^{p'}_t L^{q'}_{x}}\,.
\end{equation}
An advantage of \eqref{stri-1-bis} with respect to \eqref{strichartz_localisee}
is that we have the same number of variables in both sides of the estimates.
Coming back to the definition of $T$ and $T^{\star}$, we can write
$$
T(T^{\star}(G))=
 \int_{0}^1
 \psi^2(D_x/N) e^{\pm i(t-\tau)\sqrt{-\Delta_{\R^3}}}G(\tau)d\tau\,.
$$
Now by using the triangle inequality,  for a fixed $t\in[0,1]$, we can write 
\begin{equation}\label{TT}
\|T(T^{\star}(G))\|_{L^q_{x}}\leq \int_{0}^1\
\big\| 
 \psi^2(D_x/N) e^{\pm i(t-\tau)\sqrt{-\Delta_{\R^3}}}G(\tau)
\big\|_{L^q_{x}}d\tau.
\end{equation}
On the other hand, using the Fourier transform, we can write 
$$
\psi^2(D_x/N) e^{\pm it\sqrt{-\Delta_{\R^3}}}(f)
=\int_{\R^3}
\psi^2(\xi/N)e^{\pm it |\xi|}e^{ix\cdot\xi}\hat{f}(\xi)d\xi\,.
$$
Therefore,
$$
\psi^2(D_x/N) e^{\pm it\sqrt{-\Delta_{\R^3}}}(f)
=\int_{\R^3} K(t,x-x')f(x')dx,
$$
where 
$$
K(t,x-x')=\int_{\R^3}
\psi^2(\xi/N)e^{\pm it |\xi|}e^{i(x-x')\cdot\xi}d\xi\,.
$$
A simple change of variable leads to
$$
K(t,x-x')= N^3\int_{\R^3}
\psi^2(\xi)e^{\pm it N|\xi|}e^{iN(x-x')\cdot\xi}d\xi\,.
$$
In order to estimate $K(t,x-x')$, we invoke the following proposition. 
\begin{prop}[soft stationary phase estimate estimate]\label{st_faza_gen}
Let $d\geq 1$.
For every $\Lambda>0$, $N\geq 1$ there exists $C>0$ such that for every 
$\lambda\geq 1$, every $a\in C^{\infty}_0(\R^d)$, 
satisfying
$$
\sup_{|\alpha|\leq 2N}\sup_{x\in\R^d}|\partial^{\alpha}a(x)|\leq \Lambda,
$$
every $\varphi\in C^{\infty}({\rm supp}(a))$ satisfying
\beq
\label{up2}
\sup_{2\leq |\alpha|\leq 2N+2}\,
\sup_{x\in
{{\rm supp}(a)}
}|\partial^{\alpha}\varphi(x)|\leq \Lambda
\eeq
one has the bound
\beq
\label{estphase}
\Big|\int_{\R^d}e^{i\lambda\varphi(x)}a(x)dx\Big|\leq 
C\int_{{\rm supp}(a)}\frac{dx}{(1+\lambda|\nabla\varphi(x)|^{2})^{N}}\,.
\eeq
\end{prop}
\begin{rema}
{\rm Observe that in \eqref{up2}, we do not require upper bounds for the first derivatives of $\varphi$. }
\end{rema}
We will give the proof of Proposition~\ref{st_faza_gen} later. 
Let us first show how to use it in order to complete the proof of \eqref{strichartz_localisee}.
We claim that
\begin{equation}\label{K(t)}
|K(t,x-x')|\lesssim N^3 (tN)^{-1}=N^2t^{-1}\,.
\end{equation}
Estimate \eqref{K(t)} trivially follows from the expression defining $K(t,x-x')$ for $|tN|\leq 1$ (one simply ignores the oscillation term).
For $|Nt|\geq 1$, using Proposition~\ref{st_faza_gen} (with $a=\psi^2$, $N=2$ and $d=3$), we get the bound 
$$
|K(t,x-x')|\lesssim N^3
\int_{{\rm supp}(\psi)}\frac{d\xi}{(1+|tN||\nabla\varphi(\xi)|^{2})^{2}}\,,
$$
where
$$
\varphi(\xi)=\pm |\xi|+\frac{N(x-x')\cdot\xi}{t}\,.
$$
Observe that $\varphi$ is $C^{\infty}$ on the support of $\psi$ and moreover it satisfies the assumptions of  Proposition~\ref{st_faza_gen}.
We next observe that
\begin{equation}\label{vvv}
\int_{{\rm supp}(\psi)}\frac{d\xi}{(1+|tN||\nabla\varphi(\xi)|^{2})^{2}}
\lesssim  (tN)^{-1}\,.
\end{equation}
Indeed,  since $\nabla\varphi(\xi)=\pm\frac{\xi}{|\xi|}+t^{-1}N(x-x')$ we obtain that
one can split the support of integration in regions such that there are two different $j_1,j_2\in \{1,2,3\}$ such that 
one can perform the change of variable 
$$
\eta_{j_1}=\partial_{\xi_{j_1}}\varphi(\xi),\quad 
\eta_{j_2}=\partial_{\xi_{j_2}}\varphi(\xi),
$$
with a non-degenerate Hessian. More precisely, we have 
$$
\det\left(
\begin{array}{ll}
\partial_{\xi_1}^2\varphi(\xi)  & \partial_{\xi_1,\xi_2}^2\varphi(\xi)
\\
 \partial_{\xi_1,\xi_2}^2\varphi(\xi)  & \partial_{\xi_2}^2\varphi(\xi)
\end{array}
\right)=\frac{\xi_3^2}{|\xi|^4}
$$
which is not degenerate for $\xi_3\neq 0$. Therefore for $\xi_3\neq 0$, we can choose $j_1=1$ and $j_2=2$.
Similarly,  $\xi_1\neq 0$, we can choose $j_1=2$ and $j_2=3$ and for  $\xi_2\neq 0$, we can choose $j_1=1$ and $j_2=3$. 
Therefore, 
using that the support of $\psi$ does not meet zero, after splitting the support of the integration in three regions, 
by choosing the two "good" variables and by neglecting the integration with respect to the remaining variable, we obtain that 
$$
\int_{{\rm supp}(\psi)}\frac{d\xi}{(1+|tN||\nabla\varphi(\xi)|^{2})^{2}}
\lesssim  
\int_{\R^2}
\frac{d \eta_{j_1}  d \eta_{j_2}}{(1+|tN|\, (|\eta_{j_1}|^{2}
+|\eta_{j_2}|^2)^{2}}
\lesssim
(tN)^{-1}\,.
$$
Thus, we have \eqref{vvv} which in turn implies \eqref{K(t)}.

Thanks to \eqref{K(t)}, we arrive at the estimate 
$$
\big\|  \psi^2(D_x/N) e^{\pm i(t-\tau)\sqrt{-\Delta_{\R^3}}}G(\tau)\big\|_{L^\infty_{x}}\lesssim N^2|t-\tau|^{-1}\|G(\tau)\|_{L^1_{x}}\,.
$$
On the other hand, we also have the trivial bound 
$$
\big\|  \psi^2(D_x/N) e^{\pm i(t-\tau)\sqrt{-\Delta_{\R^3}}}G(\tau)\big\|_{L^2_{x}}\lesssim \|G(\tau)\|_{L^2_{x}}\,.
$$
Therefore using the basic Riesz-Torin interpolation theorem, we arrive at the bound
$$
\big\|  \psi^2(D_x/N) e^{\pm i(t-\tau)\sqrt{-\Delta_{\R^3}}}G(\tau)\big\|_{L^q_{x}}\lesssim 
\frac{N^{\frac{4}{p}}}{|t-\tau|^{\frac{2}{p}}}
\|G(\tau)\|_{L^{q'}_{x}}\,.
$$
Therefore coming back to \eqref{TT}, we get
$$
\|T(T^{\star}(G))\|_{L^q_{x}}\lesssim \int_{0}^1
\frac{N^{\frac{4}{p}}}{|t-\tau|^{\frac{2}{p}}}
\big\| G(\tau)\big\|_{L^{q'}_{x}}d\tau\,.
$$
Therefore, the estimate \eqref{stri-1-bis} would follow from the one dimensional estimate 
\begin{equation}\label{HLS}
\big\|\int_{\R}
\frac{f(\tau)}{|t-\tau|^{\frac{2}{p}}}
d\tau\big\|_{L^p(\R)}\lesssim \|f\|_{L^{p'}(\R)}\,.
\end{equation}
Thanks to our assumption, one has $\frac{2}{p}<1$ and also
$$
1+\frac{1}{p}=\frac{1}{p'}+\frac{2}{p}\,.
$$
Therefore estimate \eqref{HLS} is precisely the Hardy-Littlewood-Sobolev inequality (cf. \cite{LL}). 
This completes the proof of  \eqref{strichartz_localisee}, once we provide the proof of Proposition~\ref{st_faza_gen}.
\begin{proof}[Proof of Proposition~\ref{st_faza_gen}]
We follow \cite{FRT}.  Consider the first order differential operator defined by 
$$
L\equiv 
\frac{1}{i(1+\lambda|\nabla \varphi|^2)}\sum_{j=1}^{d}\partial_{j}\varphi 
\partial_{j}
+
\frac{1}{1+\lambda|\nabla \varphi|^2}\,.
$$
which satisfies
$
L(e^{i\lambda\varphi})=e^{i\lambda\varphi}\,.
$
We have that 
$$
\int_{\R^d}e^{i\lambda\varphi(x)}a(x)dx=
\int_{\R^d}L(e^{i\lambda\varphi(x)})a(x)dx
=
\int_{\R^d}e^{i\lambda\varphi(x)}
\tilde{L}(a(x))dx,
$$
where $\tilde{L}$ is defined by
\begin{multline*}
\tilde{L}(u)=-\sum_{j=1}^{d}\frac{\partial_{j}\varphi}{i(1+\lambda|\nabla\varphi|^2)}\partial_{j}u+
\Big(
-\sum_{j=1}^{d}\frac{\partial_{j}^{2}\varphi}{i(1+\lambda|\nabla\varphi|^2)}+
\\
\sum_{j=1}^{d}\frac{2\lambda\partial_{j}\varphi\,(\nabla\varphi\cdot\nabla\partial_{j}\varphi)}{i(1+\lambda|\nabla\varphi|^2)^2}\Big)u +\frac{1}{1+\lambda|\nabla\varphi|^2}u\,.
\end{multline*}
As a consequence, we get the bound
\beq
\label{LN}\Big|\int_{\R^d}e^{i\lambda\varphi(x)}a(x)dx\Big|\leq 
\int_{\R^d} | \tilde{L}^N a|,
\eeq
where $N\in \mathbb{N}$.
To conclude, we need to estimate the coefficients of $\tilde{L}.$ We shall use the notation $\langle u\rangle= ( 1 + |u|^2)^{\frac12}$ and we set  $\lambda = \mu^2$. 
 At first, we consider
 $$ F(x)= Q( \mu^2 | \nabla \varphi(x) |^2), \quad Q(u)= { 1 \over 1 + u }, \, u\geq 0.$$
 We clearly have
 \begin{equation}\label{F1}
  F \lesssim   \langle \mu \nabla \varphi\rangle^{-2}
  \end{equation}
 and we shall estimate the derivatives of $F$.
 Set 
$$
 \Lambda^k(x) =  \sup_{ 2 \leq |\alpha | \leq k}
     |\partial^\alpha \varphi(x) |.
$$
 We have the following statement.
 \begin{lemm}\label{berlin1}
 For  $|\alpha|=k\geq 1$, we have the bound
   \begin{equation}
   \label{dF}
   | \partial^\alpha F(x) | \lesssim C( \Lambda^{k+1}(x)) \Big( 
 { 1 \over \langle\mu \nabla \varphi(x)\rangle^2} +
  { \mu^k \over \langle\mu \nabla \varphi(x)\rangle^{k+2}}\Big),
  \end{equation}    
where $C:\R^{+}\rightarrow \R^{+}$ is a suitable continuous increasing function (which can change from line to line and can always be taken of the form $C(t)=(1+t)^M$ for a sufficiently large $M$).  
 \end{lemm}
 \begin{proof}
  Using an induction on $k$,  we get that
  $\partial^\alpha F$  for $| \alpha |  = k\geq 1$ is a linear combination of terms under the form
  $$ T_{q}=  Q^{(m)}( \mu^2 |\nabla \varphi|^2)
   \Big(\partial^{\gamma_{1}} ( \mu^2 |\nabla \varphi|^2) \Big)^{q_{1}}
    \cdots \Big( \partial^{\gamma_{k}} (\mu^2 |\nabla \varphi |^2)\Big)^{q_{k}}$$
     where
   \begin{equation}
   \label{rel}
    q_{1}+ \cdots + q_{k}= m \quad \textrm{and} \quad
     \sum |\gamma_{i}| q_{i}= k,\quad q_i\geq 0.
     \end{equation}
   Since 
   $|Q^{(m)}(u)|  \lesssim   \langle u\rangle^{ -m- 1} ,$
   we get 
    $$ |T_{q}|\lesssim 
     {1 \over \langle\mu \nabla \varphi \rangle^{ 2}}
      \left( {\mu  \over \langle\mu \nabla \varphi \rangle}\right)^{2m}
   \Big|
   \Big(\partial^{\gamma_{1}} (  |\nabla \varphi|^2) \Big)^{q_{1}}
    \cdots \Big( \partial^{\gamma_{k}} ( |\nabla \varphi |^2)\Big)^{q_{k}}
    \Big|.$$
   Moreover, by the Leibnitz formula
   \begin{equation*}
\partial^{\gamma_{i}} (|\nabla \varphi|^2)\leq \left\{
\begin{array}{l l }
C(\Lambda^{2})|\nabla \varphi|, & \textrm{ if } |\gamma_i|=1,\\
C(\Lambda^{|\gamma_i|+1})(|\nabla \varphi|+1), &  \textrm{ if } |\gamma_i|>1.
\end{array}
\right.
\end{equation*}     
We therefore have the following bound for $T_q$
  \begin{eqnarray*} |T_{q}|
  & \lesssim  & C(\Lambda^{k+1})  {1 \over \langle\mu \nabla \varphi \rangle^{ 2}}
      \left( {\mu \over \langle\mu \nabla \varphi \rangle}\right)^{2m} \left( |\nabla \varphi|^{m} +  |\nabla \varphi|^{\sum_{|\gamma_i|=1}q_i}\right)\\
   &\lesssim& C(\Lambda^{k+1}){1 \over \langle\mu \nabla \varphi \rangle^{ 2}}\left[
      \left( {\mu \over \langle\mu \nabla \varphi \rangle}\right)^{m}  + \left({\mu  \over \langle\mu \nabla \varphi \rangle} \right)^{m+\sum_{|\gamma_i|>1}q_i}\right].
\end{eqnarray*}
Next, by using \eqref{rel}, we note that
$$ m+\sum_{|\gamma_i|>1}q_i =  \sum_{|\gamma_i|>1}2q_i + \sum_{|\gamma_i|=1}q_i \leq  \sum |\gamma_{i}| q_{i}= k.$$
Therefore,  we get
 \begin{equation*}
 |T_{q}| \lesssim C(\Lambda^{k+1}) \Big( { 1 \over \langle\mu \nabla \varphi\rangle^2} +
  { \mu^k \over \langle\mu \nabla \varphi\rangle^{k+2}}\Big).
  \end{equation*}
   This completes the proof of Lemma~\ref{berlin1}.
  \end{proof}
We are now in position to prove the following statement.
\begin{lemm}\label{berlin2}
For $N\in \mathbb{N}$, we can write $\tilde{L}^N$ under the form
\beq 
\label{formeLN}\tilde{L}^Nu=
 \sum_{ |\alpha | \leq N} a^{(N)}_{\alpha} \partial^\alpha u
 \eeq
 with the estimates
 \begin{equation}
 \label{aaN}
  |a_{\alpha}^{(N)}(x)| \lesssim C(\Lambda^{ N+ 2}(x) ) { 1 \over \langle \mu \nabla \varphi(x) \rangle^N }
  \end{equation}
 and more generally for $ |\beta|=k$,
 \beq
 \label{daaN}
|\partial^\beta a_{\alpha}^{(N)}(x)| \lesssim C(\Lambda^{ N+ k + 2}(x) )\Big( { 1 \over \langle \mu \nabla \varphi(x) \rangle^N }  + {\mu^k \over\langle \mu \nabla \varphi(x)\rangle^{N+k}} \Big).
\eeq
\end{lemm}
\begin{proof}
  We reason by induction on $N$. First, we notice that $\tilde{L}$ is under the form
 $$ \tilde{L}=  \sum_{j=1}^d a_{j } \partial_{j} + b,$$
 where 
 $$ a_{j}=  i \partial_{j} \varphi F, \quad
  b= F+ i \sum_{j=1}^d \partial_{j}\big( \partial_{j}\varphi F \big)= F+ \sum_{j=1}^d \partial_{j} a_{j}.$$
 Consequently, by using \eqref{F1},  we get that
 \begin{eqnarray}
 \label{aj} & &  |a_{j}| \lesssim  { 1 \over \mu } { 1 \over \langle \mu \nabla \varphi \rangle}
 \end{eqnarray}
 and by the Leibnitz formula, since $\partial^\alpha a_{j}$ for $|\alpha | \geq 1$
  is a linear combination of terms under the form
  $$ ( \partial^{\beta}\partial_{j} \varphi)  \partial^\gamma F, \quad |\beta|+ |\gamma| =
   |\alpha|, 
   $$ 
   we get by using \eqref{dF} that  for $|\alpha |=k\geq 1 $,
\begin{equation}\label{da}
  |\partial^\alpha a_{j} | 
\lesssim     C(\Lambda^{ k  + 1}) \Big(
     { 1 \over \langle \mu \nabla \varphi \rangle} + { \mu^{k-1} \over \langle \mu \nabla \varphi \rangle^{k+1}}
\Big).
 \end{equation}
 Consequently, we also find thanks to \eqref{da}, \eqref{dF} that for
  $|\alpha |=k\geq 0,$
 \begin{equation}
 \label{db}
 |\partial^\alpha b | \lesssim C(\Lambda^{k+2})
  \Big(  { 1 \over \langle \mu \nabla \varphi \rangle} + { \mu^{k} \over \langle \mu \nabla \varphi \rangle^{k+2}}
  \Big).
\end{equation}
Using  \eqref{da}  \eqref{db}, we obtain that the assertion of the lemma holds true for $N=1$. 
Next, let us assume that it is true at the order $N$.
  We have
 $$ (\tilde{L})^{N+1} u =
    \sum_{j=1}^d \sum_{ |\alpha |\leq N}
     \Big(a_{j} a^{(N)}_{\alpha} \partial_{j} \partial^\alpha u 
      + a_{j}  \partial_{j} a^{(N)}_{\alpha}
       \partial^\alpha u\Big) +  \sum_{ |\alpha | \leq N }  b a^{(N)}_{\alpha}  \partial^\alpha u.$$
 Consequently, we get that the coefficients are under the form 
 \begin{eqnarray*}
& &  a_{\alpha}^{(N+1)}=  a_{j} a_{\beta}^{(N)}
, \quad |\alpha|=N+1, \, |\beta|=N, \\
& & a_{\alpha}^{(N+1)} = a_{j} \partial_{j} a_{\beta}^{(N)} + a_{j} a_{\gamma}^{(N)}
 + b a_{\delta}^{(N)}, \quad
   |\beta|= |\delta |=  |\alpha|, \,  |\gamma|= |\alpha |-1.
 \end{eqnarray*}
 Therefore, by using \eqref{aj} and \eqref{daaN}, we get
  that \eqref{aaN} is true for $N+1$. 
  In order to prove  \eqref{daaN} for $N+1$, we need to evaluate  $\partial^\gamma  a_{\alpha}^{(N+1)}$. 
  The estimate of the contribution of all terms except  $\partial^\gamma(a_{j} \partial_{j} a_{\beta}^{(N)})$  
  follows directly from the induction hypothesis.  
  In order to estimate $\partial^\gamma( a_{j} \partial_{j} a_{\beta}^{(N)})$, we need to invoke \eqref{aj} and \eqref{da} and the induction hypothesis. This completes the proof of Lemma~\ref{berlin2}.
     \end{proof}
  Finally, thanks to \eqref{LN} and  Lemma~\ref{berlin2}, we get
  $$
  \Big| \int_{\mathbb{R}^d}  e^{i \lambda \varphi(x)} a(x)\, dx \Big|
   \lesssim K    \int_{{\rm supp(a)}} { dx \over  {( 1 + \lambda |\nabla \varphi |^2) }^{N\over 2} }\, dx,
   $$
   where 
   $$
   K\equiv 
   (\sup_{x\in
   {\rm supp(a)}
   } \Lambda^{N+2}(x))
   \big(\sup_{x\in\R^d} \sup_{ | \alpha | \leq N} | \partial^\alpha a(x)|\big).
   $$
    This completes the proof of Proposition~\ref{st_faza_gen}.
  \end{proof}    
This completes the proof of Proposition~\ref{simm}.
\end{proof}
This completes the proof of Proposition~\ref{Strichartz_R3}.
\end{proof}
\begin{rema}
{\rm
If in the proof of the Strichartz estimates, we use the triangle inequality instead of the square function theorem and the Young inequality instead of the Hardy-Littlewood-Sobolev inequality, we would obtain slightly less precise estimates.
These estimates are sufficient to get all sub-critical well-posedness results. However in the case of initial data with critical Sobolev regularity the finer arguments using the square function and the  
 Hardy-Littlewood-Sobolev inequality are essentially needed. 
}
\end{rema}
\section{Local well-posedness in $H^s\times H^{s-1}$, $s\geq 1/2$}
In this section, we shall use the Strichartz estimates in order to improve the well-posedness result of
Proposition~\ref{prop.local}. 
We shall be able to consider initial data in the more singular Sobolev spaces  $H^s\times H^{s-1}$, $s\geq 1/2$.
We start by a definition.
\begin{defi}\label{admi}
For $0\leq s < 1$, a couple of real numbers $(p,q), \frac 2 s\leq p\leq + \infty$ is $s$-admissible if
$$\frac{1}{p}+\frac{3}{q}=\frac{3}{2}-s.
$$
For $T>0$, $0\leq s < 1$, we define the spaces 
\begin{equation}\label{eq.espstri}
 X^s_T= C ([0, T]; H^s( \T^3)) \bigcap_{(p,q) \text{ $s$- admissible}} L^p((0, T) ; L^q(\T^3))
\end{equation}
and its "dual space"
\begin{equation}\label{eq.spdual}
 Y^s_T= 
 \bigcup_ {(p,q) \text{ $s$- admissible}} L^{p'}((0, T) ; L^{q'}(\T^3))
\end{equation}
$(p',q')$ being the conjugate couple of $(p,q)$, equipped with their natural norms (notice that to define these spaces, we  can keep only the extremal couples corresponding to $p= 2/s$ and $p= + \infty$ respectively).
\end{defi} 
We can now state the non homogeneous Strichartz estimates for the three dimensional wave equation on the torus $\T^3$.
\begin{theo}\label{cor.stri}
For every $0< s <1 $, every $s$-admissible couple $(p,q)$, there exists $C>0$ such that
for every $T\in]0,1]$, every $F\in Y^{1-s}_T$, every $(u_0,u_1)\in H^s(\T^3)\times H^{s-1}(\T^3)$ one has
\begin{equation}\label{Xs}
\|S(t)(u_0,u_1)\|_{X^s_T}\leq C(\|u_0\|_{H^s(\T^3)}+\|u_1\|_{H^{s-1}(\T^3)})
\end{equation}
and
\begin{equation}\label{eq.inhomog}
\Big\|\int_0^t \frac {\sin((t-\tau)
\sqrt{ -  \Delta})}{\sqrt{- \Delta}} (F(\tau)) d\tau\Big\|_{X^s_T}\leq C \|F\|_{Y^{1-s}_T}
\end{equation} 
\end{theo}
\begin{proof}
Thanks to the H\"older inequality,
in order to prove \eqref{Xs}, it suffices the consider the two end point cases for $p$, i.e. $p=2/s$ and $p=\infty$ (the estimate in $C ([0, T]; H^s( \T^3))$ is straightforward). 
The case $p=2/s$ follows from Theorem~\ref{Strichartz}. The case $p=\infty$ results from the Sobolev embedding. This ends the proof of \eqref{Xs}.

Let us next turn to \eqref{eq.inhomog}.  We first observe that
\begin{equation}\label{5h}
\Big\|\int_0^t \frac {\sin((t-\tau)
\sqrt{ -  \Delta})}{\sqrt{- \Delta}} (F(\tau)) d\tau\Big\|_{
C ([0, T]; H^s( \T^3))
}\leq C \|F\|_{Y^{1-s}_T}
\end{equation} 
follows by duality from \eqref{Xs}.
Thanks to \eqref{5h}, we obtain that it suffices to show
\begin{equation}\label{sans_zero}
\Big\|\int_0^t \frac {\sin((t-\tau)
\sqrt{ -  \Delta})}{\sqrt{- \Delta}} (F(\tau)) d\tau\Big\|_{
 L^{p_1}_T L^{q_1}}\leq C 
 \|F\|_{ L^{p_2'}_T L^{q_2'}},
\end{equation} 
where $(p_1,q_1)$ is $s$-admissible and $(p_2',q_2')$ are such that $(p_2,q_2)$ are $(1-s)$- admissible and where for shortness we set
$$
L^p_T L^q \equiv  L^{p}((0, T) ; L^{q}(\T^3)).
$$
Denote by $\Pi_0$ the projector on the zero Fourier mode on $\T^3$, i.e. 
$$
\Pi_0(f)=(2\pi)^{-3}\int_{\T^3}f(x)dx\,.
$$
We have the bound 
$$
\Big\|\int_0^t \frac {\sin((t-\tau)
\sqrt{ -  \Delta})}{\sqrt{- \Delta}} (\Pi_0 F(\tau)) d\tau\Big\|_{L^p_T L^q}
\leq C\|F\|_{L^1((0, T) ; L^1(\T^3))}\,.
$$
By the H\"older inequality 
\begin{equation*}
\|F\|_{L^1((0, T) ; L^1(\T^3))}\leq C \|F\|_{ L^{p_2'}_T L^{q_2'}}
\end{equation*}
and therefore, it suffices to show the bound 
\begin{equation}\label{4h}
\Big\|\int_0^t \frac {\sin((t-\tau)
\sqrt{ -  \Delta})}{\sqrt{- \Delta}} (\Pi_0^{\perp}F(\tau)) d\tau\Big\|_{L^{p_1}_TL^{q_1}}\leq C \|F\|_{ L^{p_2'}_T L^{q_2'}  },
\end{equation} 
where 
$$
\Pi_0^{\perp}\equiv 1-\Pi_0\,.
$$
By writing the $\sin$ function as a sum of exponentials, we obtain that \eqref{4h} follows from
\begin{equation}\label{sans_zero_bis}
\Big\|\int_0^t 
e^{\pm i(t-\tau)\sqrt{-\Delta}}
((- \Delta)^{-\frac{1}{2}}\Pi_0^{\perp}F(\tau)) d\tau\Big\|_{
 L^{p_1}_T L^{q_1}
}\leq C \|F\|_{ L^{p_2'}_T L^{q_2'}
}\,.
\end{equation} 
Observe that $(- \Delta)^{-\frac{1}{2}}\Pi_0^{\perp}$ is well defined as a bounded operator from $H^s(\T^3)$ to $H^{s+1}(\T^3)$.
Set 
$$
K\equiv 
e^{\pm it \sqrt{-\Delta}}\Pi_{0}^\perp\,.
$$
Thanks to \eqref{Xs},  by writing
$$
e^{\pm it \sqrt{-\Delta}}\Pi_{0}^\perp=
\cos(t\sqrt{-\Delta})\Pi_{0}^\perp\pm i\sin( t\sqrt{-\Delta})
 (- \Delta)^{-\frac{1}{2}}
\Pi_{0}^\perp     (- \Delta)^{\frac{1}{2}}\, ,
$$
we see that the map $K$ is bounded from $H^s(\T^3)$ to $X^s_T$. Consequently, the dual map  $K^*$, defined by 
$$
 K^*(F) = 
\int_{0}^T e^{\mp i\tau \sqrt{-\Delta}}\Pi_{0}^\perp(F(\tau))d\tau
$$
is bounded from $Y^s$ to $H^{-s}(\T^3)$.
Using the last property with $s$ replaced by $1-s$ (which remains in $]0,1[$ if $s\in]0,1[$), we obtain the following sequence of continuous mappings
\begin{equation}\label{reditza}
 L^{p_2'}_T L^{q_2'}
\stackrel{K^{\star}}{\longrightarrow}H^{s-1}(\T^3)
\stackrel{ (- \Delta)^{-\frac{1}{2}}\Pi_0^{\perp}}{\longrightarrow}H^{s}(\T^3)
\stackrel{K}{\longrightarrow}
 L^{p_1}_T L^{q_1}\,.
\end{equation}
On the other hand, we have
$$ 
\big(K
\circ
 ((- \Delta)^{-\frac{1}{2}}\Pi_0^{\perp})
\circ
K^*\big)(F) =
\int_0^T
e^{\pm i(t-\tau)\sqrt{-\Delta}}
((- \Delta)^{-\frac{1}{2}}\Pi_0^{\perp}F(\tau)) d\tau\
$$
Therefore, we obtain the bound 
\begin{equation}\label{sans_zero_tris}
\Big\|\int_0^T
e^{\pm i(t-\tau)\sqrt{-\Delta}}
((- \Delta)^{-\frac{1}{2}}\Pi_0^{\perp}F(\tau)) d\tau\Big\|_{
 L^{p_1}_T L^{q_1}
 }\leq C \|F\|_{L^{p_2'}_T L^{q_2'}}\,.
\end{equation} 
The passage from \eqref{sans_zero_tris} to \eqref{sans_zero_bis} can be done by using the Christ-Kiselev \cite{Christ} argument, as we explain below.
By a density argument it suffices to prove \eqref{sans_zero_bis} for $F\in C^{\infty}( [0,T]\times\T^3 )$. 
We can of course also assume that 
$$
\|F\|_{L^{p_2'}_T L^{q_2'}}=1.
$$
For $n\geq 1$ an integer and $m=0,1,\cdots, 2^{n}$, we define $t_{n,m}$ as
$$
\int_{0}^{t_{n,m}}\|F(\tau)\|_{L^{q_2'}(\T^3)}^{p_2'}d\tau=m2^{-n}\,.
$$
Of course $0=t_{n,0}\leq t_{n,1}\leq\cdots\leq t_{n,2^n}=T$.
Next, we observe that for $0\leq \alpha<\beta\leq 1$ there is a unique $n$ such that $\alpha\in [2m2^{-n},(2m+1)2^{-n})$ and $\beta\in [(2m+1)2^{-n},(2m+2)2^{-n})$ for some  $m\in \{0,1,\cdots, 2^{n-1}-1\}$.
Indeed, this can be checked by writing the representations of $\alpha$ and $\beta$ in base $2$ (the number $n$ corresponds to the first different digit of $\alpha$ and $\beta$).
Therefore, if we denote by $\chi_{\tau<t}(\tau,t)$ the characteristic function of the set $\{(\tau,t):0\leq \tau<t\leq T\}$ then we can write
\begin{equation}\label{christ}
\chi_{\tau<t}(\tau,t)=\sum_{n=1}^\infty\sum_{m=0}^{2^{n-1}-1}
\chi_{n,2m}(\tau)\chi_{n,2m+1}(t),
\end{equation}
where $\chi_{n,m}$ ($m=0,1,\cdots, 2^n$) denotes the characteristic function of the interval $[t_{n,m},t_{n,(m+1)})$.
Indeed, in order to achieve \eqref{christ}, it suffices to apply the previous observation for every $:0\leq \tau<t\leq T$ with $\alpha$ and $\beta$ defined as 
$$
\alpha=\int_{0}^{\tau}\|F(s)\|_{L^{q_2'}(\T^3)}^{p_2'}ds,\quad
\beta=\int_{0}^t\|F(s)\|_{L^{q_2'}(\T^3)}^{p_2'}ds\,.
$$
Therefore,  thanks to \eqref{christ}, we can write
$$
\int_0^t
e^{\pm i(t-\tau)\sqrt{-\Delta}}
((- \Delta)^{-\frac{1}{2}}\Pi_0^{\perp}F(\tau)) d\tau
$$
as
$$
\sum_{n=1}^\infty\sum_{m=0}^{2^{n-1}-1}
\chi_{n,2m+1}(t)
\int_0^{T}
e^{\pm i(t-\tau)\sqrt{-\Delta}}
((- \Delta)^{-\frac{1}{2}}\Pi_0^{\perp}\chi_{n,2m}(\tau)F(\tau)) d\tau\,.
$$
 The goal is to evaluate the  $L^{p_1}_T L^{q_1}$ norm of the last expression. 
 Using that for a fixed $n$, $\chi_{n,2m+1}(t)$ have disjoint supports, we obtain that the  $L^{p_1}_T L^{q_1}$ norm of the last expression can be estimated by
 $$
\sum_{n=1}^\infty
\Big(
\sum_{m=0}^{2^{n-1}-1}
\big\|
\int_0^{T}
e^{\pm i(t-\tau)\sqrt{-\Delta}}
((- \Delta)^{-\frac{1}{2}}\Pi_0^{\perp}\chi_{n,2m}(\tau)F(\tau)) d\tau
\big\|^{p_1}_{L^{p_1}_T L^{q_1}}
\Big)^{\frac{1}{p_1}}\,.
$$
Now, using  \eqref{sans_zero_tris}, we obtain that the last expression is bounded by
\begin{equation}\label{chch}
C
 \sum_{n=1}^\infty
\Big(
\sum_{m=0}^{2^{n-1}-1}
\big\|
\chi_{n,2m}(\tau)F(\tau)
\big\|^{p_1}_{L^{p_2'}_T L^{q_2'}}
\Big)^{\frac{1}{p_1}}\,.
\end{equation}
 By definition
 $$
 \big\|
\chi_{n,2m}(\tau)F(\tau)
\big\|^{p_2'}_{L^{p_2'}_T L^{q_2'}}=2^{-n}
$$
and therefore \eqref{chch} equals to
$$
C
 \sum_{n=1}^\infty
\Big(
\sum_{m=0}^{2^{n-1}-1}
2^{-\frac{n p_1}{p_2'}}
\Big)^{\frac{1}{p_1}}
\leq
C \sum_{n=1}^\infty 2^{n(\frac{1}{p_1}-\frac{1}{p_2'})}
\,.
$$
The last series is convergent since by the definition of admissible pairs  it follows that $p_2'<2<p_1$. 
Therefore we proved that 
 \eqref{sans_zero_tris} indeed implies  \eqref{sans_zero_bis}.
This completes the proof of Theorem~\ref{cor.stri}.
\end{proof}
%
%
We can now use Theorem~\ref{cor.stri} in order to get the following improvement of Proposition~\ref{prop.local}.
\begin{theo}[low regularity local well-posedness]\label{prop.local.low}
Let $s>1/2$.
Consider the cubic defocusing wave equation
\begin{equation}\label{model_bis}
(\partial_t^2-\Delta)u+u^3=0\,,
\end{equation}
posed on $\T^3$.
There exist positive  constants  $\gamma$, $c$ and $C$ such that for every $\Lambda\geq 1$, every 
$$
(u_0,u_1)\in H^s(\T^3)\times H^{s-1}(\T^3)
$$
satisfying
\begin{equation}\label{berlin3}
\|u_0\|_{H^s}+\|u_1\|_{H^{s-1}}\leq \Lambda
\end{equation}
there exists a unique solution of \eqref{model_bis} on the time interval $[0, T]$, 
$
T\equiv c\Lambda^{-\gamma}
$ 
with initial data
$$
u(0,x)=u_0(x), \quad \partial_t u(0,x)=u_1(x)\,.
$$
Moreover the solution satisfies
$$
\|(u,\partial_t u)\|_{L^\infty([0,T],H^s(\T^3)\times H^{s-1}(\T^3))}\leq C\Lambda,
$$
$u$ is unique in the class $X^s_T$ described in Definition~\ref{admi} and the dependence with respect to the initial data and with respect to the time is continuous.
More precisely, if $u$ and $\tilde{u}$ are two solutions of \eqref{model_bis} with initial data satisfying \eqref{berlin3} then 
\begin{multline}\label{mouh}
\|(u-\tilde{u},\partial_t u-\partial_t \tilde{u})\|_{L^\infty([0,T],H^s(\T^3)\times H^{s-1}(\T^3))}\leq 
\\
C\big(
\|u(0)-\tilde{u}(0)\|_{H^s(\T^3)}+\|\partial_t u(0)-\partial_t \tilde{u}(0)\|_{H^{s-1}(\T^3)}
\big).
\end{multline}
Finally, if 
$$
(u_0,u_1)\in H^\sigma(\T^3)\times H^{\sigma-1}(\T^3)
$$
for some $\sigma\geq s$ then there exists $c_\sigma>0$ such that
$$
(u,\partial_t u)\in
C([0,c_{\sigma} \Lambda^{-\gamma}];H^\sigma(\T^3)\times H^{\sigma-1}(\T^3))\,.
$$
\end{theo}
\begin{proof}
We shall suppose that $s\in (1/2,1)$, the case $s\geq 1$ being already treated in Proposition~\ref{prop.local}.
As in the proof of Proposition~\ref{prop.local}, we solve the integral equation
\begin{equation*}
u(t)=S(t)(u_0,u_1)-\int_{0}^t \frac{\sin((t-\tau)\sqrt{-\Delta})}{\sqrt{-\Delta}}((u^3(\tau))d\tau
\end{equation*}
by a fixed point argument.
Recall that
$$
\Phi_{u_0,u_1}(u)=  S(t)(u_0,u_1)-\int_{0}^t \frac{\sin((t-\tau)\sqrt{-\Delta})}{\sqrt{-\Delta}}((u^3(\tau))d\tau.
$$
We shall estimate $\Phi_{u_0,u_1}(u)$ in the spaces $X_T^s$ introduced in Definition~\ref{admi}. 
Thanks to Theorem~\ref{cor.stri}
$$
\|S(t)(u_0,u_1)\|_{X^s_T}\leq  C(\|u_0\|_{H^s(\T^3)}+\|u_1\|_{H^{s-1}(\T^3)})\,.
$$
Another use of Theorem~\ref{cor.stri} gives
$$
\Big\|
\int_{0}^t \frac{\sin((t-\tau)\sqrt{-\Delta})}{\sqrt{-\Delta}}((u^3(\tau))d\tau
\Big\|_{X^s_T}
\leq C\|u^3\|_{L^\frac{2}{1+s}_T L^{\frac{2}{2-s}}}
=
C\|u\|^3_{L^\frac{6}{1+s}_T L^{\frac{6}{2-s}}}
\,.
$$
Observe that the couple $(\frac{2}{1+s}, \frac{2}{2-s})$ is the dual of  $(\frac{2}{1-s}, \frac{2}{s})$ which is the end point $(1-s)$- admissible couple. 
We also observe that if $(p,q)$ is an $s$- admissible couple then $\frac{1}{q}$ ranges in the interval
$
[\frac{1}{2}-\frac{s}{2},\frac{1}{2}-\frac{s}{3}].
$
The assumption  $s\in (1/2,1)$ implies 
$$
\frac{1}{2}-\frac{s}{2}<\frac{2-s}{6}<\frac{1}{2}-\frac{s}{3}\,.
$$
Therefore $q^{\star}\equiv \frac{6}{2-s}$ is such that there exists $p^{\star}$ such that
$(p^{\star},q^{\star})$ is an $s$- admissible couple.
By definition $p^{\star}$ is such that
$$
\frac{1}{p^{\star}}+\frac{3}{q^{\star}}=\frac{3}{2}-s\,.
$$
The last relation implies that
$$
\frac{1}{p^{\star}}=\frac{1}{2}-\frac{s}{2}\,.
$$
Now, using the H\"older inequality in time, we obtain
$$
\|u\|_{L^\frac{6}{1+s}_T L^{\frac{6}{2-s}}}\leq 
T^{\frac{2s-1}{3}}
\|u\|_{L^{p^{\star}}_T L^{q^{\star}}}
$$
which in turn implies 
$$
\Big\|
\int_{0}^t \frac{\sin((t-\tau)\sqrt{-\Delta})}{\sqrt{-\Delta}}((u^3(\tau))d\tau
\Big\|_{X^s_T}
\leq CT^{2s-1}\|u\|_{X^s_T}^3\,.
$$
Consequently
$$
\|\Phi_{u_0,u_1}(u)\|_{X^s_T}\leq 
 C(\|u_0\|_{H^s(\T^3)}+\|u_1\|_{H^{s-1}(\T^3)})
 + CT^{2s-1}\|u\|_{X^s_T}^3\,.
 $$
 A similar argument yields 
 \begin{equation}\label{mouh2}
\|\Phi_{u_0,u_1}(u)
-\Phi_{u_0,u_1}(v)
\|_{X^s_T}\leq 
  CT^{2s-1}
  \big(\|u\|_{X^s_T}^2+\|v\|^2_{X^s_T}\big)\|u-v\|_{X^s_T} \,.
 \end{equation}
Now, one obtains the existence and the uniqueness statements as in the proof of Proposition~\ref{prop.local}.
Estimate \eqref{mouh} follows from \eqref{mouh2} and a similar estimate obtained after differentiation of the Duhamel formula with respect to $t$.
The propagation of regularity statement can be obtained as in the proof of Proposition~\ref{prop.local}.
This completes the proof of Theorem~\ref{prop.local.low}.
\end{proof}
Concerning the uniqueness statement, we also have the following corollary which results from the proof of Theorem~\ref{prop.local.low}.
\begin{cor}\label{sun}
Let $s>1/2$. Let $(p^{\star},q^{\star})$ be the $s$- admissible couple defined by
$$
p^\star=\frac{2}{1-s},\quad q^{\star}\equiv \frac{6}{2-s}\,.
$$
Then the solutions constructed in Theorem~\ref{prop.local.low} is unique in the class 
$$
L^{p^\star}([0,T];L^{q^{\star}}(\T^3))\,.
$$
\end{cor}
\begin{rema}
{\rm 
As a consequence of  Theorem~\ref{prop.local.low}, we have that for each 
$
(u_0,u_1)\in H^s(\T^3)\times H^{s-1}(\T^3)
$
there is a solution with a maximum time existence $T^{\star}$ and if $T^{\star}<\infty$ than necessarily 
\begin{equation}\label{b-up}
\lim_{t\rightarrow T^{\star}}
\|(u(t),\partial_t u(t))\|_{H^s(\T^3)\times H^{s-1}(\T^3)}
=\infty.
\end{equation}
}
\end{rema}
One can also prove a suitable local well-posedness in the case $s=1/2$ but in this case the dependence of the existence time on the initial data is more involved. 
Here is a precise statement.
\begin{theo}\label{critical}
Consider the cubic defocusing wave equation
\begin{equation}\label{model_tris}
(\partial_t^2-\Delta)u+u^3=0\,,
\end{equation}
posed on $\T^3$.
For every 
$$
(u_0,u_1)\in H^{\frac{1}{2}}(\T^3)\times H^{-\frac{1}{2}}(\T^3)
$$
there exists a time $T>0$ and a unique solution of \eqref{model_tris} 
in
$$
L^4([0,T]\times\T^3)\times C([0,T];H^{\frac{1}{2}}(\T^3)),
$$
with initial data
$$
u(0,x)=u_0(x), \quad \partial_t u(0,x)=u_1(x)\,.
$$
\end{theo}
\begin{proof}
For $T>0$, using the Strichartz estimates of Theorem~\ref{cor.stri}, we get 
\begin{eqnarray*}
\|\Phi_{u_0,u_1}(u)\|_{L^{4}([0,T]\times\T^3)}
& \leq & 
 \|S(t)(u_0,u_1)\|_{L^{4}([0,T]\times\T^3)}+C\|u^3\|_{L^{4/3}([0,T]\times\T^3)}
\\
& = & 
 \|S(t)(u_0,u_1)\|_{L^{4}([0,T]\times\T^3)}
 + C\|u\|^3_{L^{4}([0,T]\times\T^3)}\,.
\end{eqnarray*}
Similarly, we get 
\begin{multline*}
\|\Phi_{u_0,u_1}(u)
-\Phi_{u_0,u_1}(v)
\|_{L^{4}([0,T]\times\T^3)}
\\
\leq C
  \big(\|u\|_{L^{4}([0,T]\times\T^3)}^2+\|v\|^2_{L^{4}([0,T]\times\T^3)}\big)\|u-v\|_{  L^{4}([0,T]\times\T^3) } \,.
 \end{multline*}
Therefore if $T$ is small enough then we can construct the solution by a fixed point argument  in $L^{4}([0,T]\times\T^3)$. 
In addition, the Strichartz estimates of Theorem~\ref{cor.stri} yield that the obtained solution belongs to $C([0,T];H^{\frac{1}{2}}(\T^3))$.
This completes the proof of Theorem~\ref{critical}.
\end{proof}
\begin{rema}
{\rm 
Observe that for data in $H^{\frac{1}{2}}(\T^3)\times H^{-\frac{1}{2}}(\T^3)$
we no longer have the small factor $T^{\kappa}$, $\kappa>0$ in the estimates for $\Phi_{u_0,u_1}$.
This makes that the dependence of the existence time $T$ on the data $(u_0,u_1)$ is much less explicit. 
In particular, we can no longer conclude that the existence time is the same for a fixed ball in $H^{\frac{1}{2}}(\T^3)\times H^{-\frac{1}{2}}(\T^3)$ 
and therefore we do not have the blow-up criterium \eqref{b-up} (with $s=1/2$). 
}
\end{rema}
\section{A constructive way of seeing the solutions}\label{picard}
In the proof of Theorem~\ref{prop.local.low},  we used the contraction mapping principle in order to construct the solutions. 
Therefore, one can define the solutions in a constructive way via the Picard iteration scheme. 
More precisely,  for $(u_0,u_1)\in H^s(\T^3)\times H^{s-1}(\T^3)$, we define the sequence $(u^{(n)})_{n\geq 0}$ as $u^{(0)}=0$ and for a given $u^{(n)}$, $n\geq 0$, we define $u^{(n+1)}$ as the solutions of the linear wave equation 
$$
(\partial_t^2-\Delta)u^{(n+1)}+(u^{(n)})^3=0,\quad u(0)=u_0,\,\, \partial_{t}u(0)=u_1. 
$$
Thanks to (the  proof of) Theorem~\ref{prop.local.low} the sequence $(u^{(n)})_{n\geq 0}$ is converging in $X^s_T$, and in particular in $C([0,T];H^s(\T^3))$ for 
$$
T\approx (\|u(0)\|_{H^s(\T^3)}+\|\partial_t u(0)\|_{H^{s-1}(\T^3)})^{-\gamma},\quad \gamma>0.
$$
One has that 
$$
u^{(1)}=S(t)(u_0,u_1)
$$
and for $n\geq 1$,
$$
u^{(n+1)}=u^{(1)}+{\mathcal T}(u^{(n)},u^{(n)},u^{(n)}),
$$
where the trilinear map ${\mathcal T}$ is defined as 
$$
{\mathcal T}(u,v,w)=
-\int_{0}^t \frac{\sin((t-\tau)\sqrt{-\Delta})}{\sqrt{-\Delta}}((u(\tau) v(\tau) w(\tau))d\tau.
$$
One then may compute 
$$
u^{(2)}=u^{(1)}+{\mathcal T}(u^{(1)},u^{(1)},u^{(1)}).
$$
The expression for $u^{(3)}$ is then
\begin{multline*}
u^{(3)}=u^{(1)}+{\mathcal T}(u^{(1)},u^{(1)},u^{(1)})
+3{\mathcal T}(u^{(1)},u^{(1)},{\mathcal T}(u^{(1)},u^{(1)},u^{(1)}))
\\
+3{\mathcal T}(u^{(1)},{\mathcal T}(u^{(1)},u^{(1)},u^{(1)}),{\mathcal T}(u^{(1)},u^{(1)},u^{(1)}))
\\
+
{\mathcal T}({\mathcal T}(u^{(1)},u^{(1)},u^{(1)}),{\mathcal T}(u^{(1)},u^{(1)},u^{(1)}),{\mathcal T}(u^{(1)},u^{(1)},u^{(1)})).
\end{multline*}
We now observe that for $n\geq 2$, the $n$th Picard iteration $u^{(n)}$ is a sum from $j=1$ to $j=3^{n-1}$ of $j$-linear expressions of $u^{(1)}$. Moreover the first $3^{n-2}$ terms of this sum contain the $(n-1)$th iteration. 
Therefore the solution can be seen as an infinite sum of multi-linear expressions of $u^{(1)}$. 
The Strichartz inequalities we proved can be used to show that for $s\geq 1/2$,
$$
\|{\mathcal T}(u,v,w)\|_{H^s(\T^3)}\lesssim \|u\|_{H^s(\T^3)}\|v\|_{H^s(\T^3)}\|w\|_{H^s(\T^3)}\,.
$$
The last estimate can be used to analyse the multi-linear expressions in the expansion and to show its convergence.
Observe that, we do not exploit any regularising effect in the terms of the expansion. The  ill-posedness result of the next section, will  basically show that such an effect is in fact not possible.
In our probabilistic approach in the next chapter, we will exploit that the trilinear term in the expression defining the solution is more regular in the scale of the Sobolev spaces than the linear one, 
almost surely with respect to a probability measure on $H^s$, $s<1/2$. 
In principle, it is not excluded that a regularising effect appears on the $5$-linear expression or further. 
However, it looks difficult to exploit further smoothing effects in the context of the cubic defocusing wave equation.
On the other hand, in the case of the 1d cubic Schr\"odinger equation such a phenomenon looks possible and does not seem  being fully understood (see \cite{COh}).
\section{Global well-posedness in $H^s\times H^{s-1}$,  for  some $s<1$}
One may naturally ask whether the solutions obtained in  Theorem~\ref{prop.local.low} can be extended globally in time. 
Observe that one can not use the argument of Theorem~\ref{prop.global} because there is no a priori bound available at the $H^s$, $s\neq 1$ regularity. 
One however has the following partial answer.
\begin{theo}[low regularity global well-posedness]\label{I-method}
Let $s>13/18$. Then the local solution obtained in  Theorem~\ref{prop.local.low}  can be extended globally in time. 
\end{theo}
For the proof of  Theorem~\ref{I-method}, we refer to \cite{GP, KPV, R}. Here, we only present the main idea (introduced in \cite{CKSTT}). 
Let $(u_0,u_1)\in H^s(\T^3)\times H^{s-1}(\T^3)$ for some $s\in(1/2,1)$. Let $T\gg 1$. For $N\geq 1$, we define a smooth Fourier multiplier acting as $1$ for frequencies $n\in\Z^3$ 
such that $|n|\leq N$ and acting as $N^{1-s}|n|^{s-1}$ for frequencies $|n|\geq 2N$. A concrete choice of $I_N$ is 
$
I_{N}(D)=I\big((-\Delta)^{1/2}/N\big),
$
where $I(x)$ is a smooth function which equals $1$ for $x\leq 1$ and which equals $x^{s-1}$ for $x\geq 2$. In other words $I(x)$ is one  for $x$ close to zero and decays like $x^{s-1}$ for $x\gg 1$.  
We choose $N=N(T)$ such that for the times of the local existence the modified energy (which is well-defined in $H^s\times H^{s-1}$)
$$
\int_{\T^3}\Big(
(\partial_t I_N u)^2+(\nabla I_N u)^2+\frac{1}{2}(I_N u)^4
\Big)
$$
does not vary much. This allows to extend the local solutions up to time $T\gg 1$. 
The analysis contains two steps, a local existence argument for $I_N u$ under the assumption that the modified energy remains below a fixed size and an energy increase estimate which is the substitute of the energy 
conservation used in the proof of  Theorem~\ref{prop.global}.
More precisely, we choose $N$ as $N=T^{\gamma}$ for some $\gamma=\gamma(s)\gg 1$. 
With this choice of $N$ the initial size of the modified energy is $T^{\gamma(1-s)}$.
The local well-posedness argument assures that $I_N u$ (and thus $u$ as well) exists on time of size $T^{-\beta}$ for some $\beta>0$ as far as the modified energy remains $\lesssim T^{\gamma(1-s)}$.
The main part of the analysis is to get an energy increase estimate showing that on the local existence time the modified energy does not increase more then $T^{-\alpha}$ for some $\alpha>0$. 
In order to arrive at time $T$ we need to iterate $\approx T^{1+\beta}$ times the local existence argument. In order to ensure that at each step of the iteration the modified energy remains $\lesssim T^{\gamma(1-s)}$, we need to impose the condition 
\begin{equation}\label{cond_I}
T^{1+\beta}\, T^{-\alpha}\lesssim T^{\gamma(1-s)}, \quad T\gg 1.
\end{equation}
As far as \eqref{cond_I} is satisfied, we can extend the local solutions globally in time. The condition  \eqref{cond_I} imposes the lower bound on $s$ involved in the statement of Theorem~\ref{I-method}.
One may conjecture that the global well-posedness in  Theorem~\ref{I-method} holds for any $s>1/2$. 
\section{Local ill-posedness in $H^s\times H^{s-1}$, $s\in (0,1/2)$}
It turns out that the restriction $s>1/2$ in Theorem~\ref{prop.local.low} is optimal. 
Recall that the classical notion of well-posedness in the Hadamard sense requires the existence, the uniqueness and the continuous dependence with respect to the initial data. 
A very classical example of contradicting the continuous dependence with respect to the initial data for a PDE is the initial value problem for the Laplace equation with data in Sobolev spaces.
Indeed, consider 
\begin{equation}\label{laplace}
(\partial_t^2+\partial_x^2)v=0,\quad v\,:\, \R_t\times \T_{x}\longrightarrow \R.
\end{equation}
The equation \eqref{laplace} has the explicit solution
$$
v_{n}(t,x)=e^{-\sqrt{n}}{\rm sh}(nt)\cos(nx).
$$
Then for every $(s_1,s_2)\in\R^2$, $v_n$ satisfies
$$
\|(v_n(0),\partial_t v_n(0))\|_{H^{s_1}(\T)\times H^{s_2}(\T)}\lesssim  e^{-\sqrt{n}}n^{\max(s_1,s_2+1)}\longrightarrow 0,
$$
as $n$ tends to $+\infty$ but for $t\neq 0$,
$$
\|(v_n(t),\partial_t v_n(t))\|_{
H^{s_1}(\T)\times H^{s_2}(\T)
}\gtrsim  e^{n|t|}\,e^{-\sqrt{n}}
n^{\min(s_1,s_2+1)}
\longrightarrow +\infty,
$$
as $n$ tends to $+\infty$.
Consequently \eqref{laplace} in not well-posed in
$
H^{s_1}(\T)\times H^{s_2}(\T)
$
or every $(s_1,s_2)\in\R^2$
because of the lack of continuous dependence with respect to the initial data $(0,0)$. 

It turns out that a similar phenomenon happens in the context of the cubic defocusing wave equation with low regularity initial data.
As we shall see below the mechanism giving the lack of continuous dependence is however quite different compared to \eqref{laplace}.
Here is the precise statement.
\begin{theo}\label{ill}
Let us fix $s\in (0,1/2)$ and $(u_0,u_1)\in C^{\infty}(\T^3)\times C^\infty(\T^3)$.
Then there exist $\delta >0$, a sequence  $(t_n)_{n=1}^{\infty}$ 
of positive numbers tending to zero and a sequence $(u_n(t,x))_{n=1}^{\infty}$ of $C(\R;C^{\infty}(\T^3))$ functions
such that
$$
(\partial_{t}^{2}-\Delta)u_n+u_n^3=0
$$
with
$$
\|(u_{n}(0)-u_0,\partial_t u_n(0)-u_1)\|_{H^s(\T^3)\times H^{s-1}(\T^3)}
\leq C [\log(n)]^{-\delta}\rightarrow_{n\rightarrow + \infty} 0
$$
but
$$
\|(u_{n}(t_n),\partial_t u_n(t_n))\|_{H^s(\T^3)\times H^{s-1}(\T^3)}\geq C[ \log(n)]^{\delta}\rightarrow_{n\rightarrow + \infty} +\infty.
$$
In particular, for every $T>0$,
$$
\lim_{n\rightarrow + \infty} \|(u_n(t),\partial_t u_n(t))\|_{L^\infty([0,T];H^s(\T^3)\times H^{s-1}(\T^3))}=+\infty.
$$
\end{theo}
\begin{proof}[Proof of Theorem~\ref{ill}]
We follow \cite{CCT,BT1/2, xia}.
Consider 
\begin{equation}\label{concon}
(\partial_{t}^{2}-\Delta)u+u^3=0
\end{equation}
subject to initial conditions
\begin{equation}\label{dataa}
(u_0(x)+\kappa_{n}n^{\frac{3}{2}-s}\varphi(nx), u_1(x)),\qquad n\gg 1\,,
\end{equation}
where $\varphi$ is a nontrivial bump function on $\R^3$ and
$$
\kappa_{n}\equiv [\log(n)]^{-\delta_1},
$$
with $\delta_1>0$ to be fixed later. Observe that for $n\gg 1$, we can see $\varphi(nx)$ as a $C^\infty$ function on $\T^3$.

Thanks to Theorem~\ref{prop.global}, we obtain that 
(\ref{concon}) with data given by (\ref{dataa}) has a unique global smooth solution which we denote by $u_{n}$.
Moreover $u_n\in C(\R;C^{\infty}(\T^3))$ thanks the propagation of the higher Sobolev regularity and the Sobolev embeddings. 

Next, we consider the ODE
\begin{equation}\label{V}
V''+V^3=0,\quad V(0)=1,\,\, V'(0)=0.
\end{equation}
\begin{lemm}\label{ode_V}
The Cauchy problem \eqref{V} has a global smooth (non constant) solution $V(t)$ which is periodic.
\end{lemm}
\begin{proof}
One defines locally in time the solution of \eqref{V} by an application of the Cauchy-Lipschitz theorem.
In order to extend the solutions globally in time, we multiply (\ref{V}) by $V'$. This gives that the solutions of \eqref{V} satisfy
$$
\frac{d}{dt}\big(
(V'(t))^2+\frac{1}{2}(V(t))^4
\big)=0
$$
and therefore taking into account the initial conditions, we get
\begin{equation}\label{relrel}
(V'(t))^2+\frac{1}{2}(V(t))^4=\frac{1}{2}\,.
\end{equation}
The relation \eqref{relrel} implies that $(V(t),V'(t))$ can not go to infinity in finite time. Therefore the local solution of \eqref{V} is defined globally in time. 
Let us finally show that $V(t)$ is periodic in time. 
We first observe that thanks to \eqref{relrel}, $|V(t)|\leq 1$ for all times $t$.
Therefore $t=0$ is a local maximum of $V(t)$. We claim that there is $t_0>0$ such that $V'(t_0)=0$. Indeed, otherwise $V(t)$ is decreasing on $[0,+\infty)$ which implies that $V'(t)\leq 0$ and from \eqref{relrel}, we deduce 
$$
V'(t)=-\sqrt{\frac{(1-(V(t)))^4}{2}}\,.
$$
Integrating the last relation between zero and a positive $t_0$ gives
$$
t_0=\sqrt{2}\int_{V(t_0)}^{1}\frac{dv}{\sqrt{1-v^4}} \, .
$$
Therefore 
$$
t_0\leq \sqrt{2}\int_{-1}^{1}\frac{dv}{\sqrt{1-v^4}}
$$
and we get a contradiction for $t_0\gg 1$. Hence, we indeed have that there is $t_0>0$ such that $V'(t_0)=0$. Coming back to \eqref{relrel} and using that $V(t_0)<1$, we deduce that $V(t_0)=-1$. 
Therefore $t=t_0$ is a local minimum of $V(t)$.
We now can show exactly as before that there exists $t_1>t_0$ such that $V'(t_1)=0$ and $V(t_1)>-1$. Once again using \eqref{relrel}, we infer that $V(t_1)=1$, i.e. $V(0)=V(t_1)$ and $V'(0)=V'(t_1)$. By the uniqueness part of the Cauchy-Lipschitz theorem, we obtain that $V$ is periodic with period $t_1$.
This completes the proof of Lemma~\ref{ode_V}.
\end{proof}
We next denote by $v_n$ the solution of
\begin{equation}\label{ode}
\partial_{t}^2v_{n}+v_{n}^{3}=0,\quad (v_{n}(0),\partial_{t}v_{n}(0))=
(\kappa_{n}n^{\frac{3}{2}-s}\varphi(nx), 0).
\end{equation}
It is now clear that
$$
v_{n}(t,x)=\kappa_{n}n^{\frac{3}{2}-s}\varphi(nx)V\big(t\kappa_{n}n^{\frac{3}{2}-s}\varphi(nx)\big).
$$  
In the next lemma, we collect the needed bounds on $v_n$.
\begin{lemm}\label{v_n}
Let 
$$
t_{n}\equiv [\log(n)]^{\delta_2}n^{-(\frac{3}{2}-s)}
$$
for some $\delta_2>\delta_1$. Then,  we have the following bounds for $t\in [0,t_n]$,
\begin{equation}\label{parvo_vn}
\|\Delta(v_n)(t,\cdot)\|_{H^1(\T^3)}\leq C[\log(n)]^{3\delta_2}n^{3-s},
\end{equation}
\begin{equation}\label{vtoro_vn}
\| \Delta(v_n)(t,\cdot)\|_{L^2(\T^3)}\leq C[\log(n)]^{2\delta_2}n^{2-s}\,,
\end{equation}
\begin{equation}\label{ii}
\|\nabla^{k}v_{n}(t,\cdot)\|_{L^{\infty}(\T^3)}\leq C[\log(n)]^{k\delta_2}n^{\frac{3}{2}-s+k}\, ,k=0,1,\cdots.
\end{equation}
Finally,  there exists $n_0\gg 1$ such that for $n\geq n_0$, 
\begin{equation}\label{convex}
\|v_{n}(t_n,\cdot)\|_{H^{s}(\T^3)}\geq C\kappa_{n}(t_n\kappa_{n}n^{\frac{3}{2}-s})^{s}
=C[\log(n)]^{-(s+1)\delta_1+s\delta_2}\,.
\end{equation}
\end{lemm}
\begin{proof}
Estimates \eqref{parvo_vn} and \eqref{vtoro_vn} follow from the general bound
\begin{equation}\label{general_bound}
\|v_n(t,\cdot)\|_{H^\sigma(\T^3)}\leq C\kappa_n 
(t_n\kappa_{n}n^{\frac{3}{2}-s})^{\sigma}n^{\sigma-s}\,,
\end{equation}
where $t\in [0,t_n]$ and $\sigma\geq 0$.
For integer values of $\sigma$, the bound \eqref{general_bound} is a direct consequence of the definition of $v_n$. For fractional values of $\sigma$ one needs to invoke an elementary interpolation inequality in the Sobolev spaces. 
Estimate \eqref{ii} follows directly from the definition of $v_n$. The proof of \eqref{convex} is slightly more delicate. 
We first observe that for $n\gg 1$, we have the lower bound
\begin{equation}\label{dolna}
\|v_n(t_n,\cdot)\|_{H^1(\T^3)}\geq c\kappa_n 
(t_n\kappa_{n}n^{\frac{3}{2}-s})n^{1-s}\,.
\end{equation}
Now, we can obtain \eqref{convex} by invoking \eqref{general_bound} (with $\sigma=2$),  the lower bound \eqref{dolna} and the interpolation inequality 
$$
\|v_{n}(t_n,\cdot)\|_{H^1(\T^3)}\leq \|v_{n}(t_n,\cdot)\|^{\theta}_{H^s(\T^3)}
 \|v_{n}(t_n,\cdot)\|^{1-\theta}_{H^2(\T^3)}
$$
for some $\theta>0$. It remains therefore to show \eqref{dolna}. 
After differentiating once the expression defining $v_n$, we see that \eqref{dolna} follows from the following statement. 
\begin{lemm}\label{lem.expl}
Consider a smooth not identically zero periodic function $V$ and  a non trivial bump function 
$\phi \in C^\infty_0( \mathbb{R}^d)$.
Then there exist $c>0$ and $\lambda_0\geq 1$ such that for every $\lambda >\lambda_0$ 
$$ 
\| \phi (x) V(\lambda \phi(x))\|_{L^2(\mathbb{R}^d)} \geq c \,.
$$
\end{lemm}
\begin{proof} 
We can suppose that the period of $V$ is $2\pi L$ for some $L>0$. 
Consider the Fourier expansion of $V$,
$$ 
V(t) = \sum_{n\in \mathbb{Z}}v_n e^{i\frac{n}{L} t},\quad |v_n| \leq C_N (1+|n|)^{-N}\,.
$$
We can assume that there is an open ball $B$ of $\R^d$ such that for some $c_0>0$, $|\partial_{x_1}\phi(x)|\geq c_0$ on $B$.
Let $0\leq \psi\leq 1$ be a non trivial $C^\infty_0(B)$  function. 
We can write 
$$
\| \phi (x) V(\lambda \phi(x))\|_{L^2(\mathbb{R}^d)}^2
\geq  
\| \psi(x)\phi (x) V(\lambda \phi(x))\|_{L^2(B)}^2
=
I_1+I_2,
$$
where 
$$
I_{1}= \sum_{n\in \mathbb{Z}}
|v_n|^2
\int_{B}(\psi(x)\phi(x))^2 dx,
$$
and  
$$
I_{2}= \sum_{n_1\neq n_2}v_{n_1}\overline{v_{n_2}}\int_{B}  e^{i\lambda\frac{n_1-n_2}{L}\phi(x)}\, (\psi(x)\phi(x))^2  \,dx.
$$
Clearly $I_1>0$ is independent of $\lambda$. 
On the other hand
$$
e^{i\lambda\frac{n_1-n_2}{L}\phi(x)}
=
\frac{L}{i\lambda (n_1-n_2)\partial_{x_1}\phi(x)}
\partial_{x_1}\big(
e^{i\lambda\frac{n_1-n_2}{L}\phi(x)}
\big).
$$
Therefore, after an integration by parts, we obtain that $|I_{2}|\lesssim \lambda^{-1}$.
This completes the proof of Lemma~\ref{lem.expl}.
\end{proof}
This completes the proof of Lemma~\ref{v_n}.
\end{proof}
We next consider the semi-classical energy
\begin{multline*}
E_{n}(u)\equiv n^{-(1-s)}\big(\|\partial_{t}u\|_{L^{2}(\T^3)}^{2}+
\|\nabla u\|_{L^{2}(\T^3)}^{2}\big)^{\frac{1}{2}}+
\\
n^{-(2-s)}\big(\|\partial_{t}u\|_{H^{1}(\T^3)}^{2}+\|\nabla
u\|_{H^{1}(\T^3)}^{2}\big)^{\frac{1}{2}}\,.
\end{multline*}
We are going to show that for very small times $u_n$ and $v_n+S(t)(u_0,u_1)$ are close with respect to $E_n$
but these small times are long enough to get the needed amplification of the $H^s$ norm.  We emphasise that this
amplification is a phenomenon only related to the solution of (\ref{ode}).
Here is the precise statement.
\begin{lemm}\label{compar}
There exist $\varepsilon>0$, $\delta_2>0$ and $C>0$ such that for $\delta_1<\delta_2$, if we set
$$
t_{n}\equiv [\log(n)]^{\delta_2}n^{-(\frac{3}{2}-s)}
$$
then for every $n\gg 1$, every $t\in [0,t_n]$,
$$
E_{n}\big(
u_{n}(t)-v_{n}(t)-S(t)(u_0,u_1)
\big)
\leq Cn^{-\varepsilon}\,.
$$
Moreover, 
\begin{equation}\label{s}
\|u_{n}(t)-v_n(t)-S(t)(u_0,u_1)\|_{H^{s}(\T^3)}\leq Cn^{-\varepsilon}\,.
\end{equation}
\end{lemm}
\begin{proof}
Set $u_{L}=S(t)(u_0,u_1)$ and $w_{n}=u_{n}-u_{L}-v_{n}$. 
Then $w_n$ solves the equation
\begin{equation}\label{eq_wn}
(\partial_{t}^{2}-\Delta)w_n=\Delta v_n-3v_{n}^{2}(u_L+w_{n})-3v_{n}(u_L+w_{n})^{2}-(u_L+w_{n})^{3},
\end{equation}
with initial data 
\begin{equation*}
(w_{n}(0,\cdot),\partial_{t}w_{n}(0,\cdot))=(0,0)\,.
\end{equation*}
Set
$$
F\equiv \Delta v_n-3v_{n}^{2}(u_L+w_{n})-3v_{n}(u_L+w_{n})^{2}-(u_L+w_{n})^{3}\,.
$$
Multiplying the equation \eqref{eq_wn} with $\partial_t w_n$ and integrating over $\T^3$ gives
\begin{equation*}
\Big|
\frac{d}{dt}
\big(\|\partial_{t}w_n(t)\|_{L^{2}(\T^3)}^{2}+
\|\nabla w_n(t)\|_{L^{2}(\T^3)}^{2}\big)
\Big|
\lesssim
\|\partial_t w_n(t)\|_{L^2(\T^3)}
\|F(t)\|_{L^2(\T^3)}
\end{equation*}
which in turn implies 
\begin{equation}\label{Energy1}
\Big|
\frac{d}{dt}
\big(\|\partial_{t}w_n(t)\|_{L^{2}(\T^3)}^{2}+
\|\nabla w_n(t)\|_{L^{2}(\T^3)}^{2}\big)^{\frac{1}{2}}
\Big|
\lesssim 
\|F(t)\|_{L^2(\T^3)}\,.
\end{equation}
Similarly, by first differentiating \eqref{eq_wn} with respect to the spatial variables, we get the bound
 \begin{equation}\label{Energy2}
\Big|\frac{d}{dt}\big(\|\partial_{t}w_n(t)\|_{H^{1}(\T^3)}^{2}+\|\nabla w_n(t)\|_{H^{1}(\T^3)}^{2}\big)^{\frac{1}{2}}\Big|
\lesssim 
\|F(t)\|_{H^1(\T^3)}\,.
\end{equation}
Now, using \eqref{Energy1} and \eqref{Energy2}, we obtain the estimate 
\begin{equation*}
\Big|\frac{d}{dt}\Big(E_{n}(w_n(t))\Big)\Big|
\leq C n^{-(2-s)}\|F(t)\|_{H^1(\T^3)}+Cn^{-(1-s)}\|F(t)\|_{L^2(\T^3)}\,.
\end{equation*}
Therefore using \eqref{parvo_vn}, \eqref{vtoro_vn}, we get 
\begin{multline}\label{energia}
\Big|\frac{d}{dt}\Big(E_{n}(w_n(t))\Big)\Big|
\leq 
C\Big([\log(n)]^{3\delta_2}n
\\
+n^{-(2-s)}\|G(t,\cdot)\|_{H^1(\T^3)}+n^{-(1-s)}\|G(t,\cdot)\|_{L^2(\T^3)}\Big)\,,
\end{multline}
where $G\equiv G_1+G_2$ with
$$
G_1= -3v_n^2 u_L- 3v_n u_L^2-u_L^3
$$
and
$$
G_2 = -3(u_L+v_n)^2 w_n - 3(u_L+v_n) w_n^2-w_n^3.
$$
Since $u_L\in C^{\infty}(\mathbb{R}\times \T^3)$ is independent of $n$, using (\ref{ii}) and (\ref{general_bound}) we can estimate $G_1$ as follows 
$$ 
n^{-(l-s)}\|G_1(t,\cdot))\|_{H^{l-1}(\T^3)}\lesssim [\log n]^{\delta_2}n^{\frac{1}{2}-s}
\lesssim 
[\log(n)]^{3\delta_2}n,\quad l=1,2.
$$
Writing for $t\in [0,t_n]$,
$$ w_n(t,x) = \int_0^t \partial_t w_n(\tau, x) d\tau,
$$
we obtain
\begin{equation}\label{i}
\|w_{n}(t,\cdot)\|_{H^k(\T^3)}\leq C[\log(n)]^{\delta_2}n^{-(\frac{3}{2}-s)}
\sup_{0\leq \tau\leq t}\|\partial_{t}w_{n}(\tau,\cdot)\|_{H^{k}(\T^3)}\,.
\end{equation}
Set
$$
e_{n}(w_n(t))\equiv \sup_{0\leq \tau\leq t}E_{n}(w_n(\tau))\,.
$$
Observe that $e_{n}(w_n(t))$ is increasing. 
Using (\ref{i}) (with $k=0,1$), (\ref{ii}) and the Leibniz rule, we get that for   $t\in [0,t_n]$ and for $l=1,2$,
\begin{equation*}
n^{-(l-s)}\|(u_{L}(t)+ v_{n}(t))^{2}w_{n}(t)\|_{H^{l-1}(\T^3)}\leq C[\log(n)]^{l\delta_2}n^{\frac{3}{2}-s}
e_{n}(w_n(t))\,.
\end{equation*}
Thanks to the Gagliardo-Nirenberg inequality, and \eqref{i} with $k=0$, we get for  $t\in [0,t_n]$,
\begin{eqnarray}
\label{iii}
\|w_{n}(t,\cdot)\|_{L^{\infty}(\T^3)}
& \leq & C\|w_{n}(t,\cdot)\|_{H^2(\T^3)}^{\frac{3}{4}}\|w_n(t,\cdot)\|_{L^2(\T^3)}^{\frac{1}{4}}
\\
\nonumber
& \leq & Cn^{\frac{3}{2}-s}e_{n}(w_{n}(t))\,.
\end{eqnarray}
Hence, we can use \eqref{iii} to treat the quadratic and cubic terms in $w_n$ and  to get the bound 
\begin{equation*}
n^{-(l-s)}\| G_2(t,\cdot)\|_{H^{l-1}(\T^3)}\leq C[\log(n)]^{l\delta_2}n^{\frac{3}{2}-s}
\big(e_{n}(w_n(t))+[e_{n}(w_n(t))]^3\big)\,.
\end{equation*}
Therefore, coming back to (\ref{energia}), we get for $t\in [0,t_n]$,
\begin{multline*}
\Big|\frac{d}{dt}\Big(E_{n}(w_n(t))\Big)\Big|\leq 
C[\log(n)]^{3\delta_2}n
\\
+C[\log(n)]^{2\delta_2}n^{\frac{3}{2}-s}\big(e_{n}(w_n(t))+[e_{n}(w_n(t))]^3\big)\,.
\end{multline*}
We now observe that
$$
\frac{d}{dt}
\big(e_{n}(w_n(t))\big)
\leq 
\Big|\frac{d}{dt}\Big(E_{n}(w_n(t))\Big)\Big|
$$
is resulting directly from the definition.
Therefore, we have the bound
\begin{multline}
\label{eee_mm}
\frac{d}{dt}\Big(
e_{n}(w_n(t))
\Big)
\leq 
C[\log(n)]^{3\delta_2}n
\\
+C[\log(n)]^{2\delta_2}n^{\frac{3}{2}-s}\big(e_{n}(w_n(t))+[e_{n}(w_n(t))]^3\big)\,.
\end{multline}
We first suppose that $e_{n}(w_n(t))\leq 1$. This property holds for  small values of $t$ since 
$$
E_{n}(w_{n}(0))\lesssim n^{-(1-s)}\,.
$$
In addition, the estimate for $e_{n}(w_n(t))$ we are looking for is much stronger than $e_{n}(w_n(t))\leq 1$.
Therefore, once we prove the desired estimate for $e_n(w_n(t))$ under the assumption $e_{n}(w_n(t))\leq 1$, 
we can use a bootstrap argument to get the estimate without the assumption $e_{n}(w_n(t))\leq 1$.

Estimate \eqref{eee_mm} yields that for $t\in [0,t_n]$,
$$
\frac{d}{dt}
(e_{n}(w_{n}(t)))
\leq
C[\log(n)]^{3\delta_2}n
+C[\log(n)]^{2\delta_2}n^{\frac{3}{2}-s}
e_{n}(w_{n}(t))
$$
and consequently 
$$
\frac{d}{dt}
\Big(
e^{
-Ct [\log(n)]^{2\delta_2}n^{\frac{3}{2}-s}
}
e_{n}(w_{n}(t))
\Big)
\leq
C[\log(n)]^{3\delta_2}\, n\,
e^{
-Ct [\log(n)]^{2\delta_2}n^{\frac{3}{2}-s}
}\,.
$$
An integration of the last estimate gives that for  $t\in [0,t_n]$,
\begin{eqnarray*}
e_{n}(w_{n}(t))
& \leq &
C \big(n^{-(1-s)}
+
[\log(n)]^{\delta_2}n^{s-\frac{1}{2}}
\big)
e^{Ct[\log(n)]^{2\delta_2}n^{\frac{3}{2}-s}}
\\
& \leq & 
C 
\big(
n^{-(1-s)}
+
[\log(n)]^{\delta_2}n^{s-\frac{1}{2}}
\big)
e^{C[\log(n)]^{3\delta_2}}\,.
\end{eqnarray*}
(one should see $\delta_2$ as $3\delta_2-2\delta_2$ and $s-1/2$ as $1-(3/2-s)$).
Since $s<1/2$, by taking $\delta_2>0$ small enough, we obtain that there exists $\varepsilon>0$ such that for $t\in [0,t_n]$,
$$
E_{n}(w_n(t))\leq Cn^{-\varepsilon}
$$
and in particular one has for $t\in [0,t_n]$,
\begin{equation}\label{parvo}
\|\partial_{t}w_{n}(t,\cdot)\|_{L^2(\T^3)}+\|\nabla w_{n}(t,\cdot)\|_{L^2(\T^3)}\leq
Cn^{1-s-\varepsilon}\,.
\end{equation}
We next estimate $\|w_{n}(t,\cdot)\|_{L^2}$. 
We may write for $t\in[0,t_n]$,
$$
\|w_{n}(t,\cdot)\|_{L^2(\T^3)}= \|\int_0^t \partial_tw_{n}(\tau,\cdot) d\tau  \|_{L^2(\T^3)}\leq ct_{n}\sup_{0\leq\tau\leq t}\|\partial_{t}w_{n}(\tau,\cdot)\|_{L^2(\T^3)}\,.
$$
Thanks to (\ref{parvo}) and the definition of $t_n$, we get
$$
\|w_{n}(t,\cdot)\|_{L^2(\T^3))}\leq
C[\log(n)]^{\delta_2}n^{-(\frac{3}{2}-s)}n^{1-s}n^{-\varepsilon}\,.
$$
Therefore, since $s<1/2$, 
\begin{equation}\label{vtoro}
\|w_{n}(t,\cdot)\|_{L^2(\T^3)}\leq Cn^{-s-\varepsilon}\,.
\end{equation}
An interpolation between (\ref{parvo}) and (\ref{vtoro}) yields (\ref{s}). 
This completes the proof of Lemma~\ref{compar}.
\end{proof}
Using Lemma~\ref{compar}, we may write
$$
\|u_{n}(t_n,\cdot)\|_{H^{s}(\T^3)}\geq \|v_{n}(t_n,\cdot)\|_{H^{s}(\T^3)} -C-Cn^{-\varepsilon}\,.
$$
Recall that \eqref{convex} yields
\begin{equation*}
\|v_{n}(t_n,\cdot)\|_{H^{s}(\T^3)}\geq C[\log(n)]^{-(s+1)\delta_1+s\delta_2}\,,
\end{equation*}
provided $n\gg 1$.
Therefore, by choosing $\delta_1$ small enough (depending on $\delta_2$ fixed in Lemma~\ref{compar}), 
we obtain that the exists $\delta>0$ such that
$$
\|v_n(t_n,\cdot)\|_{H^s(\T^3)}\geq C [\log(n)]^{\delta},\quad n\gg 1
$$
which in turn implies that
$$
\|u_n(t_n,\cdot)\|_{H^s(\T^3)}\geq C [\log(n)]^{\delta},\quad n\gg 1\,.
$$
This completes the proof of Theorem~\ref{ill}.
\end{proof}
Theorem~\ref{ill} implies that the Cauchy problem associated with the cubic focusing wave equation,
$$
(\partial_{t}^{2}-\Delta)u+u^3=0
$$
is ill-posed in $H^s(\T^3)\times H^{s-1}(\T^3)$ for $s<1/2$ because of the lack of continuous dependence for any $C^{\infty}(\T^3)\times C^{\infty}(\T^3)$ initial data.  

For future references, we also state the following consequence of  Theorem~\ref{ill}.
\begin{theo}\label{ill_bis}
Let us fix $s\in (0,1/2)$, $T>0$ and 
$$
(u_0,u_1)\in H^s(\T^3)\times H^{s-1}(\T^3)\,.
$$
Then there exists  a sequence $(u_n(t,x))_{n=1}^{\infty}$ of $C(\R;C^{\infty}(\T^3))$ functions
such that
$$
(\partial_{t}^{2}-\Delta)u_n+u_n^3=0
$$
with
$$
\lim_{n\rightarrow + \infty} \|(u_{n}(0)-u_0,\partial_t u_n(0)-u_1)\|_{H^s(\T^3)\times H^{s-1}(\T^3)}=0
$$
but 
$$
\lim_{n\rightarrow + \infty} \|(u_n(t),\partial_t u_n(t))\|_{L^\infty([0,T];H^s(\T^3)\times H^{s-1}(\T^3))}=+\infty.
$$
\end{theo}
\begin{proof}
Let $(u_{0,m}, u_{1,m})_{m=1}^{\infty}$ be a sequence of $C^{\infty}(\T^3)\times C^{\infty}(\T^3)$ functions such that
$$
\lim_{m\rightarrow + \infty} \|(u_0-u_{0,m},u_1-u_{1,m})\|_{H^s(\T^3)\times H^{s-1}(\T^3)}=0\,.
$$
For a fixed $m$, we apply Theorem~\ref{ill} in order to find a sequence 
 $(u_{m,n}(t,x))_{n=1}^{\infty}$ of $C(\R;C^{\infty}(\T^3))$ functions
such that
$$
(\partial_{t}^{2}-\Delta)u_{m,n}+u_{m,n}^3=0
$$
with
$$
\lim_{n\rightarrow + \infty} \|(u_{m,n}(0)-u_{0,m},\partial_t u_{m,n}(0)-u_{1,m})\|_{H^s(\T^3)\times H^{s-1}(\T^3)}=0
$$
and for every $m\geq 1$,
\begin{equation}\label{mn}
\lim_{n\rightarrow + \infty} \|(u_{m,n}(t),\partial_t u_{m,n}(t))\|_{L^\infty([0,T];H^s(\T^3)\times H^{s-1}(\T^3))}=+\infty.
\end{equation}
Now, using the triangle inequality, we obtain that for every $l\geq 1$ there is $M_0(l)$ such that for every $m\geq M_0(l)$ there is $N_0(m)$ such that for every $n\geq N_0(m)$,
$$
 \|(u_{m,n}(0)-u_{0},\partial_t u_{m,n}(0)-u_{1})\|_{H^s(\T^3)\times H^{s-1}(\T^3)}<\frac{1}{l}\,.
$$
Thanks to \eqref{mn}, we obtain that for every $m\geq 1$ there exists $N_1(m)\geq N_0(m)$ such that for every $n\geq N_1(m)$, 
$$
\|(u_{m,n}(t),\partial_t u_{m,n}(t))\|_{L^\infty([0,T];H^s(\T^3)\times H^{s-1}(\T^3))} >l\,.
$$
We now observe that 
$$
u_{l}(t,x)\equiv u_{M_0(l),N_1(M_0(l))}(t,x),\quad l=1,2, 3,\cdots
$$
is a sequence of solutions of the cubic defocusing wave equation satisfying the conclusions of  Theorem~\ref{ill_bis}.
\end{proof}
\begin{rema}
{\rm
It is worth mentioning that we arrive without too much complicated technicalities to a sharp local well-posedness result in the context of the cubic wave equation because  
we do not need a smoothing effect to recover derivative losses 
neither in the nonlinearity nor in the non homogeneous Stricahartz estimates. The $X^{s,b}$ spaces of Bourgain are an efficient tool to deal with these two difficulties. These developments go beyond the scope of these lectures. 
}
\end{rema}
\section{Extensions to more general nonlinearities}
One may consider the wave equation with a more general nonlinearity than the cubic one. Namely, let us consider the nonlinear wave equation
\begin{equation}\label{alphaa}
(\partial_{t}^{2}-\Delta)u+|u|^{\alpha}u=0,
\end{equation}
posed on $\T^3$
where $\alpha>0$ measures the "degree" of the nonlienarity. 
If $u(t,x)$ is a solution of \eqref{alphaa} posed on $\R^3$,  than so is $u_{\lambda}(t,x)=\lambda^{\frac{2}{\alpha}}u(\lambda t,\lambda x)$.
Moreover 
$$
\|u_{\lambda}(t,\cdot)\|_{H^s}\approx \lambda^{\frac{2}{\alpha}}\lambda^{s}\lambda^{-\frac{3}{2}}\|u(\lambda t,\cdot)\|_{H^s}
$$
which implies that $H^s$ with $s=\frac{3}{2}-\frac{2}{\alpha}$ is the critical Sobolev  regularity for \eqref{alphaa}.
Based on this scaling argument one may wish to expect that for $s>\frac{3}{2}-\frac{2}{\alpha}$ the Cauchy problem associated with \eqref{alphaa} is well-posed in $H^s\times H^{s-1}$ and that for $s<\frac{3}{2}-\frac{2}{\alpha}$ it is ill-posed in $H^s\times H^{s-1}$. In this chapter, we verified that this is indeed the case for $\alpha=2$.
For $2<\alpha<4$, a small modification of the proof of  Theorem~\ref{prop.local.low} shows that \eqref{alphaa} is locally well-posed in $H^s\times H^{s-1}$ for $s\in (\frac{3}{2}-\frac{2}{\alpha},\alpha)$.
Then, as in the proof of  Theorem~\ref{prop.global}, we can show that \eqref{alphaa} is globally well-posed in $H^1\times L^2$. 
Moreover a small modification of the proof of Theorem~\ref{ill} shows that for $s\in (0,\frac{3}{2}-\frac{2}{\alpha})$ the Cauchy problem for \eqref{alphaa} is locally ill-posed in $H^s\times H^{s-1}$. 
 For $\alpha=4$, we can prove a local well-posedness statement for \eqref{alphaa} as in  Theorem~\ref{critical}. The global well-posedness in $H^1\times L^2$ for $\alpha=4$ is  much more delicate than the globalisation argument of 
 Theorem~\ref{prop.global}. It is however possible to show that \eqref{alphaa} is globally well-posed in $H^1\times L^2$ (see \cite{Gri1,Gri2,ShStru1,ShStru2}). 
 The new global infirmation for $\alpha=4$,  in addition to the conservation of the energy, is the Morawetz estimate which is a quantitative way to contradict the blow-up criterium in the case $\alpha=4$.
For $\alpha>4$ the Cauchy problem associated with \eqref{alphaa} is still locally well-posed in  $H^s\times H^{s-1}$ for  some $s>\frac{3}{2}-\frac{2}{\alpha}$. 
The global well-posedness (i.e. global existence, uniqueness and propagation of regularity) of \eqref{alphaa}  for $\alpha>4$ is an outstanding open problem. 
For $\alpha>4$, the argument used in Theorem~\ref{ill}  may allow to construct weak solutions in $H^1\times L^2$ with initial data in $H^\sigma$ for $1<\sigma<\frac{3}{2}-\frac{2}{\alpha}$ which are losing their $H^\sigma$ regularity. 
See \cite{L} for such a result for \eqref{alphaa}, posed on $\R^3$.
\chapter[Probabilistic well-posedness]{Probabilistic global well-posedness for the $3d$ cubic wave equation in $H^s$, $s\in [0,1]$ }\label{Chapter2}
\section{Introduction}
Consider again the Cauchy problem for the cubic defocusing wave equation
\begin{equation}\label{NLW_prob}
\begin{gathered}
(\partial_t^2-\Delta) u+u^3=0,\quad u:\R \times \T^3\rightarrow \R,
\\ 
u|_{t=0}=u_0,\,\, \partial_t u|_{t=0}=u_1,\quad\quad (u_0,u_1)\in {\mathcal H}^s(\T^3),
 \end{gathered}
\end{equation}
where 
$$
 {\mathcal H}^s(\T^3)\equiv H^s(\T^3)\times H^{s-1}(\T^3)\,.
 $$
 In the previous chapter, we have shown that \eqref{NLW_prob} is (at least locally in time) well-posed in ${\mathcal H}^s(\T^3)$, $s\geq 1/2$. The main ingredient in the proof for $s\in[1/2,1)$ was the Strichartz estimates for the linear wave equation. We have also shown that for $s\in (0,1/2)$ the Cauchy problem  \eqref{NLW_prob}  is ill-posed in  ${\mathcal H}^s(\T^3)$.
 
One may however ask whether some sort of well-posedness for \eqref{NLW_prob} survives for $s<1/2$. 
We will show bellow that this is indeed possible, if we accept to "randomise" the initial data.
This means that we will endow  ${\mathcal H}^s(\T^3)$, $s\in (0,1/2)$ with suitable probability measures and we will show that the Cauchy problem  \eqref{NLW_prob}  is well-posed in a suitable sense for initial data $(u_0,u_1)$ on a set of full measure.

Let us now describe these measures.  Starting from 
$(u_0,u_1)\in {\mathcal H}^s$ given by their Fourier series 
$$ 
u_{j}(x)=a_{j}+\sum_{n\in\Z^3_{\star}}\Big(b_{n,j}\,\cos(n\cdot x)+c_{n,j}\sin(n\cdot x)\Big), \quad j=0,1,
$$
we define $u_{j}^\omega$ by
\begin{equation}\label{coord}
u_{j}^\omega(x)=\alpha_{j}(\omega)a_{j}+\sum_{n\in\Z^3_{\star}}\Big(\beta_{n,j}(\omega)b_{n,j}\,\cos(n\cdot x)+\gamma_{n,j}(\omega)c_{n,j}\sin(n\cdot x)\Big),
\end{equation}
where $(\alpha_{j}(\omega),\beta_{n,j}(\omega),\gamma_{n,j}(\omega))$, $n\in\Z^3_{\star}$, $j=0,1$ 
is a sequence of real random variables on a probability space $(\Omega,p,{\mathcal F})$.
We assume that the random variables $(\alpha_j,\beta_{n,j},\gamma_{n,j})_{n\in\Z^3_{\star},j=0,1}$ 
are independent identically distributed real random variables with a distribution $\theta$ satisfying 
\begin{equation}\label{subgauss}
\exists\, c>0,\quad \forall\,\gamma\in\R,\quad
\int_{-\infty}^{\infty}e^{\gamma x}d\theta(x)
\leq e^{c\gamma^2}
\end{equation}
(notice that under the assumption \eqref{subgauss} the random variables are necessarily of mean zero).
Typical examples (see Remark~\ref{exx} bellow) of random variables satisfying \eqref{subgauss} are the standard Gaussians, i.e. 
$$
d\theta(x)=(2\pi)^{-\frac{1}{2}}e^{-\frac{x^2}{2}}dx
$$ 
(with an identity in \eqref{subgauss}) or the Bernoulli variables 
$$
d\theta(x)=\frac{1}{2}(\delta_{-1}+\delta_{1})\,.
$$
An advantage of the Bernoulli randomisation is that it keeps the ${\mathcal H}^s$ norm of the original function.
The Gaussian randomisation has the advantage to "generate" a dense set in ${\mathcal H}^s$ via the map 
\begin{equation}\label{eq.proba}
 \omega \in \Omega \longmapsto (u_0^\omega, u_1^\omega)\in{\mathcal H}^s
 \end{equation}
for most of $(u_0,u_1)\in {\mathcal H}^s$ (see Proposition~\ref{prop.1.2} below).
\begin{defi} 
{\rm
For fixed $(u_0, u_1) \in \mathcal{H}^s$, the map \eqref{eq.proba} is a measurable map from $(\Omega,{\mathcal F})$ to ${\mathcal H}^s$ endowed with the Borel sigma algebra 
since the partial sums form a Cauchy sequence
in $L^2(\Omega;{\mathcal H}^s)$. Thus \eqref{eq.proba} endows the space ${\mathcal H}^s(\T^3)$ with a probability measure which is the direct image of $p$. Let us denote this measure by $\mu_{(u_0, u_1)}$. Then
$$\forall\, A \subset \mathcal{H}^s,\,\, 
\mu_{(u_0, u_1)} (A)= p ( \omega\in \Omega\,:\, (u_0^\omega, u_1^\omega) \in A).
$$
Denote by ${\mathcal M}^s$ the set of measures obtained following this construction :
$${\mathcal M}^s= \bigcup_{(u_0, u_1) \in \mathcal{H}^s} \{ \mu_{(u_0, u_1)}\}\,.
$$
}
\end{defi}
Here are two basic properties of these measures.
\begin{prop} \label{prop.1.2}
For any $s'>s$, if 
$(u_0, u_1) \notin \mathcal{H}^{s'}$, then 
$$ 
\mu_{(u_0, u_1)}( \mathcal{H}^{s'})=0\,.
$$
In other words,  the randomisation \eqref{eq.proba} does not regularise in the scale
of the $L^2$-based Sobolev spaces (this fact is obvious for the Bernoulli randomisation). 
Next, if $(u_0,u_1)$ have all their Fourier coefficients different from zero and if
${\rm supp}(\theta)=\R$ then {\rm supp}($\mu_{(u_0, u_1)})={\mathcal H^s}$.
In other words, under these assumptions, for any $(w_0, w_1)\in \mathcal{H}^s$ and any~$\epsilon >0$, 
\begin{equation}\label{eq.dense}
\mu_{(u_0, u_1)} ( \{ (v_0, v_1)\in\mathcal{H}^s\,:\, \| (w_0, w_1) - (v_0, v_1)\|_{\mathcal{H}^s} < \epsilon \}) >0,
\end{equation} 
or in yet other words,  any  set of full $\mu_{(u_0, u_1)}$-measure is dense in $\mathcal{H}^s$ .
\end{prop}
We have the following global existence and uniqueness result for typical data with respect to an element of ${\mathcal M}^s$. 
\begin{theo}[existence and uniqueness]\label{main}
Let us fix $s\in(0,1)$ and  $\mu\in {\mathcal M}^s$. 
Then, there exists a full $\mu$ measure set $\Sigma \subset  {\mathcal H}^s(\T^3)$ 
such that for every $(v_0, v_1)\in \Sigma$, there exists a unique global solution $v$  of the nonlinear wave equation
 \begin{equation}\label{valna}
 (\partial_t^2-\Delta)v+v^3=0,\quad (v(0),\partial_t v(0))=(v_0,v_1)
 \end{equation}
satisfying
 $$
(v(t),\partial_t v(t)) \in \big(S(t)(v_0, v_1),\partial_t S(t)(v_0, v_1)\big)+ C(\R; H^1(\T^3) \times L^2(\T^3)).
 $$ 
Furthermore, if we denote by 
$$\Phi(t) (v_0, v_1)\equiv (v(t),\partial_t v(t))$$ 
the flow thus defined, the set $\Sigma$ is invariant by the map $\Phi(t)$, namely
$$
\Phi(t)(\Sigma)=\Sigma,\qquad \forall\, t\in \R.
$$
\end{theo}
The next statement gives quantitative bounds on the solutions.
\begin{theo}[quantitative bounds]\label{bornes}
Let us fix $s\in(0,1)$ and  $\mu\in {\mathcal M}^s$.  Let $\Sigma$ be the set constructed in  Theorem~\ref{main}.
Then for every $\varepsilon>0$ there exist $C, \delta >0$ such that for every $(v_0, v_1)\in \Sigma$, 
there exists $M>0$ such that the global solution to~\eqref{valna} constructed in Theorem~\ref{main} satisfies 
$$
v(t)= S(t) \Pi_0^{\perp}(v_0, v_1)+ w(t),
$$
with
\begin{equation*}
\|(w(t), \partial_t w(t)) \|_{\mathcal{H}^1(\T^3)} \leq 
C (M+ |t|)^{\frac {1-s} s + \varepsilon} 
\end{equation*}
and 
\begin{equation*}
\mu ((v_0,v_1) \,:\,M>\lambda) \leq  C e^{-\lambda^\delta}. 
\end{equation*}
\end{theo}
\begin{rema}
{\rm
Recall that $\Pi_0$ is the orthogonal projector on the zero Fourier mode and $\Pi_0^{\perp} = {\text{Id}} - \Pi_0$.
}
\end{rema}
We now further discuss the uniqueness of the obtained solutions. For $s>1/2$, we have the following statement. 
\begin{theo}[unique limit of smooth solutions for $s>1/2$]\label{uniqueness_1}
Let $s\in (1/2,1)$.
With the notations of the statement of Theorem~\ref{main}, let  us fix an initial datum $(v_0, v_1)\in \Sigma$ with a corresponding global solution $v(t)$.
Let $(v_{0,n},v_{1,n})_{n=1}^{\infty}$ be a sequence of ${\mathcal H}^1(\T^3)$ such that
$$
\lim_{n\rightarrow\infty}\|(v_{0,n}-v_0,v_{1,n}-v_1)\|_{{\mathcal H}^s(\T^3)}=0\,.
$$
Denote by $v_n(t)$ the solution of the cubic defocusing wave equation with data $(v_{0,n}, v_{1,n})$ defined in
Theorem~\ref{prop.global}. Then for every $T>0$,
$$
\lim_{n\rightarrow\infty}\|(v_{n}(t)-v(t),\partial_t v_{n}(t)-\partial_t v(t))\|_{L^\infty([0,T]; {\mathcal H}^s(\T^3))}=0\,.
$$
\end{theo}
Thanks to  Theorem~\ref{ill_bis}, we know that for $s\in (0,1/2)$ the result of Theorem~\ref{uniqueness_1} cannot hold true !
We only have a partial statement. 
\begin{theo}[unique limit of particular smooth solutions for $s<1/2$]\label{uniqueness_2}
Let $s\in (0,1/2)$.
With the notations of the statement of Theorem~\ref{main}, let  us fix an initial datum $(v_0, v_1)\in \Sigma$ with a corresponding global solution $v(t)$.
Let $(v_{0,n},v_{1,n})_{n=1}^{\infty}$ be the sequence of $C^\infty(\T^3)\times C^\infty(\T^3)$ defined as the usual regularisation by convolution, i.e. 
$$
v_{0,n}=v_{0}\star \rho_n,\quad  v_{1,n}=v_{1}\star \rho_n\,,
$$
where $(\rho_n)_{n=1}^{\infty}$ is an approximate identity. 
Denote by $v_n(t)$ the solution of the cubic defocusing wave equation with data $(v_{0,n}, v_{1,n})$ defined in
Theorem~\ref{prop.global}. Then for every $T>0$,
$$
\lim_{n\rightarrow\infty}\|(v_{n}(t)-v(t),\partial_t v_{n}(t)-\partial_t v(t))\|_{L^\infty([0,T]; {\mathcal H}^s(\T^3))}=0\,.
$$
\end{theo}
\begin{rema}
{\rm 
We emphasise that the result of  Theorem~\ref{ill_bis} applies for the elements of $\Sigma$. More precisely, thanks to   Theorem~\ref{ill_bis}, we have that 
for every $(v_0, v_1)\in \Sigma$ there is a sequence  $(v_{0,n},v_{1,n})_{n=1}^{\infty}$  of elements of $C^\infty(\T^3)\times C^\infty(\T^3)$
such that 
$$
\lim_{n\rightarrow\infty}\|(v_{0,n}-v_0, v_{1,n}-v_1)\|_{{\mathcal H}^s(\T^3)}=0
$$
but such that if we denote by  $v_n(t)$ the solution of the cubic defocusing wave equation with data $(v_{0,n}, v_{1,n})$ defined in
Theorem~\ref{prop.global} then for every $T>0$, 
$$
\lim_{n\rightarrow\infty}\|(v_{n}(t),\partial_t v_{n}(t))\|_{L^\infty([0,T];  {\mathcal H}^s(\T^3))}=\infty\,.
$$
Therefore the choice of the particular regularisation of the initial data in Theorem~\ref{uniqueness_2} is of key importance.
It would be interesting to classify the  "admissible type of regularisations"  allowing to get a statement such as  Theorem~\ref{uniqueness_2} .
}
\end{rema}
\begin{rema}
{\rm 
We can also see the solutions constructed in Theorem~\ref{main} as  the (unique) limit as $N$ tends to infinity of the solutions of the following truncated versions of the cubic defocusing wave equation.
$$
 (\partial_t^2-\Delta)S_N u+S_{N}((S_N u)^3)=0,
$$
where $S_N$ is a Fourier multiplier localising on modes  of size $\leq N$. 
The convergence of a subsequence can be obtained by a compactness argument (cf. \cite{BTT-smf}).
The convergence of the whole sequence however requires strong solutions techniques. 
}
\end{rema}
The next question is whether some sort of continuous dependence with respect to the initial data survives in the context of  Theorem~\ref{main}.
In order to state our result concerning the continuous dependence with respect to the initial data, we  
recall that for any event $B$ (of non vanishing probability) the conditioned probability $p( \cdot \vert B)$ is the natural probability measure supported by $B$, defined by 
$$ p( A\vert B) = \frac{ p (A\cap B) } {p( B)}\,.
$$ 
We have the following statement.
\begin{theo}[conditioned continuous dependence]\label{th_continuity}
Let us fix  $s\in (0,1)$,  let $A>0$, let $B_A\equiv (V\in {\mathcal H}^s : \|V\|_{{\mathcal H}^s}\leq A)$  
be the closed ball of radius $A$ centered at the origin of ${\mathcal H}^s$ and let $T>0$. 
Let  $\mu\in {\mathcal M}^s$ and suppose that $\theta$ (the law of our random variables) is symmetric.
Let $\Phi(t)$ be the flow of the cubic wave equations defined $\mu$ almost everywhere in Theorem~\ref{main}.
Then for $\varepsilon, \eta>0$,  we have the bound 
\begin{multline}\label{dimanche}
 \mu\otimes\mu\Big((V,V')\in {\mathcal H}^s\times {\mathcal H}^s\,:
  \| \Phi(t) (V) - \Phi(t) (V') \|_{X_T} >\varepsilon  \Bigm {\vert} 
  \\
  \| V-V'\|_{\mathcal{H}^s}< \eta \,\,{\rm and}\,\,( V,V')\in B_A\times B_A   \Big) \leq  g(\varepsilon,\eta),
\end{multline}
where
$
X_{T}\equiv (C ([0,T]; \mathcal{H}^s)\cap L^4([0,T]\times\T^3))\times C([0,T];H^{s-1})
$
and  $g(\varepsilon,\eta)$ is such that
$$
\lim_{\eta\rightarrow 0}g(\varepsilon,\eta)=0,\qquad \forall\,\varepsilon>0.
$$
Moreover, if for $s\in(0,1/2)$  we assume in addition that the support of $\mu$ is the whole ${\mathcal H}^s$ (which is true if in the definition of the measure $\mu$, we have $a_{i}, b_{n,j}, c_{n,j} \neq 0,  \forall n \in \mathbb{Z}^d$ and the support of the distribution function of the random variables is $\mathbb{R}$), then there exists 
$\varepsilon>0$ such that for every $\eta>0$ the left hand-side in \eqref{dimanche} is positive.
\end{theo} 
A probability measure $\theta$ on $\R$ is called symmetric if 
$$
\int_{\R}f(x)d\theta(x)=\int_{\R}f(-x)d\theta(x),\quad \forall\,\, f\in L^1(d\theta).
$$
A real random variable is called symmetric if its distribution is a symmetric measure on $\R$.

The result of Theorem~\ref{th_continuity} is saying that as soon as $\eta \ll\varepsilon$, among the initial data which are $\eta$-close to each other, the probability of finding two for which the corresponding solutions to~\eqref{NLW_prob} do not remain $\varepsilon$ close to each other, is very small. 
The last part of the statement is saying that  the deterministic version of 
the uniform continuity property \eqref{dimanche} does not hold and somehow that one cannot get rid of a probabilistic approach in the question concerning
 the continuous dependence (in ${\mathcal H}^s$, $s<1/2$) with respect to the data. The  ill-posedenss result of Theorem~\ref{ill} will be of importance in the proof of the last part of Theorem~\ref{th_continuity}.
\section{Probabilistic Strichartz estimates}
\begin{lemm}\label{lem1}
Let $(l_n(\omega))_{n=1}^{\infty}$ 
be a sequence of real, independent random variables with associated sequence of
distributions $(\theta_{n})_{n=1}^{\infty}$.
Assume that $\theta_{n}$ satisfy the property
\begin{equation}\label{property}
\exists\, c>0\,:\, \forall\,\gamma\in\R,\,\forall\, n\geq 1,\,
\Big|\int_{-\infty}^{\infty}e^{\gamma x}d\theta_{n}(x)\Big|\leq e^{c\gamma^{2}}\,.
\end{equation}
Then there exists $\alpha>0$ such that 
for every $\lambda >0$, every sequence $(c_n)_{n=1}^{\infty}\in l^2$ of real numbers,
\begin{equation}\label{khin}
p\Big(\omega\,:\,\big|\sum_{n=1}^{\infty} c_n l_n(\omega) \big|>\lambda\Big)\leq
2 e^{-\frac{\alpha\lambda^2}{\sum_{n}c_n^2}}
\, .
\end{equation}
As a consequence there exists $C>0$ such that for every $p\geq 2$, every $(c_n)_{n=1}^{\infty}\in l^2$,
\begin{equation}\label{khinbis}
\big\|\sum_{n=1}^{\infty} c_n l_n(\omega) \big\|_{L^{p}(\Omega)}
\leq C\sqrt{p}\big(\sum_{n=1}^{\infty}c_n^2\big)^{1/2}
.
\end{equation}
\end{lemm}
\begin{rema}
{\rm
The property \eqref{property} is equivalent to assuming that $\theta_n$ are of zero mean and assuming that 
\begin{equation}\label{compp}
\exists\, c>0,\, C>0\,:\, \forall\,\gamma\in\R,\,\forall\, n\geq 1,\,
\Big|\int_{-\infty}^{\infty}e^{\gamma x}d\theta_{n}(x)\Big|\leq C\, e^{c\gamma^{2}}\,.
\end{equation}
}
\end{rema}
\begin{rema}\label{exx} 
{\rm
Let us notice that (\ref{property}) is readily satisfied if $(l_n(\omega))_{n=1}^{\infty}$ are standard 
real Gaussian or standard Bernoulli variables. Indeed in the case of Gaussian
$$
\int_{-\infty}^{\infty}e^{\gamma x}d\theta_{n}(x)=
\int_{-\infty}^{\infty}e^{\gamma x}\, e^{-x^2/2}\frac{dx}{\sqrt{2\pi}}
=e^{\gamma^2/2}\,.
$$
In the case of  Bernoulli variables one can obtain that (\ref{property}) 
is satisfied by invoking the inequality
$$
\frac{e^{\gamma }+e^{-\gamma }}{2}\leq e^{\gamma^2/2},\quad \forall\, \gamma\in\R.
$$
More generally, we can observe that \eqref{compp} holds if $\theta_n$ is compactly supported. 
}
\end{rema}
\begin{rema}
{\rm
In the case of Gaussian we can see Lemma~\ref{lem1} as a very particular case of a $L^p$ smoothing 
properties of the Hartree-Foch heat flow (see e.g. \cite[Section~3]{Tz3} for more details on this issue).
}
\end{rema}
\begin{proof}[Proof of Lemma~\ref{lem1}.]
For $t>0$ to be determined later, using the independence and (\ref{property}), we obtain 
\begin{eqnarray*}
\int_{\Omega}\, e^{t\sum_{n\geq 1}c_n l_n(\omega)}dp(\omega)
& = & \prod_{n\geq 1}\int_{\Omega}e^{t c_n l_n(\omega)}dp(\omega)
\\
& = &  \prod_{n\geq 1}\int_{-\infty}^{\infty}e^{tc_n x}\, d\theta_{n}(x)
\\
& \leq & 
\prod_{n\geq 1}e^{c(t c_n)^2}= e^{(ct^2)\sum_{n}c_n^2}\, .
\end{eqnarray*}
Therefore
$$
e^{(ct^2)\sum_{n}c_n^2}\geq e^{t\lambda}\,\,\,  p\,(\omega\,:\,\sum_{n\geq 1} c_n l_n(\omega)>\lambda)
$$
or equivalently,
$$
p\,(\omega\,:\,\sum_{n\geq 1} c_n l_n(\omega)>\lambda)\leq e^{(ct^2)\sum_{n}c_n^2}\,\,\,
e^{-t\lambda}\, .
$$
We choose $t$ as
$$
t\equiv \frac{\lambda}{2c \sum_{n}c_n^2}\,.
$$
Hence
$$
p\,(\omega\,:\,\sum_{n\geq 1} c_n l_n(\omega)>\lambda)\leq
e^{-\frac{\lambda^2}{4c\sum_{n}c_n^2}}\, .
$$
In the same way (replacing $c_n$ by $-c_n$), we can show that
$$
p\,(\omega\,:\,\sum_{n\geq 1} c_n l_n(\omega)<-\lambda)\leq
e^{-\frac{\lambda^2}{4c\sum_{n}c_n^2}}
$$
which completes the proof of~\eqref{khin}. To deduce~\eqref{khinbis}, we  write
\begin{eqnarray*}
\|\sum_{n=1}^{\infty} c_n l_n(\omega)\|^p_{L^p(\Omega)} & = & 
{p} \int_0^{+\infty} p(\omega\,:\, |\sum_{n=1}^{\infty} c_n l_n(\omega)|>\lambda ) \lambda^{p-1} d \lambda
\\
& \leq  &
Cp \int_0^{+\infty} \lambda^{p-1} e^{-\frac{c \lambda^2}{\sum_n c_n^2}} d\lambda
\\
& \leq & Cp  (C\sum_n c_n^2)^{\frac p 2}\int_0^{+\infty} \lambda^{p-1} e^{-\frac{\lambda^2} 2} d\lambda
\\
& \leq & C  (Cp \sum_n c_n^2)^{\frac p 2}
\end{eqnarray*}
which completes the proof of Lemma~\ref{lem1}.
\end{proof}
As a consequence of  Lemma~\ref{lem1}, we get the following "probabilistic" Strichartz estimates.
\begin{theo}\label{lll}
Let us fix $s\in (0,1)$ and let  $\mu\in {\mathcal M}^s$ be induced via the map \eqref{eq.proba} from the couple $(u_0,u_1)\in {\mathcal H}^s$.
Let us also fix  $\sigma\in (0,s]$, $2\leq p_1<+\infty$, $2\leq p_2 \leq + \infty$ and $\delta >1+ \frac 1 {p_1}$.
Then there exists a positive constant $C$ such that for every $p\geq 2$,
\begin{equation}\label{kando1}
\Big\|
\|\langle t \rangle ^{- \delta} S(t) (v_0,v_1)\|_{L^{p_1} (  \R_t ; L^{p_2}( \T^3))}
\Big\|_{L^p({\mu})}
\leq  C\sqrt{p}
\|(u_0, u_1)\|_{\mathcal{H}^{\sigma}(\T^3)}\,.
\end{equation}
As a consequence for every $T>0$ and $p_1\in [1,\infty)$, $p_2\in [2,\infty]$,
\begin{equation}\label{stri_prob}
 \|S(t) (v_0,v_1)\|_{L^{p_1}([0,T] ; L^{p_2}( \T^3))}<\infty ,\quad \mu{\rm\,-\,almost\,\, surely.}
\end{equation}
Moreover, there exist two positive constants $C$ and $c$ such that for
every $\lambda>0$,
\begin{multline}\label{kando2}
\mu\Big((v_0,v_1)\in {\mathcal H}^s\,:\,
\|\langle t \rangle ^{- \delta} S(t) (v_0,v_1)\|_{L^{p_1} (  \R_t ; L^{p_2}( \T^3))}> \lambda 
\Big)\leq 
\\
C\exp\Bigl(- \frac {c\lambda^2} {\|(u_0, u_1)\|^2_{\mathcal{H}^{\sigma}(\T^3)}}\Bigr)\,.
\end{multline}
\end{theo}
\begin{rema}
{\rm
Observe that \eqref{stri_prob} applied for $p_2=\infty$ displays an improvement of $3/2$ derivatives with respect to the Sobolev embedding which is stronger than the improvement obtained by the (deterministic) Strichartz estimates  (see Remark~\ref{beat_sob}). The proof of Theorem~\ref{lll} exploits the random oscillations of the initial data while the proof of the deterministic Strichartz estimates exploits in a crucial (and subtle) manner the time oscillations of  $S(t)$. In the proof of Theorem~\ref{lll}, we simply neglect these times oscillations. 
}
\end{rema}
\begin{rema}
{\rm
In the proof of Theorem~\ref{lll}, we shall make use of the Sobolev spaces $W^{\sigma,q}(\T^3)$, $\sigma\geq 0$, $q\in (1,\infty)$, defined via the norm
$$
\|u\|_{W^{\sigma,q}(\T^3)}=\|(1-\Delta)^{\sigma/2}u\|_{L^q(\T^3)}\,.
$$
}
\end{rema}
\begin{proof}[Proof of Theorem~\ref{lll}]
We have that
$$
\Big\|
\|\langle t \rangle ^{- \delta}\Pi_0  S(t) (v_0,v_1)\|_{L^{p_1} (  \R_t ; L^{p_2}( \T^3))}
\Big\|_{L^p({\mu})}
$$
equals
\begin{equation}\label{expp1}
\Big\|
\|\langle t \rangle ^{- \delta}
(\alpha_{0}(\omega)a_0+t\alpha_1(\omega) a_1)
\|_{L^{p_1} (  \R_t ; L^{p_2}( \T^3))}
\Big\|_{L^p_{\omega}}\,.
\end{equation}
A trivial application of  Lemma~\ref{lem1} implies that 
$$
\|\alpha_j(\omega)\|_{L^p_{\omega}}
\leq C\sqrt{p},\quad j=0,1.
$$
Therefore, using that $\delta>1+1/p_1$ the expression \eqref{expp1} can be bounded by
$$
(2\pi)^{\frac{3}{p_2}}
\Big\|
\|\langle t \rangle ^{- \delta}
(\alpha_{0}(\omega)a_0+t\alpha_1(\omega) a_1)
\|_{L^{p_1} (  \R_t )}
\Big\|_{L^p_{\omega}}
\leq 
C\sqrt{p}(|a_0|+|a_1|)\,.
$$
Therefore, it remains to estimate 
$$
\Big\|
\|\langle t \rangle ^{- \delta}\Pi_0^{\perp}  S(t) (v_0,v_1)\|_{L^{p_1} (  \R_t ; L^{p_2}( \T^3))}
\Big\|_{L^p({\mu})}\,.
$$
By a use of the H\"older inequality on $\T^3$, we observe that it suffices to estimate 
$$
\Big\|
\|\langle t \rangle ^{- \delta}\Pi_0^{\perp}  S(t) (v_0,v_1)\|_{L^{p_1} (  \R_t ; L^{\infty}( \T^3))}
\Big\|_{L^p({\mu})}\,.
$$
Let $q<\infty$ be such that $\sigma>3/q$. Then by the Sobolev embedding $W^{\sigma,q}(\T^3)\subset C^0(\T^3)$, we have
$$
\|\Pi_0^{\perp}  S(t) (v_0,v_1)\|_{L^\infty(\T^3)}\leq C
\|(1-\Delta)^{\sigma/2}\Pi_0^{\perp}  S(t) (v_0,v_1)\|_{L^q(\T^3)}\,.
$$
Therefore, we need to estimate 
$$
\Big\|
\|\langle t \rangle ^{- \delta}
(1-\Delta)^{\sigma/2}
\Pi_0^{\perp}  S(t) (v_0,v_1)\|_{L^{p_1} (  \R_t ; L^{q}( \T^3))}
\Big\|_{L^p({\mu})}
$$
which equals 
\begin{equation}\label{pl1}
\Big\|
\|\langle t \rangle ^{- \delta}
(1-\Delta)^{\sigma/2}
\Pi_0^{\perp}  S(t) (u_0^\omega,u_1^\omega)\|_{L^{p_1} (  \R_t ; L^{q}( \T^3))}
\Big\|_{L^p_{\omega}}\,.
\end{equation}
By using the H\"older inequality in $\omega$, we observe that it suffices to evaluate the last quantity only for $p>\max(p_1,q)$.
For such values of $p$, using the Minkowski inequality, we can estimate \eqref{pl1} by
\begin{equation}\label{pl2}
\Big\|
\big\|
\langle t \rangle ^{- \delta}
(1-\Delta)^{\sigma/2}
\Pi_0^{\perp}  S(t) (u_0^\omega,u_1^\omega)
\big\|_{L^p_{\omega}}
\Big\|_{L^{p_1} (  \R_t ; L^{q}( \T^3))}\,.
\end{equation}
Now, we can write $(1-\Delta)^{\sigma/2}
\Pi_0^{\perp}  S(t) (u_0^\omega,u_1^\omega)$ as
\begin{multline*}
\sum_{n\in\Z^3_{\star}}
\langle n\rangle^{\sigma}\Big(\big(\beta_{n,0}(\omega)b_{n,0}\cos(t|n|)+\beta_{n,1}(\omega)b_{n,1}\frac{\sin(t|n|)}{|n|}\big)\cos(n\cdot x)
\\
+\big(\gamma_{n,0}(\omega)c_{n,0}\cos(t|n|)+\gamma_{n,1}(\omega)c_{n,1}\frac{\sin(t|n|)}{|n|}\big)\sin(n\cdot x)\Big),
\end{multline*}
with
$$
\sum_{n\in\Z^3_{\star}}
\langle n\rangle^{2\sigma}\Big(|b_{n,0}|^2+|c_{n,0}|^2+|n|^{-2}(|b_{n,1}|^2+|c_{n,1}|^2)\Big)
\leq C\|(u_0, u_1)\|_{\mathcal{H}^\sigma(\T^3)}^2\,.
$$
Now using \eqref{khinbis} of Lemma~\ref{lem1} and the boundedness of $\sin$ and $\cos$ functions, we obtain that \eqref{pl2} can be bounded by 
\begin{equation}\label{pl3}
C\Big\|
\langle t \rangle ^{- \delta}
C\sqrt{p}\|(u_0, u_1)\|_{\mathcal{H}^\sigma(\T^3)}
\Big\|_{L^{p_1} (  \R_t ; L^{q}( \T^3))}\,.
\end{equation}
Since $\delta>1+1/p_1$, we can estimate \eqref{pl3} by 
$$
C\sqrt{p}\|(u_0, u_1)\|_{\mathcal{H}^\sigma(\T^3)}\,.
$$
This completes the proof of  \eqref{kando1} 
Let us finally show how \eqref{kando1} implies \eqref{kando2}.
Using the Tchebichev inequality and \eqref{kando1}, we have that
$$
\mu\Big((v_0,v_1)\in {\mathcal H}^s\,:\,
\|\langle t \rangle ^{- \delta} S(t) (v_0,v_1)\|_{L^{p_1} (  \R_t ; L^{p_2}( \T^3))}> \lambda 
\Big)\
$$
is bounded by
$$
\lambda^{-p}
\Big\|
\|\langle t \rangle ^{- \delta} S(t) (v_0,v_1)\|_{L^{p_1} (  \R_t ; L^{p_2}( \T^3))}
\Big\|_{L^p({\mu})}^p\leq
\big(C\lambda^{-1}\sqrt{p}\|(u_0, u_1)\|_{\mathcal{H}^\sigma(\T^3)}
\big)^p
$$
We now choose $p$ as  
\begin{equation*}
C\lambda^{-1}\sqrt{p}\|(u_0, u_1)\|_{\mathcal{H}^\sigma(\T^3)}= \frac 1 2
 \Leftrightarrow  p= \frac{ \lambda^2 \|(u_0, u_1)\|^{-2}_{\mathcal{H}^\sigma(\T^3)} } {4C^2} ,
\end{equation*}
which yields \eqref{kando2}.
This completes the proof of  Theorem~\ref{lll}.
\end{proof}
The proof of Theorem~\ref{lll}  also implies the following statement. 
\begin{theo}\label{lll_bis}
Let us fix $s\in (0,1)$ and let  $\mu\in {\mathcal M}^s$ be induced via the map \eqref{eq.proba} from the couple $(u_0,u_1)\in {\mathcal H}^s$.
Let us also fix  $p\geq 2$, $\sigma\in (0,s]$ and $q<\infty$ such that $\sigma>3/q$.
Then for every $T>0$,
\begin{equation}\label{stri_prob_bis}
 \|S(t) (v_0,v_1)\|_{L^{p}([0,T] ; W^{\sigma,q}( \T^3))}<\infty ,\quad \mu{\rm\,-\,almost\,\, surely.}
\end{equation}
\end{theo}
\section{Regularisation effect in the Picard iteration expansion}
Consider the Cauchy problem
\begin{equation}\label{NLW_prob_kik}
(\partial_t^2-\Delta) u+u^3=0,\quad 
u|_{t=0}=u_0,\,\, \partial_t u|_{t=0}=u_1,
\end{equation}
where $(u_0, u_1)$ is a typical element on the support of $\mu\in {\mathcal M}^s$, $s\in (0,1)$.
According to the discussion in Section~\ref{picard} of the previous chapter, for small times depending on $(u_0,u_1)$, we can hope to represent the solution of \eqref{NLW_prob_kik} as 
$$
u=\sum_{j=1}^{\infty}Q_{j}(u_0,u_1),
$$
where $Q_j$ is homogeneous of order $j$ in $(u_0,u_1)$. 
We have that 
\begin{eqnarray*}
Q_1(u_0,u_1) & = & S(t)(u_0,u_1),
\\
Q_{2}(u_0,u_1) & = & 0,
\\
Q_{3}(u_0,u_1) & = &-\int_{0}^t \frac{\sin((t-\tau)\sqrt{-\Delta})}{\sqrt{-\Delta}}
\big(
S(\tau)(u_0,u_1)
\big)^3d\tau,
\end{eqnarray*}
etc. 
We have that $\mu$ a.s. $Q_1\notin H^\sigma$ for $\sigma>s$. 
However, using the probabilistic Strichartz estimates of the previous section, we have that for $T>0$,
$$
\|Q_3(u_0,u_1)\|_{L^\infty_TH^1(\T^3)}\lesssim \| S(t)(u_0,u_1)\|_{L^3_T L^6(\T^3)}^3<\infty,\quad \mu{\rm\,-\,almost\,\, surely.}
$$
Therefore the second non trivial term in the formal expansion defining the solution is more regular than the initial data !
The strategy will therefore be to write the solution of \eqref{NLW_prob_kik} as 
$$
u=Q_1(u_0,u_1)+v,
$$
where $v\in H^1$ and solve the equation for $v$ by the methods described in the previous chapter. In the case of the cubic nonlinearity the deterministic analysis used to solve the equation for $v$ is particularly simple, it is in fact very close to the analysis in the proof of Proposition~\ref{prop.local}. 
For more complicated problems the analysis of the equation for $v$ could involve more advanced deterministic arguments. 
We refer to \cite{B2}, where a similar strategy is used in the context of the nonlinear Schr\"odinger  equation and to \cite{DD} where it is used in the context of stochastic PDE's.

This argument is not particularly restricted to $Q_3$. One can imagine  situations when for some $m>3$, $Q_m$ is the first element in the expansion whose regularity fits well in a deterministic analysis.
Then we can equally well look for the solutions under the form
\begin{equation}\label{bbnn}
u=\sum_{j=1}^{m-1} Q_j(u_0,u_1)+v, 
\end{equation}
and treat $v$ by a deterministic analysis.  It is worth noticing that such  a situation occurs in the work on parabolic PDE's with a singular random source term \cite{GIP, H1, H2}. 
In these works in expansions of type \eqref{bbnn} the random initial data $(u_0,u_1)$ should be replaced by the random source term (the white noise). Let us also mention that in the case of parabolic equations the deterministic smoothing comes from elliptic regularity estimates while in the context of the wave equation we basically rely on the smoothing estimate \eqref{wave_regularity}. 
\section{The local existence result}
\begin{prop}\label{prop.local_pak}
Consider the problem
\begin{equation}\label{model_pak}
(\partial_t^2-\Delta)v+(f+v)^3=0\,.
\end{equation}
There exists a constant $C$ such that for every time interval $I=[a,b]$ of size $1$, 
every $\Lambda\geq 1$, every 
$
(v_0,v_1,f)\in H^1\times L^2\times L^3(I,L^6)
$
satisfying
$$
\|v_0\|_{H^1}+\|v_1\|_{L^2}+\|f\|^3_{L^3(I,L^6)}\leq \Lambda
$$
there exists a unique solution on the time interval $[a,a+C^{-1}\Lambda^{-2}]$ of \eqref{model_pak} with initial data
$$
v(a,x)=v_0(x), \quad \partial_t v(a,x)=v_1(x)\,.
$$
Moreover the solution satisfies
$
\|(v,\partial_t v)\|_{L^\infty([a,a+C^{-1} \Lambda^{-2}],H^1\times L^2)}\leq C\Lambda,
$
$(v,\partial_t v)$ is unique in the class $L^\infty([a,a+C^{-1} \Lambda^{-2}],H^1\times L^2)$ and the dependence in time is continuous.
\end{prop}
\begin{proof}
The proof is very similar to the proof of Proposition~\ref{prop.local}.
By translation invariance in time, we can suppose that $I=[0,1]$.
We can rewrite the problem as
\begin{equation}\label{Duhamel_pak}
v(t)=S(t)(v_0,v_1)-\int_{0}^t \frac{\sin((t-\tau)\sqrt{-\Delta})}{\sqrt{-\Delta}}((f(\tau)+v(\tau))^3d\tau\,.
\end{equation}
Set
$$
\Phi_{v_0,v_1,f}(v)\equiv  S(t)(v_0,v_1)-\int_{0}^t \frac{\sin((t-\tau)\sqrt{-\Delta})}{\sqrt{-\Delta}}((f(\tau)+v(\tau))^3d\tau.
$$
Then for $T\in (0,1]$, using the Sobolev embedding $H^1(\T^3)\subset L^6(\T^3)$, we get
\begin{multline*}
\|\Phi_{v_0,v_1,f}(v)\|_{L^\infty([0,T],H^1)} 
\\
 \leq   C\big(\|v_0\|_{H^1}+\|v_1\|_{L^2}+\int_0^T\|f(\tau)\|_{L^6}^3d\tau\big)+ T\sup_{\tau\in[0,T]}\|v(\tau)\|_{L^6}^3
\\
\leq  C\big(\|v_0\|_{H^1}+\|v_1\|_{L^2}+\|f\|^3_{L^3(I,L^6)}\big)+T\|v\|^3_{L^\infty([0,T],H^1)}.
\end{multline*}
It is now clear that for $T\approx \Lambda^{-2}$ the map $\Phi_{u_0,u_1,f}$ send the ball
$$
\{v:\|v\|_{L^\infty([0,T],H^1)}\leq C\Lambda\}
$$
into itself. Moreover by a similar argument, we obtain that this map is a contraction on the
same ball. Thus we obtain the existence part and the bound on $v$ in $H^1$. The estimate of $\|\partial_t v\|_{L^2}$ follows by differentiating
in $t$ the Duhamel formula \eqref{Duhamel_pak}. 
This completes the proof of Proposition~\ref{prop.local_pak}.
\end{proof}
\section{Global existence}
In this section, we complete the proof of Theorem~\ref{main}.
We search $v$ under the form $v(t)= S(t) (v_0, v_1) +w(t)$. Then $w$ solves
\begin{equation}\label{eq.1}
(\partial_t^2-\Delta)w+(S(t)(v_0, v_1)+w)^3=0,\quad w\mid_{t=0} =0, \quad \partial_t w\mid_{t=0} = 0.
\end{equation}
Thanks to Theorem~\ref{lll} and  Theorem~\ref{lll_bis},  we have that $\mu$-almost surely, 
\begin{equation}\label{eq.borne}
\begin{aligned}
g(t) &= \|S(t)(v_0, v_1)\|^3_{L^6( \T^3)}\in L^1_{\text{loc}} ( \R_t),
\\
f(t) &=  \|S(t)(v_0, v_1)\|_{W^{\sigma,q}( \T^3)}\in L^1_{\text{loc}} ( \R_t),
\end{aligned}
\end{equation}
$\sigma>3/q$. The local existence for \eqref{eq.1}  follows from Proposition~\ref{prop.local_pak} and the first estimate in~\eqref{eq.borne}. 
We also deduce from Proposition~\ref{prop.local_pak}, 
that as long as the $H^1\times L^2$ norm of $(w, \partial_t w)$ remains bounded, the solution $w$ of~\eqref{eq.1} exists. Set
$$ 
{\mathcal E}(w(t)) = \frac 1 2 \int_{\T^3}\big( (\partial_t w)^2 + |\nabla_x w|^2 + \frac 1 2 w^4\big) dx\,.
$$
Using the equation solved by $w$, we now compute 
\begin{eqnarray*}
\frac{d} {dt} {\mathcal E}(w(t)) &=  &\int_{\T^3}\big( \partial_t w \partial_t^2 w + \nabla_x \partial _t w \cdot \nabla_x w + \partial_t w\, w ^3 \big)dx
\\
& =  &\int_{\T^3} \partial_t w \Bigl(\partial_t^2 w -\Delta w + w^3\Bigr) dx
\\
& =  &\int_{\T^3} \partial_t w \Bigl(w^3-  (S(t)(v_0, v_1)+ w)^3\Bigr) dx.
\end{eqnarray*}
Now, using the Cauchy-Schwarz inequality, the H\"older inequalities and the Sobolev embedding  $W^{\sigma,q}(\T^3)\subset C^0(\T^3)$, we can write
\begin{equation*}
\begin{aligned}
\frac{d} {dt} {\mathcal E}(w(t))  &\leq C \big({\mathcal E}(w(t))\big)^{1/2}  \|w^3-  (S(t)(v_0, v_1)+ w)^3\|_{L^2( \T^3)}
\\
&\leq C \big({\mathcal E}(w(t))\big)^{1/2} 
\Bigl(\|S(t)(v_0, v_1)\|^3_{L^6( \T^3)} + \|S(t)(v_0, v_1)\|_{L^\infty( \T^3)} \|w^2\|_{L^2( \T^3)} \Bigr)\\
&\leq C \big({\mathcal E}(w(t))\big)^{1/2} 
\Bigl(\|S(t)(v_0, v_1)\|^3_{L^6( \T^3)} + \|S(t)(v_0, v_1)\|_{W^{\sigma,q}( \T^3)} \|w^2\|_{L^2( \T^3)} \Bigr)\\
&\leq 
C \big({\mathcal E}(w(t))\big)^{1/2} 
 \Big(g(t) + f(t) \big({\mathcal E}(w(t))\big)^{1/2}  \Big)
\end{aligned}
\end{equation*}
and consequently, according to Gronwall inequality and~\eqref{eq.borne}, $w$ exists globally in time.

This completes the proof of the existence and uniqueness part of Theorem~\ref{main}.
Let us now turn to the construction of an invariant set. 
Define the sets
\begin{multline*}
\Theta\equiv
\big\{
(v_0,v_1)\in {\mathcal H}^s\,:\,\|S(t)(v_0, v_1)\|^3_{L^6( \T^3)}\in L^1_{\text{loc}} ( \R_t),
\\
 \|S(t)(v_0, v_1)\|_{W^{\sigma,q}( \T^3)}\in L^1_{\text{loc}} ( \R_t)
\big\}
\end{multline*}
and
$
\Sigma\equiv\Theta+{\mathcal H}^1.
$
Then $\Sigma$ is of full $\mu$ measure for every $\mu\in {\mathcal H}^s$, since so is $\Theta$.
We have the following proposition.
\begin{prop}\label{th.1.bis}
Assume that $s>0$ and let us fix $\mu\in {\mathcal M}^s$.
Then, for every $(v_0, v_1) \in \Sigma$, there exists a unique global solution 
$$
(v(t),\partial_t v(t)) \in (S(t)(v_0, v_1),\partial_t S(t)(v_0, v_1))+ C(\R; H^1(\T^3) \times L^2(\T^3))$$ 
of the nonlinear wave equation
\begin{equation}\label{shart}
(\partial_t^2-\Delta)v+v^3=0,\quad (v(0,x),\partial_{t} v(0,x))=(v_0(x),v_1(x))\, .
\end{equation}
Moreover for every $t\in\R$, 
$
(v(t),\partial_t v(t))\in\Sigma
$
and thus by the time reversibility $\Sigma$ is invariant under the flow of \eqref{shart}.
\end{prop}
\begin{proof}
By assumption, we can write $(v_0,v_1)=(\tilde{v}_0,\tilde{v}_1)+(w_0,w_1)$ with
$(\tilde{v}_0,\tilde{v}_1)\in\Theta$ and $(w_0,w_1)\in {\mathcal H}^1$.
We search $v$ under the form
 $$
 v(t)= S(t) (\tilde{v}_0, \tilde{v}_1) +w(t)\,.
 $$ 
 Then $w$ solves
\begin{equation*}
(\partial_t^2-\Delta_{\T^3})w+(S(t)(\tilde{v}_0, \tilde{v}_1)+w)^3=0,\quad w\mid_{t=0} =w_0, \quad \partial_t w\mid_{t=0} = w_1\,.
\end{equation*}
Now, exactly as before, we obtain that
$$
\frac{d} {dt} {\mathcal E}(w(t)) \leq 
C \big({\mathcal E}(w(t))\big)^{1/2} 
 \Big(g(t) + f(t) \big({\mathcal E}(w(t))\big)^{1/2}  \Big),
 $$
 where
$$
g(t)= \|S(t)(\tilde{v}_0, \tilde{v}_1)\|^3_{L^6( \T^3)},\quad f(t) =  \|S(t)(\tilde{v}_0, \tilde{v}_1)\|_{W^{\sigma,q}( \T^3)}.
$$
Therefore thanks to the Gronwall lemma, using that ${\mathcal E}(w(0))$ is well defined, we obtain the global existence for $w$.
Thus the solution of \eqref{shart} can be written as
 $$
 v(t)= S(t) (\tilde{v}_0, \tilde{v}_1) +w(t),\quad (w,\partial_t w)\in C(\R;{\mathcal H}^1).
 $$
Coming back to the definition of $\Theta$, we observe that 
$$
S(t)(\Theta)=\Theta.
$$
Thus $(v(t),\partial_t v(t))\in \Sigma$.
\end{proof}
This completes the proof of Theorem~\ref{main}.
\section{Unique limits of smooth solutions}
In this section, we present the proofs of Theorem~\ref{uniqueness_1} and Theorem~\ref{uniqueness_2}. 
\begin{proof}[Proof of Theorem~\ref{uniqueness_1}]
Thanks to Theorem~\ref{main}, the Sobolev embeddings and Theorem~\ref{lll}  we obtain that 
$$
(v,\partial_t  v)\in C(\R;{\mathcal H}^s(\T^3))
$$
and 
$$
v\in L^{p^\star}_{loc}(\R;L^{q^{\star}}(\T^3))\,,
$$
where $(p^\star,q^\star)$ are as in  Corollary~\ref{sun} (observe that $q^\star\leq 6$).
Once, we have this information the proof of Theorem~\ref{uniqueness_1} follows from Theorem~\ref{prop.local.low} (here we use the assumption $s>1/2$) and Corollary~\ref{sun}. 
Indeed, let us fix $T>0$ and let $\Lambda$ be such that
$$
\sup_{0\leq t\leq T}\|(v(t),\partial_t v(t))\|_{{\mathcal H}^s(\T^3)}<\Lambda-1\,.
$$
Let $\tau>0$ be the time of existence associated with $\Lambda$ in Theorem~\ref{prop.local.low}.
We now cover the interval $[0,T]$ with intervals of size $\tau$ and using iteratively the continuous dependence statement of  Theorem~\ref{prop.local.low} and the uniqueness statement given by Corollary~\ref{sun}, we obtain that 
$$
\lim_{n\rightarrow\infty}\|(v_{n}(t)-v(t),\partial_t v_{n}(t)-\partial_t v(t))\|_{L^\infty([0,T]; {\mathcal H}^s(\T^3))}=0\,.
$$
This completes the proof of  Theorem~\ref{uniqueness_1}.
\end{proof}
We now turn to the proof of Theorem~\ref{uniqueness_2} which is slightly more delicate.
\begin{proof}[Proof of Theorem~\ref{uniqueness_2}]
For $(v_0,v_1)\in \Sigma$ we decompose the solution as
$$
v(t)=S(t)(v_0,v_1)+w(t),\quad w(0)=0,\, \partial_t w(0)=0.
$$
Similarly, we decompose the solutions issued from $(v_{0,n},v_{1,n})$ as
$$
v_n(t)=S(t)(v_{0,n},v_{1,n})+w_n(t),\quad w_n(0)=0,\, \partial_t w_n(0)=0.
$$
Using the energy estimates of the previous section, we obtain that
$$
\frac{d} {dt} {\mathcal E}(w_n(t)) \leq 
C \big({\mathcal E}(w_n(t))\big)^{1/2} 
 \Big(g_n(t) + f_n(t) \big({\mathcal E}(w(t))\big)^{1/2}  \Big),
$$
where 
$$
g_n(t) = \|S(t)(v_{0,n}, v_{1,n})\|^3_{L^6( \T^3)},\quad f_n(t) =  \|S(t)(v_{0,n}, v_{1,n})\|_{W^{\sigma,q}( \T^3)}.
$$
Therefore
$$
 ({\mathcal E}(w_n(t)))^{1/2}\leq 
 C
\big( \int_{0}^t  g_n(\tau)d\tau \big)
 e^{\int_0^t f_{n}(\tau) d\tau} \,.
 $$
Using that 
\begin{equation}\label{veneta2}
S(t)(v_{0,n}, v_{1,n})=\rho_n \star \big(S(t)(v_0,v_1)\big),
\end{equation}
and the fact that $(v_0,v_1)\in \Sigma$, we obtain that
$$
\lim_{n\rightarrow\infty}
\int_0^t g_n(\tau)d\tau=\int_0^t g(\tau)d\tau,\quad 
\lim_{n\rightarrow\infty}
\int_0^t f_n(\tau)d\tau=
\int_0^t f(\tau)d\tau,
$$
where $g(t)$ and $f(t)$ are defined in \eqref{eq.borne}.
Therefore, we obtain that for every $T>0$ there is $C>0$ such that for every $n$,
\begin{equation}\label{veneta1}
\sup_{0\leq t\leq T}\|(w_n(t),\partial_t w_n(t))\|_{{\mathcal H}^1(\T^3)}\leq C.
\end{equation}
Next, we observe that $w$ and $w_n$ solve the equations
$$
(\partial_t^2-\Delta)w+(S(t)(v_0,v_1)+w)^3=0
$$
and 
$$
(\partial_t^2-\Delta)w_n+(S(t)(v_{0,n},v_{1,n})+w_n)^3=0.
$$
Therefore
$$
(\partial_t^2-\Delta)(w-w_n)=-\big((S(t)(v_0,v_1)+w)^3-S(t)(v_{0,n},v_{1,n})+w_n)^3\big).
$$
We multiply the last equation by $\partial_t(w-w_n)$, and by using the Sobolev embedding $H^1(\T^3)\subset L^6(\T^3)$  and the H\"older inequality, we arrive at the bound
\begin{multline*}
\frac{d}{dt}
\|
(w-w_n,\partial_ t w -\partial_t w_n)
\|_{{\mathcal H}^1(\T^3)}
\\
\leq 
C\big(
\|S(t)(v_0-v_{0,n}, v_1-v_{1,n})\|_{L^6(\T^3)}+\|w-w_n\|_{H^1(\T^3)}
\big)
\\
\times
\Big(
\|S(t)(v_{0}, v_{1})\|^2_{L^6(\T^3)}
+
\|S(t)(v_{0,n}, v_{1,n})\|^2_{L^6(\T^3)}
\\
+
\|w\|^2_{H^1(\T^3)}+\|w_n\|^2_{H^1(\T^3)}
\Big).
\end{multline*}
Using \eqref{veneta1} and the properties of the solutions obtained in Theorem~\ref{main}, we obtain
\begin{multline*}
\frac{d}{dt}
\|
(w-w_n,\partial_ t w -\partial_t w_n)
\|_{{\mathcal H}^1(\T^3)}
\\
\leq 
C\big(
\|S(t)(v_0-v_{0,n}, v_1-v_{1,n})\|_{L^6(\T^3)}+\|w-w_n\|_{H^1(\T^3)}
\big)
\\
\times
\Big(
\|S(t)(v_{0}, v_{1})\|^2_{L^6(\T^3)}
+
\|S(t)(v_{0,n}, v_{1,n})\|^2_{L^6(\T^3)}
+
C
\Big).
\end{multline*}
The last inequality implies the following bound for $t\in [0,T]$,
\begin{multline}\label{veneta3}
\|
(w(t)-w_n(t),\partial_ t w(t) -\partial_t w_n(t))
\|_{{\mathcal H}^1}
\\
\leq C
\int_{0}^t
\|S(\tau)(v_0-v_{0,n}, v_1-v_{1,n})\|_{L^6}
\\
\big(\|S(\tau)(v_{0}, v_{1})\|^2_{L^6}
+
\|S(\tau)(v_{0,n}, v_{1,n})\|^2_{L^6}
+
C\big)
d\tau
\\
\exp\left(
\int_{0}^t
(\|S(\tau)(v_{0}, v_{1})\|^2_{L^6}
+
\|S(\tau)(v_{0,n}, v_{1,n})\|^2_{L^6}
+
C)d\tau
\right).
\end{multline}
More precisely, we used that if $x(t)\geq 0$ satisfies the differential inequality
$$
\dot{x}(t)\leq C z(t)(y(t)+x(t)),\quad x(0)=0,
$$
for some $z(t)\geq 0$ and $y(t)\geq 0$ then
$$
x(t)\leq C\int_{0}^t y(\tau)z(\tau) d\tau 
\exp\big(\int_{0}^t z(\tau)d\tau\big)\,.
$$
Coming back to \eqref{veneta3} and using the H\"older inequality, we get for $t\in [0,T]$,
\begin{multline}\label{veneta4}
\|(w(t)-w_n(t),\partial_ t w(t) -\partial_t w_n(t))\|_{{\mathcal H}^1}
\\
\leq C
\|S(t)(v_0-v_{0,n}, v_1-v_{1,n})\|_{L^2_{T}L^6}
\\
\times\big(
\|S(t)(v_{0}, v_{1})\|^2_{L^4_T L^6}
+
\|S(t)(v_{0,n}, v_{1,n})\|^2_{L^4_TL^6}
+
C\big)
\\
\times\exp\left(
\int_{0}^t
(\|S(\tau)(v_{0}, v_{1})\|^2_{L^6}
+
\|S(\tau)(v_{0,n}, v_{1,n})\|^2_{L^6}
+
C)d\tau
\right).
\end{multline}
Recalling \eqref{veneta2}, we obtain that for $1<p<\infty$,
$$
\lim_{n\rightarrow\infty}\int_{0}^T
\|S(\tau)(v_0-v_{0,n}, v_1-v_{1,n})\|^p_{L^6(\T^3)}d\tau=0.
$$
Therefore \eqref{veneta4} implies that 
$$
\lim_{n\rightarrow\infty}
\|(w(t)-w_n(t),\partial_ t w(t) -\partial_t w_n(t))
\|_{ L^\infty([0,T]; {\mathcal H}^1(\T^3))}=0\,.
$$
Recall that 
$$
v(t)=S(t)(v_0,v_1)+w(t),\quad 
v_n(t)=S(t)(v_{0,n},v_{1,n})+w_n(t).
$$
Using once again  \eqref{veneta2}
and
$$
\partial_t S(t)(v_{0,n}, v_{1,n})=\rho_n \star \big(\partial_t S(t)(v_0,v_1)\big)
$$
we get 
\begin{multline*}
\lim_{n\rightarrow\infty}
\|
(
S(t)(v_0,v_1)-S(t)(v_{0,n},v_{1,n}),
\\
\partial_t S(t)(v_0,v_1)- \partial_t S(t)(v_{0,n},v_{1,n})
)
\|_{L^\infty([0,T];  {\mathcal H}^s(\T^3))}=0
\end{multline*}
and consequently 
$$
\lim_{n\rightarrow\infty}
\|(v(t)-v_n(t),\partial_ t v(t) -\partial_t v_n(t))
\|_{L^\infty([0,T];  {\mathcal H}^s(\T^3))}=0\,.
$$
This completes the proof of  Theorem~\ref{uniqueness_2}.
\end{proof}
\begin{rema}
{\rm
In the proof of Theorem~\ref{uniqueness_2}, we essentially used that the regularisation by convolution works equally well in $H^s$ and $L^p$ ($p<\infty$) and that it commutes with the Fourier multipliers such as the free evolution $S(t)$.
Any other regularisation respecting these two properties would produce smooth solutions converging to the singular dynamics constructed in Theorem~\ref{main}.
}
\end{rema}
\section{Conditioned large deviation bounds}
In this section, we prove conditioned large deviation bounds which are the main tool in the proof of the Theorem~\ref{th_continuity}.
\begin{prop}\label{ochak}
Let $\mu\in {\mathcal M}^s$, $s\in(0,1)$ and suppose that the real random variable with distribution $\theta$, involved in the definition of $\mu$ is
symmetric. Then for $\delta> 1+ \frac{ 1} {p_1} $, $2\leq p_1<\infty$ and $2\leq p_2\leq\infty$ there exist positive constants
$c,C$ such that for every positive $\varepsilon,\lambda,\Lambda$ and $A$,
\begin{multline}\label{eq.cond}
\mu\otimes\mu\Big(
((v_0,v_1),(v'_0,v'_1))\in {\mathcal H}^s\times {\mathcal H}^s\,:\,\
\\
\|\langle t \rangle ^{- \delta} S(t) (v_0- v'_0,v_1- v'_1)\|_{L^{p_1} (  \R_t ; L^{p_2}( \T^3))}> 
\lambda
\\
{\,\rm\,or\,}
\|\langle t \rangle ^{- \delta} S(t) (v_0+ v'_0,v_1+ v'_1)\|_{L^{p_1} (  \R_t ; L^{p_2}( \T^3))}> 
\Lambda
\Big\vert
 \|(v_0- v'_0,v_1- v'_1)\|_{\mathcal{H}^s(\T^3)}\leq \varepsilon
 \\
 {\rm\,and\,} \|(v_0+ v'_0,v_1+ v'_1)\|_{\mathcal{H}^s(\T^3)}\leq A  \Big) 
\leq 
C\Big(e^{-c\frac{\lambda^2}{\varepsilon^2}}+
e^{-c\frac{\Lambda^2}{A^2}}\Big).
\end{multline}
\end{prop}
We shall make use of the following elementary lemmas. 
\begin{lemm}\label{lem.cond}
For $j=1,2$, let $E_j$ be two Banach spaces endowed with measures $\mu_j$. Let $f: E_1\times E_2\rightarrow\C$ and $g_1,g_2:E_2\rightarrow \C$ be three measurable functions. Then
\begin{multline*}
\mu_1\otimes\mu_2
\Big(
(x_1,x_2)\in E_1\times E_2\,:\, |f(x_1,x_2)|> \lambda\Big\vert|\, g_1(x_2)|\leq \varepsilon,\,\,
\\
 |g_2(x_2)|\leq A
\Big)
\leq
\sup_{x_2\in E_2, |g_1(x_2)|\leq \varepsilon, |g_2(x_2)|\leq A }
\mu_1(x_1\in E_1\,:\, |f(x_1,x_2)|>\lambda)\,,
 \end{multline*}
 where by $\sup$ we mean the essential supremum.
 \end{lemm} 
 \begin{lemm}\label{ll2}
 Let $g_1$ and $g_2$ be two independent identically distributed real random variables with symmetric distribution. Then  $g_1\pm g_2$ have  symmetric distributions.
 Moreover if $h$ is a Bernoulli random variable independent of $g_1$ then $hg_1$ has the same distribution as $g_1$.
 \end{lemm} 
\begin{proof}[Proof of Proposition~\ref{ochak}.]
 Define
 $$
 {\mathcal E}\equiv \R\times \R^{\Z^3_{\star}}\times \R^{\Z^3_{\star}},
 $$ 
 equipped with the natural Banach space structure coming from the $l^\infty$ norm. 
 We endow ${\mathcal E}$ with a probability measure $\mu_0$ defined via the map
 $$ 
 \omega\mapsto 
 \Big(
 k_{0}(\omega), \big(l_{n}(\omega)\big)_{n\in\Z^3_{\star}},
 \big(h_{n}(\omega)\big)_{n\in\Z^3_{\star}}
 \Big) ,
 $$
where $(k_0,l_{n}, h_n)$ is a system of independent Bernoulli variables. 
 
For $h=\big(x,(y_n)_{n\in\Z^3_{\star}},(z_n)_{n\in\Z^3_{\star}}\big)\in{\mathcal E}$ and 
$$
u(x)=a+\sum_{n\in\Z^3_{\star}}\Big(b_{n}\cos(n\cdot x)+c_{n}\sin(n\cdot x)\Big),
$$ 
we define the operation $\odot$ by
$$
h\odot u\equiv
ax+\sum_{n\in\Z^3_{\star}}\Big(b_{n} y_{n}\cos(n\cdot x)+c_{n}z_{n}\sin(n\cdot x)\Big).
$$
Let us first evaluate the quantity
\begin{multline}\label{ven_1}
\mu\otimes\mu\Big(
((v_0,v_1),(v'_0,v'_1))\in {\mathcal H}^s\times {\mathcal H}^s\,:\,\
\\
\|\langle t \rangle ^{- \delta} S(t) (v_0- v'_0,v_1- v'_1)\|_{L^{p_1} (  \R_t ; L^{p_2}( \T^3))}> 
\lambda
\Big\vert
\\
 \|(v_0- v'_0,v_1- v'_1)\|_{\mathcal{H}^s(\T^3)}\leq \varepsilon
 {\rm\,and\,} \|(v_0+ v'_0,v_1+ v'_1)\|_{\mathcal{H}^s(\T^3)}\leq A  \Big).
\end{multline}
Observe that, thanks to Lemma~\ref{ll2}, \eqref{ven_1} equals
\begin{multline}\label{ven_2}
\mu\otimes\mu\otimes\mu_0\otimes\mu_0\Big(
((v_0,v_1),(v'_0,v'_1), (h_0,h_1))\in {\mathcal H}^s\times {\mathcal H}^s\times {\mathcal E}\times {\mathcal E}\,:\,\
\\
\|\langle t \rangle ^{- \delta} S(t) (h_0\odot (v_0-v'_0),h_1\odot (v_1- v'_1))\|_{L^{p_1} (  \R_t ; L^{p_2}( \T^3))}> 
\lambda
\Big\vert
\\
 \|(h_0\odot(v_0- v'_0),h_1\odot(v_1- v'_1))\|_{\mathcal{H}^s(\T^3)}\leq \varepsilon
 {\rm\,and\,} 
 \\
 \|(h_0\odot(v_0+ v'_0),h_1\odot(v_1+ v'_1))\|_{\mathcal{H}^s(\T^3)}\leq A  \Big). 
\end{multline}
Since the $H^s(\T^3)$ norm of a function $f$ depends only on the absolute value of its Fourier coefficients, we deduce that \eqref{ven_2} equals
\begin{multline}\label{ven_3}
\mu\otimes\mu\otimes\mu_0\otimes\mu_0\Big(
((v_0,v_1),(v'_0,v'_1), (h_0,h_1))\in {\mathcal H}^s\times {\mathcal H}^s\times {\mathcal E}\times {\mathcal E}\,:\,\
\\
\|\langle t \rangle ^{- \delta} S(t) (h_0\odot (v_0-v'_0),h_1\odot (v_1- v'_1))\|_{L^{p_1} (  \R_t ; L^{p_2}( \T^3))}> 
\lambda
\Big\vert
\\
 \|(v_0- v'_0,v_1- v'_1)\|_{\mathcal{H}^s(\T^3)}\leq \varepsilon
 {\rm\,and\,} \|(v_0+ v'_0,v_1+ v'_1)\|_{\mathcal{H}^s(\T^3)}\leq A  \Big).
\end{multline}
We now apply Lemma~\ref{lem.cond}  with $\mu_1=\mu_0\otimes\mu_0$ and $\mu_2=\mu\otimes\mu$
to get that \eqref{ven_3} is bounded by
\begin{multline}\label{ven_4}
\sup_{ \|(v_0- v'_0,v_1- v'_1)\|_{\mathcal{H}^s(\T^3)}\leq \varepsilon}
\mu_0\otimes\mu_0\Big(
 (h_0,h_1)\in  {\mathcal E}\times {\mathcal E}\,:\,\
\\
\|\langle t \rangle ^{- \delta} S(t) (h_0\odot (v_0-v'_0),h_1\odot (v_1- v'_1))\|_{L^{p_1} (  \R_t ; L^{p_2}( \T^3))}> 
\lambda \Big).
\end{multline}
We now apply Theorem~\ref{lll} (with Bernoulli variables) to obtain that \eqref{ven_1} is bounded by
$
C\exp(-c\frac{\lambda^2}{\varepsilon^2}).
$
A very similar argument gives that
\begin{multline*}
\mu\otimes\mu\Big(
((v_0,v_1),(v'_0,v'_1))\in {\mathcal H}^s\times {\mathcal H}^s\,:\,\
\\
\|\langle t \rangle ^{- \delta} S(t) (v_0+ v'_0,v_1+ v'_1)\|_{L^{p_1} (  \R_t ; L^{p_2}( \T^3))}> 
\Lambda
\Big\vert
\\
 \|(v_0- v'_0,v_1- v'_1)\|_{\mathcal{H}^s(\T^3)}\leq \varepsilon
 {\rm\,and\,} \|(v_0+ v'_0,v_1+ v'_1)\|_{\mathcal{H}^s(\T^3)}\leq A  \Big) 
\end{multline*}
is bounded by 
$
C\exp(-c\frac{\Lambda^2}{A^2}).
$
This completes the proof of Proposition~\ref{ochak}.
\end{proof}
\section{End of the proof of the conditioned continuous dependence}
In this section, we complete  the proof of Theorem~\ref{th_continuity}.
According to (a variant of) Proposition~\ref{ochak}, we have that for any 
$$
2\leq p_1<+\infty,\,
 2\leq p_2 \leq + \infty,\,
  \delta > 1+ \frac 1 {p_1},\,  
  \eta\in(0,1),\, 
$$
one has 
\begin{multline*}
\mu\otimes\mu\Big(
(V_0,V_1)\in {\mathcal H}^s\times {\mathcal H}^s\,:\,\
\|\langle t \rangle ^{- \delta} S(t) (V_0- V_1)\|_{L^{p_1} (  \R_t ; L^{p_2}( \T^3))}> 
\eta^{\frac{1}{2}}
\\
{\,\rm\,or\,}
\|\langle t \rangle ^{- \delta} S(t) (V_0 )\|_{L^{p_1} (  \R_t ; L^{p_2}( \T^3))}> 
\log \log\log(\eta^{-1})
\\
{\,\rm\,or\,}
\|\langle t \rangle ^{- \delta} S(t) (V_1 )\|_{L^{p_1} (  \R_t ; L^{p_2}( \T^3))}> 
\log \log\log(\eta^{-1})
\Big\vert
\\
 \|V_0- V_1\|_{\mathcal{H}^s(\T^3)}< \eta
 {\rm\,\,\,and\,\,\,} \|V_j\|_{\mathcal{H}^s(\T^3)}\leq A,\,j=0,1  \Big) 
\longrightarrow 0,
\end{multline*}
as $\eta\rightarrow 0$. Therefore, we can also suppose that 
\begin{equation}\label{yu1}
\|\langle t \rangle ^{- \delta} S(t) (V_0- V_1)\|_{L^{p_1} (  \R_t ; L^{p_2}( \T^3))}\leq
\eta^{\frac{1}{2}}
\end{equation}
and
\begin{equation}\label{yu2}
\|\langle t \rangle ^{- \delta} S(t) (V_j )\|_{L^{p_1} (  \R_t ; L^{p_2}( \T^3))}\leq
\log \log\log(\eta^{-1}),\, j=0,1,
\end{equation}
when we estimate the needed conditional probability.

We therefore need to estimate the difference of two solutions under the assumptions \eqref{yu1} and \eqref{yu2}, in the regime $\eta\ll 1$.
Let 
$$
v_{j}(t)=S(t)(V_j)+w_{j}(t), \quad j=0,1
$$
be two solutions of the cubic wave equation with data $V_j$. We thus have
$$
(w_{j}(0),\partial_{t}w_{j}(0))=(0,0).
$$
Applying the energy estimate, performed several times in this chapter,  for $j=0,1$, we get the bound
$$
\frac{d}{dt}\mathcal{E}^{1/2}(w_{j}(t))\leq
C\Big(
\|S(t)(V_j)\|^3_{L^{6}(\T^3)}
+
\|S(t)(V_j)\|_{L^{\infty}(\T^3)}\mathcal{E}^{1/2}(w_{j}(t))
\Big),
$$
and therefore, under the assumptions \eqref{yu1} and \eqref{yu2}, for $t\in [0,T]$ one has
\begin{eqnarray}\label{kam}
\mathcal{E}^{1/2}(w_{j}(t))
& \leq & 
C_{T}\,e^{C_T\log \log\log(\eta^{-1})}
(\log\log\log(\eta^{-1}))^3
\\
\nonumber
& \leq &
C_{T}[\log(\eta^{-1})]^{\frac{1}{20}},
\end{eqnarray}
where here and in the sequel we denote by $C_T$ different constants depending only on $T$ (but independent of $\eta$).
\par
We next estimate the difference $w_0-w_1$. Using the equations solved by $w_0$, $w_1$, we infer that
\begin{multline}\label{veneta7}
\frac{d}{dt}\|w_0(t,\cdot)-w_1(t,\cdot)\|^2_{{\mathcal H}^1(\T^3)}
\\
 \leq  2
\Big|
\int_{\T^3}\partial_{t}(w_0(t,x)-w_1(t,x))(\partial_t^2-\Delta)(w_0(t,x)-w_1(t,x))dx
\Big|
\\
 \leq 
C
\|w_0(t,\cdot)-w_1(t,\cdot)\|_{{\mathcal H}^1(\T^3)}
\\
\|(w_0+S(t)(V_0))^3-(w_1+S(t)(V_1))^3\|_{L^2(\T^3)}\,,
\end{multline}
where for shortness we denote $\|(u, \partial_t u)\|_{{\mathcal H}^1}$ simply by $\|u\|_{{\mathcal H}^1}$.

Thanks to \eqref{veneta7} and the Sobolev embedding $H^1(\T^3)\subset L^6(\T^3)$, we get that
$$
\frac{d}{dt}\|w_0(t,\cdot)-w_1(t,\cdot)\|_{{\mathcal H}^1(\T^3)}
$$
is bounded by
\begin{multline*}
C\Big(\|w_0(t,\cdot)-w_1(t,\cdot)\|_{{\mathcal H}^1(\T^3)}+\|S(t)(V_0-V_1)\|_{L^6(\T^3)}\Big)
\\
\Big(\|w_0(t,\cdot)\|_{H^1(\T^3)}^2+\|w_1(t,\cdot)\|_{H^1(\T^3)}^2
\\
+\|S(t)(V_0)\|_{L^6(\T^3)}^2+\|S(t)(V_1)\|_{L^6(\T^3)}^2\Big).
\end{multline*}
Therefore, using \eqref{kam} and the Gronwall lemma, under the assumptions \eqref{yu1} and \eqref{yu2}, for $t\in [0,T]$,
\begin{eqnarray*}
\|w_0(t,\cdot)-w_1(t,\cdot)\|_{{\mathcal H}^1(\T^3)} & \leq & C_{T}\eta^{\frac{1}{2}}[\log(\eta^{-1})]^{\frac{1}{10}}\,
e^{C_{T}[\log(\eta^{-1})]^{\frac{1}{10}}}
\\
& \leq &
C_{T}\eta^{\frac{1}{4}}\,.
\end{eqnarray*}
In particular by the Sobolev embedding
$$
\|w_0-w_1\|_{L^4([0,T]\times\T^3)}\leq C_{T}\eta^{\frac{1}{4}},
$$
and therefore under the assumption \eqref{yu1},
$$
\|v_0-v_1\|_{L^4([0,T]\times\T^3)}\leq C_{T}\eta^{\frac{1}{4}}\,.
$$
In summary, we obtained that for a fixed $\varepsilon>0$, the $\mu\otimes\mu$ measure of $V_0$, $V_1$ such that
$$
\|\Phi(t)(V_0)-\Phi(t)(V_1)\|_{X_{T}}>\varepsilon
$$
under the conditions \eqref{yu1}, \eqref{yu2}  and $\|V_0-V_1\|_{{\mathcal H}^s}<\eta$  is zero, as far as $\eta>0$ is sufficiently small.
Therefore, we obtain that the left hand side of \eqref{dimanche} tends to zero as $\eta\rightarrow 0$.
This ends the proof of the first part of Theorem~\ref{th_continuity}.
\par
For the second part of the proof of Theorem~\ref{th_continuity}, we argue by contradiction. 
Suppose thus that for every $\varepsilon>0$ there exist $\eta>0$ and $\Sigma$ of full $\mu\otimes\mu$ measure such that 
\begin{multline*}
\forall\, (V,V')\in\Sigma\cap (B_{A}\times B_{A}),\, \|V-V'\|_{{\mathcal H}^s}<\eta
\implies\,\,
\\
 \| \Phi(t) ( V) - \Phi(t) ( V') \|_{X_T} <\varepsilon.  
\end{multline*}
Let us apply the previous affirmation with $\varepsilon=1/n$, $n=1,2,3\dots$ which produces full measure sets $\Sigma(n)$. 
Set
$$
\Sigma_1\equiv\bigcap_{n=1}^{\infty}\Sigma(n).
$$
Then $\Sigma_1$ is of full $\mu\otimes\mu$ measure and we have that
\begin{multline}\label{eq.unif-cont}
\forall\,\varepsilon>0,\,
\exists\,\eta>0,\,\,\,
\forall\, (V,V')\in\Sigma_1\cap (B_{A}\times B_{A}),\, 
\\
\|V-V'\|_{{\mathcal H}^s}<\eta
\implies\,\,
 \| \Phi(t) (V) - \Phi(t) ( V') \|_{X_T} <\varepsilon.  
\end{multline}
Next for $V\in {\mathcal H}^s$ we define ${\mathcal A}(V)\subset {\mathcal H}^s$  by
$$
{\mathcal A}(V)\equiv \{V'\in {\mathcal H}^s\,:\, (V,V')\in \Sigma_1\}.
$$
According to Fubini Theorem, 
there exists   ${\mathcal E}\subset {\mathcal H}^s$  a set of full $\mu$ measure such that for every $V\in {\mathcal E}$ the set ${\mathcal A}(V)$ is a full $\mu$ measure.

We are going to extend $\Phi(t)$ to a uniformly continuous map on $B_A$.
For that purpose, we first extend  $\Phi(t)$ to a uniformly continuous map on dense set of $B_A$.
Let $\{(V_j)_{j\in \mathbb{N}}\}$ be a dense set of  $B_A$ for the $\mathcal{H}^s$ topology.
For $j\in \mathbb{N}$, we can construct by induction  a sequence $(V_{j,n})$ such that 
\begin{equation*}
V_{j,n}\in B_{A}\cap {\mathcal E}\cap\, \bigcap_{m<n}{\mathcal A}(V_{j,m})
\cap\,  \bigcap_{l<j, q\in\N}{\mathcal A}(V_{l,q})
,\quad
\|V_{j,n}-V_j\|_{{\mathcal H}^s}<1/n.
\end{equation*}
Indeed, 
the induction assumption guarantees that  the set 
$$
{\mathcal E}\cap\, \bigcap_{m<n}{\mathcal A}(V_{j,m})
 \bigcap_{l<j, q\in \N}{\mathcal A}(V_{l,q})
$$ 
has measure $1$ (as an intersection of sets of measure $1$) and consequently is dense.  
Notice that by construction, we have 
\begin{equation}
\label{eq.unif}
 (V_{k,n}, V_{l,m}) \in \Sigma_1, \forall\, k< l, \forall \,n,m \in \mathbb{N}, \text{ and } \forall\, k=l, n<m.
\end{equation}
 Using~\eqref{eq.unif} for $k=l$, we obtain according to~\eqref{eq.unif-cont} that  for any fixed $k$, the sequence $\Phi(t)(V_{k,n})_{n\in \mathbb{N}}$ is a Cauchy sequence in $X_T$ and we can can define $\overline{\Phi}(t)(V_j)$ as its limit. Using again~\eqref{eq.unif}, for $k\neq l$, we see according to~\eqref{eq.unif-cont} that the map $\overline{\Phi}(t) $ is uniformly continuous on the set  $\{(V_j)_{j\in \mathbb{N}}\}$. 
Therefore $\Phi(t)$ can be extended by density to a 
uniformly continuous map, on the whole $B_A$.
Let us denote by $\overline{\Phi(t)}$ the extension of $\Phi(t)$ to $B_A$. 
 We therefore have
\begin{multline}\label{rex}
\forall\,\varepsilon>0,\,
\exists\,\eta>0,\,\,\,
\forall\, V,V'\in
B_{A},\, 
\\
\|V-V'\|_{{\mathcal H}^s}<\eta
\implies\,\,
\| \overline{\Phi(t)} ( V) - \overline{\Phi(t)} ( V') \|_{X_T} <\varepsilon.  
\end{multline}
We have the following lemma.
\begin{lemm}\label{rex1}
For $V\in (C^{\infty}(\T^3)\times C^{\infty}(\T^3))\cap B_A$, we have that $\overline{\Phi(t)}(V)=(u,u_t)$, where $u$ is the unique classical solution on $[0,T]$ of
$$
(\partial_{t}^2-\Delta)u+u^3=0,\quad (u(0),\partial_{t}u(0))=V. 
$$
\end{lemm}
\begin{proof}
Let us first show that that first component of
$$
\overline{\Phi(t)}(V)\equiv(\overline{\Phi_1(t)}(V),\overline{\Phi_2(t)}(V))
$$ 
is a solution of the cubic wave equation.
Observe that by construction, necessarily $\overline{\Phi_2(t)}(V)=\partial_t \overline{\Phi_1(t)}(V)$ in the distributional sense (in ${\mathcal D}'((0,T)\times\T^3)$).

Again by construction, we have that 
$$
V=\lim_{n\rightarrow\infty}V_{n}\,,
$$
in ${\mathcal H}^s$ where  $V_n$ are such that
\begin{equation}\label{limit}
(\partial_{t}^2-\Delta)(\Phi_1(t)(V_n))+(\Phi_1(t)(V_n))^3=0, 
\end{equation}
with the notation $\Phi(t)=(\Phi_1(t),\Phi_2(t))$. In addition, 
$$
\overline{\Phi(t)}(V)=\lim_{n\rightarrow\infty}\Phi(t)(V_{n})\,,
$$
in $X_T$.
We therefore have that
$$
(\partial_{t}^2-\Delta)(\overline{\Phi_1(t)}(V))=\lim_{n\rightarrow\infty}(\partial_{t}^2-\Delta)(\Phi_1(t)(V_n)),
$$
in the distributional sense. Moreover, coming back to the definition of $X_T$, we also obtain that
$$
(\overline{\Phi_1(t)}(V))^3=\lim_{n\rightarrow\infty}(\Phi_1(t)(V_n))^3,
$$
in $L^{4/3}([0,T]\times\T^3)$.
Therefore, passing into the limit $n\rightarrow\infty$ in (\eqref{limit}), we obtain that $\overline{\Phi_1(t)}(V)$ solves the cubic wave equation (with data $V$).
Moreover, since $(\overline{\Phi_1(t)}(V))^3\in L^{4/3}([0,T]\times\T^3)$, it also satisfies the Duhamel formulation of the equation.
\par
Let us denote by $u(t)$, $t\in [0,T]$ the classical solution of 
$$
(\partial_{t}^2-\Delta)u+u^3=0,\quad (u(0),\partial_{t}u(0))=V, 
$$
defined by Theorem~\ref{prop.global}.  Set $v\equiv \overline{\Phi_1(t)}(V)$. Since our previous analysis has shown that $v$ is a solution of the cubic wave equation, we have that 
\begin{equation}\label{razlika}
(\partial_{t}^2-\Delta)(u-v)+u^3-v^3=0,\quad (u(0),\partial_{t}u(0))=(0,0)\,.
\end{equation}
We now invoke the $L^4-L^{4/3}$ non homogenous estimates for the three dimensional wave equation.
Namely,  thanks to Theorem~\ref{cor.stri}, we have that there exists a constant (depending on $T$)
such that for every interval $I\subset [0,T]$, the solution of the wave equation
$$
(\partial_{t}^2-\Delta)w=F,\quad (u(0),\partial_{t}u(0))=(0,0)
$$
satisfies
\begin{equation}\label{strichartz}
\|w\|_{L^4(I\times\T^3)}\leq C\|F\|_{L^{4/3}(I\times\T^3)}\,.
\end{equation}
Applying \eqref{strichartz} in the context of \eqref{razlika} together with the H\"older inequality yields the bound
\begin{equation}\label{strichartz2}
\|u-v\|_{L^4(I\times\T^3)}\leq C
\big(\|u\|_{L^{4}(I\times\T^3)}^2+\|v\|_{L^4(I\times\T^3)}^2\big)
\|u-v\|_{L^{4}(I\times\T^3)}\,.
\end{equation}
Since $u,v\in L^{4}(I\times\T^3)$, we can find a partition of intervals $I_1,\dots,I_{l}$ of $[0,T]$ such that
$$
C\big(\|u\|_{L^{4}(I_j\times\T^3)}^2+\|v\|_{L^4(I_j\times\T^3)}^2\big)<\frac{1}{2},\quad j=1,\dots, l.
$$
We now apply \eqref{strichartz2} with $I=I_j$, $j=1,\dots, l$ to conclude that $u=v$ on $I_1$, then on $I_2$ and so on up to $I_l$ which gives that $u=v$ on 
$[0,T]$. 
Thus $u=\overline{\Phi_1(t)}(V)$ and therefore also $\partial_t u=\overline{\Phi_2(t)}(V)$.
This completes the proof of Lemma~\ref{rex1}.
\end{proof}
It remains now to apply Lemma~\ref{rex1} to the sequence of smooth data in the statement of Theorem~\ref{ill} to get a contradiction with \eqref{rex}.
More precisely, if $(U_n)$ is the sequence involved in the statement of Theorem~\ref{ill}, the result of Theorem~\ref{ill} affirms that 
$\overline{\Phi(t)}(U_n)$ tends to infinity in $L^{\infty}([0,T];{\mathcal H}^s)$ while \eqref{rex} affirms that the same sequence tends to zero in the same space 
$L^{\infty}([0,T];{\mathcal H}^s)$. 
This completes the proof of Theorem~\ref{th_continuity}.
\section{Extensions to more general nonlinearities}
In the remarkable work by Oh-Pocovnicu \cite{OhP1,OhP2} (based on the previous contributions \cite{BOP, Poc}) it is shown that the result of Theorem~\ref{main} can be extended to the energy critical equation
$$
(\partial_t^2-\Delta)v+v^5=0\,.
$$
This equation is $H^1$ critical and the data is a typical element with respect to $\mu \in {\mathcal M}^s$, $s>1/2$.
We refer also to \cite{LM,xia_sun} for extensions of  Theorem~\ref{main}  to nonlinearities between cubic and quintic. 
\section{Notes} 
For the case $s=0$ and the proof of the quantitative  bounds displayed by Theorem~\ref{bornes}, we refer to  \cite{BT-JEMS}. 
For the proof of Proposition~\ref{prop.1.2}, we refer to \cite[Appendix~B]{BT1/2} and  \cite[Appenidix~B2]{BT-JEMS}.
The probabilistic part of our analysis only relies on linear bounds such as  Lemma~\ref{lem1}. In other situations multi-linear versions of these bounds are of importance (see \cite{B2,COh, NS}).
The above mentioned work by Oh-Pocovnicu relies on a much more complicated deterministic analysis (such as the concentration compactness) and also on a significant extension of the probabilistic energy bound used in the proof of 
Theorem~\ref{main}. 

Our starting point and main motivation toward the probabilistic well-posedness results presented in this chapter was the ill-posedness result of Theorem~\ref{ill} of the previous chapter. 
As already mentioned the method of proof has some similarities with the earlier work \cite{DD} or  with the even earlier work of Bourgain \cite{B2}
on the invariance of the Gibbs measure associated with the nonlinear Schr\"odinger equation
\begin{equation}\label{NLS}
(i\partial_t+\Delta)u=|u|^2 u,
\end{equation}
posed on the two dimensional torus. 
The main purpose of \cite{B2} is to show the invariance of the Gibbs measure and as a byproduct one gets the global existence and uniqueness of solutions of 
\eqref{NLS} with a suitable random data belonging a.s. to $H^{-\varepsilon}(\T^2)$ for every $\varepsilon>0$ but missing a.s. $L^2(\T^2)$. 
In  the time of writing of \cite{B2} statement such as  Theorem ~\ref{ill} or Theorem~\ref{ill_bis} were not known in the context of \eqref{NLS}.
In the recent work \cite{Oh_new}, the analogue of  Theorem ~\ref{ill} and Theorem~\ref{ill_bis} in the context of \eqref{NLS} is obtained. 
Most likely, the analysis of \cite{B2} can be adapted in order to get the analogue of Theorem~\ref{uniqueness_2} in the context of \eqref{NLS}.
As a consequence, it looks that we can see from the same view point \eqref{NLS}  with data on the support of the Gibbs measure and the cubic defocusing wave equation with random data of super-critical regularity presented in these lectures. We plan to address this issue in a future work. 
\chapter
[Globalisation via invariant measures]
{Random data global well-posedness with data of supercritical regularity via invariant measures}\label{Chapter3}
In the previous chapter, we presented a method to construct global in time solutions for the cubic defocusing wave equation posed on the three dimensional torus with random data of supercritical regularity ($H^s(\T^3)\times H^{s-1}(\T^3)$,  $s\in (0,1/2)$).
These solutions are unique in a suitable sense and depend continuously (in a conditional sense) on the initial data. 
The method we used is based on a local in time result showing that even if the data is of supercritical regularity, we can 
find a local solution written as "free evolution" (keeping the Sobolev regularity of the initial data) plus "a remainder of higher regularity". The term of higher regularity is then regular enough to allow us to deal it with the deterministic methods to treat the equation. 
The globalisation was then done by establishing an energy bound for the remainder in a probabilistic manner, here of course the energy conservation law is the key structure allowing to perform the analysis.
Moreover, we have shown that the problem is ill-posed with data of this supercritical regularity and this in turn implied the impossibility to see the constructed flow as the unique extension of the regular solutions flow. 
\\

In this chapter, we will show another method for global in time solutions for a defocusing wave equation with data of supercritical regularity. The construction of local solutions will be based on the same principle as in the previous chapter, i.e. we shall again see the solution as a  "free evolution"  plus "a remainder of higher regularity". However the globalisation will be done by a different argument (due to Bourgain \cite{B,B2}) based on
exploiting the invariance of the Gibbs measure associated with the equation. The Gibbs measure is constructed starting from the energy conservation law and therefore this energy conservation law is again the key structure allowing to perform the global in time analysis. This method of globalisation by invariant measures is working only for very particular choice of the initial data and in this sense it is much less general than the method presented in the previous chapter. On the other hand the method based on exploiting invariant measures gives a strong macroscopic information about the constructed flow, 
namely one has a precise information on the measure evolution along the time.
The method presented in the previous chapter gives essentially no information about the evolution in time under the constructed flow of the measures in ${\mathcal M}^s$. We shall come back to this issue in the next chapter. 
\\

Our model to present the method of globalisation via invariant measures will be the radial nonlinear wave equation posed on the unit ball of $\R^3$, with Dirichlet boundary conditions. 
Let $\Theta$ be the unit ball of $\R^3$. Consider the nonlinear wave equation with Dirichlet boundary condition  posed on $\Theta$,
\begin{equation}\label{6.1}
(\partial_{t}^{2}-\mathbf{\Delta})w+|w|^{\alpha}w=0,\quad (w,\partial_t w)|_{t=0}=(f_1,f_2), \qquad \alpha>0,
\end{equation}
subject to Dirichlet boundary conditions 
$$
u\mid_{\R_t \times \partial \Theta} =0,
$$
with radial real valued initial data $(f_1,f_2)$. 
\\

We now make some algebraic manipulations on (\ref{6.1}) allowing 
to write it as a first order equation in $t$. Set $u\equiv w+i{\sqrt{ - \mathbf{\Delta}}^{-1}}{\partial_t w}$. 
Observe that $ \mathbf{\Delta}^{-1}$ is well-defined because $0$ is not in the spectrum of the Dirichlet Laplacian. 
Then we have that $u$ solves the equation  
\begin{equation}\label{2}
(i\partial_{t}-\sqrt{ - \mathbf{\Delta}})u-\sqrt{ - \mathbf{\Delta}}^{-1}\big(|\Re(u)|^{\alpha}\Re(u)\big)
=0,\quad u|_{t=0}=u_{0},
\end{equation}
with $u|_{\R\times\partial\Theta}=0$,
where $u_0=f_1+i\sqrt{ - \mathbf{\Delta}}^{-1}f_2$.
We consider \eqref{2} for data in the (complex) Sobolev spaces $H^s_{rad}(\Theta)$ of radial functions. 

Equation~\eqref{2} is (formally) an Hamiltonian equation on $L^2(\Theta)$ with Hamiltonian, 
\begin{equation}\label{HHH}
\frac{1}{2}\|\sqrt{ - \mathbf{\Delta}} (u)\|_{L^2(\Theta)}^{2}+
\frac{1}{\alpha+2}\|\Re(u)\|_{L^{\alpha+2}(\Theta)}^{\alpha+2}
\end{equation}
which is (formally) conserved by the flow of (\ref{2}).
\\

Let us next discuss the measure describing the initial data set. 
For  $s<1/2$, we define the measure $\mu$ on $H^{s}_{rad}(\Theta)$ as the image measure under the map from a probability space
$(\Omega,{\mathcal A},p)$ to $H^{s}_{rad}(\Theta)$ equipped with the Borel sigma algebra, defined by
\begin{equation}\label{mapp}
\omega\longmapsto \sum_{n=1}^{\infty}\frac{ h_{n}(\omega)+i l_{n}(\omega)  }{n\pi}e_{n}\,,
\end{equation}
where $((h_{n},l_{n}))_{n=1}^{\infty}$ is a sequence of independent standard real Gaussian random variables. 
In \eqref{mapp}, the functions $(e_n)_{n=1}^{\infty}$ are the radial eigenfunctions of the Dirichlet Laplacian on $\Theta$,  associated with eigenvalues $(\pi n)^2.$
The eigenfunctions $e_n$ have the following explicit form
$$
e_n(r)=\frac{\sin(n\pi r)}{r},\quad 0\leq r\leq 1.
$$
They are the analogues of $\cos(n\cdot x)$ and $\sin(n\cdot x)$, $n\in\Z^3$ used in the analysis on $\T^3$ in the previous section.  
One has that $\mu(H^{1/2}_{rad}(\Theta))=0$.
By the method described in the previous chapter one may show that \eqref{2} is ill-posed in $H^s_{rad}(\Theta)$ for $s<\frac{3}{2}-\frac{2}{\alpha}$. 
Therefore for $\alpha>2$ the map \eqref{mapp} describes functions of supercritical Sobolev regularity (i.e. $H^s_{rad}(\Theta)$ with $s$ smaller than  $\frac{3}{2}-\frac{2}{\alpha}$). 
The situation is therefore similar to the analysis of the cubic defocusing wave equation on $\T^3$ with data in $H^s\times H^{s-1}$, $s<1/2$ considered in the previous chapter. 
As in the previous chapter, we can still get global existence and uniqueness for \eqref{2}, almost surely with respect to $\mu$.
\begin{theo}\label{thh1}
Let $s<1/2$.
Suppose that $\alpha<3$. Let us fix a real number $p$ such that $\max(4,2\alpha)<p<6$.
Then there exists a full $\mu$ measure set $\Sigma \subset  H^s_{rad}(\Theta)$ such that 
for every $u_0\in \Sigma$ there exists a unique global solution of \eqref{2}
$$
u\in C(\R,H^s_{rad}(\Theta))\cap L^p_{loc}( \mathbb{R}_t;L^p(\Theta))\,.
$$
The solution can be written as 
$$
u(t)=S(t)(u_0)+v(t),
$$
where  $S(t) = e^{-it\sqrt{ -\mathbf{\Delta}}}$ is the free evolution and $v(t)\in H^\sigma_{rad}(\Theta)$ for some $\sigma>1/2$.
Moreover
$$
\|u(t)\|_{H^s(\Theta)} \leq 
C(s) \big(\log(2+|t|)\big)^{\frac 1 2}\,.
$$
\end{theo}
The proof of Theorem~\ref{thh1} is based on the following local existence result.
\begin{prop}\label{lwp}
For a given positive number $\alpha<3$ we choose a real number $p$ such that $\max(4,2\alpha)<p<6$.
Then we fix a real number $\sigma$ by $\sigma=\frac{3}{2}-\frac{4}{p}$.
There exist $C>0$, $c\in (0,1]$, $\gamma>0$ such that for every $R\geq 1$ if we set $T=c R^{-\gamma}$ 
then for every radially symmetric $u_0$ 
satisfying 
$$
\|S(t)u_0\|_{L^p((0,2)\times\Theta)}\leq R
$$ 
there exists a unique solution $u$ of \eqref{2} such that
$$
u(t)=S(t)u_0+v(t)
$$ 
with $v\in X^{\sigma}_{T}$ (the Strichartz spaces defined in the previous chapter). 
Moreover 
$$
\|v\|_{X^\sigma_T}\leq CR.
$$ 
In particular, since $S(t)$ is $2$ periodic and thanks to the Strichartz estimates,
$$
\sup_{t\in[-T,T]}\|S(\tau)u(t)\|_{L^p(\tau\in (0,2);L^p(\Theta))}\leq CR\,.
$$
In addition, if $u_0\in H^s(\Theta)$ (and thus $s<\sigma$) then
$$
\|u(t)\|_{H^s(\Theta)}\leq
\|S(t)u_0\|_{H^s(\Theta)}+\|v(t)\|_{H^s(\Theta)}\leq
\|u_0\|_{H^s(\Theta)}+CR\,.
$$
\end{prop}
Using probabilistic Strichartz estimates for $S(t)$ as we did in the previous chapter, we can deduce the following corollary of Proposition~\ref{lwp}.
\begin{prop}\label{lwp_prop}
Under the assumptions of Proposition~\ref{lwp}  there is a set $\Sigma$ of full $\mu$ measure such that for every $u_0\in \Sigma$ there is $T>0$ and a unique solution of \eqref{2} on $[0,T]$ in 
$$
C([0,T],H^s_{rad}(\Theta))\cap L^p_{loc}( \mathbb{R}_t;L^p(\Theta)).
$$
Moreover for every $T\leq 1$ there is a set $\Sigma_T\subset\Sigma$ such that
$$
\mu(\Sigma_T)\geq 1-Ce^{-c/T^\delta},\quad c>0,\, \delta>0
$$
and such that for every $u_0\in \Sigma_T$ the time of existence is at least $T$.
\end{prop}
Let us  next define the Gibbs measures associated with \eqref{2}. 
Using \cite[Theorem~4]{AT}, we have that for $\alpha<4$ the quantity
\begin{equation}\label{antoine}
\|\sum_{n=1}^{\infty}\frac{ h_{n}(\omega)+i l_{n}(\omega)  }{n\pi}e_{n}\|_{L^{\alpha+2}(\Theta)}
\end{equation}
is finite almost surely. 
Moreover the restriction $\alpha<4$ is optimal because for $\alpha=4$ the quantity \eqref{antoine} is infinite almost surely.
Therefore, for $\alpha<4$, we can define a nontrivial measure $\rho$ as the image measure on
$H^s_{rad}(\Theta)$ by the map (\ref{mapp}) of the measure
\begin{equation}\label{density}
\exp\Big(-\frac{1}{\alpha+2}
\|\sum_{n=1}^{\infty}\frac{ h_{n}(\omega)}{n\pi}e_{n})\|_{L^{\alpha+2}(\Theta)}^{\alpha+2}\Big)dp(\omega)\,.
\end{equation}
The measure $\rho$ is the Gibbs measures associated with \eqref{2} and it can be formally seen as 
\begin{equation*}
\exp\Big(
-\frac{1}{2}\|\sqrt{ - \mathbf{\Delta}} (u)\|_{L^2(\Theta)}^{2}-
\frac{1}{\alpha+2}\|\Re(u)\|_{L^{\alpha+2}(\Theta)}^{\alpha+2}
\Big)du,
\end{equation*}
where a renormalisation of 
$$
\exp\Big(
-\frac{1}{2}\|\sqrt{ - \mathbf{\Delta}} (u)\|_{L^2(\Theta)}^{2}
\Big)du,
$$
corresponds to the measure $\mu$ and 
$$
\exp\Big(-\frac{1}{\alpha+2}\|\Re(u)\|_{L^{\alpha+2}(\Theta)}^{\alpha+2}\Big),
$$
corresponds to the density in \eqref{density}. Thanks to the conservation of the Hamiltonian \eqref{HHH}, the measure $\rho$ is expected to be invariant under the flow of \eqref{2}.
This expectation is also supported by the fact that the vector field defining \eqref{2} is (formally) divergence free.  This fact follows again from the Hamiltonian structure of \eqref{2}. 
\\

Observe that if a Borel set $A\subset H^s(\Theta)$ is of full $\rho$ measure
then $A$ is also of full $\mu$ measure. Therefore, it suffices to solve (\ref{2}) globally in time
for $u_0$ in a set of full $\rho$ measure.
\\

We now explain how the local existence result of Proposition~\ref{lwp} can be combined with invariant measure considerations in order to get global existence of the solution. 
The details can be found in \cite{BT2}.
Consider a truncated version of \eqref{2}
\begin{equation}\label{2_N}
(i\partial_{t}-\sqrt{ - \mathbf{\Delta}})u-
S_N\big(
\sqrt{ - \mathbf{\Delta}}^{-1}\big(|S_N \Re(u)|^{\alpha}S_N\Re(u)\big)
\big)
=0,
\end{equation}
where $S_N$ is a suitable "projector" tending to the identity as $N$ goes to infinity. 
Let us denote by $\Phi_N(t)$ the flow of \eqref{2_N}. This flow is well-defined for a fixed $N$ because for frequencies $\gg N$ it is simply the linear flow and for the remaining frequencies one can use that 
\eqref{2_N} has the preserved energy 
\begin{equation}\label{HHH_N}
\frac{1}{2}\|\sqrt{ - \mathbf{\Delta}} (u)\|_{L^2(\Theta)}^{2}+
\frac{1}{\alpha+2}\|S_N \Re(u)\|_{L^{\alpha+2}(\Theta)}^{\alpha+2}\,.
\end{equation}
The energy \eqref{HHH_N} allows us to define an approximated Gibbs measure $\rho_N$.
One has that $\rho_N$ is invariant under $\Phi_N(t)$ by the Liouville theorem and the invariance of complex Gaussians under rotations (for the frequencies $\gg N$).
In addition, $\rho_N$ converges in a strong sense to $\rho$ as $N\rightarrow\infty$. 

Let us fix $T\gg 1$ and a small $\epsilon>0$. Our goal is to find a set of $\rho$ residual measure $<\epsilon$ such that for initial data in this set we can solve \eqref{2} up to time $T$.

The local existence theory implies that as far as 
\begin{equation}\label{to_propagate}
 \|S(t)u\|_{L^p(\tau\in (0,2);L^p(\Theta))}\leq R,\quad R\geq 1
\end{equation}
we can define the solution of the true equation with datum $u$ for times of order $R^{-\gamma}$, $\gamma>0$, the bound \eqref{to_propagate} is propagated and moreover  {\it on the interval of existence
this solution is the limit  as $N\rightarrow\infty$ of the solutions of the truncated equation \eqref{2_N} with the same datum}. 

Our goal is to show that with a suitably chosen $R=R(T,\varepsilon)$ we can propagate the bound \eqref{to_propagate} for the solutions of the approximated equation \eqref{2_N}  (for $N\gg 1$) up to time $T$
for initial data in a set of residual $\rho$ measure $\lesssim \varepsilon$.

For $R>1$, we define the set $B_R$ as
$$
B_R=\{u\,:\,  
 \|S(t)u\|_{L^p(\tau\in (0,2);L^p(\Theta))}\leq R\}.
 $$
As mentioned the (large) number $R$ will be determined depending on $T$ and $\varepsilon$.
Thanks to the probabilistic Strichartz estimates for $S(t)$, we have the bound 
\begin{equation}\label{veneta9}
\rho(B_R^{c})<e^{-\kappa R^2}
\end{equation}
for some $\kappa >0$.  
Let $\tau\approx R^{-\gamma}$ be the local existence time associated to $R$ given by Proposition~\ref{lwp}.
Define the set $B$ by
\begin{equation}\label{veneta10}
B=\bigcap_{k=0}^{[T/\tau]}\Phi_N(-k\tau)(B_R)\, .
\end{equation}
Thanks to the local theory, we can propagate \eqref{to_propagate} for data in $B$ up to time $T$.
On the other hand, using the invariance of $\rho_N$ under $\Phi_N(t)$ and \eqref{veneta9}, we obtain that
$$
\rho_N(B^c)\lesssim T R^\gamma e^{-\kappa R^2}.
$$
We now choose $R$ so that 
$$
 T R^\gamma e^{-\kappa R^2}\sim \varepsilon.
 $$
 In other words 
 $$
 R \sim \Big(\log\big( \frac{T}{\varepsilon}\big)\Big)^{\frac{1}{2}}\,.
 $$
This fixes the value of $R$. 
With this choice of $R$, $\rho(B^c)<\varepsilon$, provided $N\gg 1$.
With this value of $R$ the set $B$ defined by \eqref{veneta10} is such that for data in $B$ we have the bound \eqref{to_propagate} up to time $T$ on a set of residual  $\rho$ measure $<\varepsilon$.
Now, we can pass to the limit $N\rightarrow \infty$ thanks to the above mentioned consequence of the local theory and hence defining the solution of the true equation \eqref{2} up to time $T$ for data in a set of $\rho$ residual measure $<\varepsilon$. 

We now apply the last conclusion with $T=2^j$ and $\varepsilon/2^j$. 
This produces a set $\Sigma_{j,\varepsilon}$ such that  $\rho((\Sigma_{j,\varepsilon})^c)<\varepsilon/2^j$ an for $u_0\in  \Sigma_{j,\varepsilon}$ we can solve \eqref{2} up to time $2^j$.
We next set
$$
\Sigma_{\varepsilon}=\bigcap_{j=1}^{\infty} \Sigma_{j,\varepsilon}\,.
$$
Clearly, we have $\rho((\Sigma_{\varepsilon})^c)<\varepsilon$ and for $u_0\in \Sigma_{\varepsilon}$, we can define a  global solution of \eqref{2}.
Finally
$$
\Sigma=\bigcup_{j=1}^\infty \Sigma_{2^{-j}}
$$
is a set of full $\rho$ measure on which we can define globally the solutions of \eqref{2}. 
The previous construction also keeps enough information allowing to get the claimed uniqueness property. 
\begin{rema}
{\rm
In \cite{BB}, the result of Theorem~\ref{thh1} was extended to $\alpha<4$ which is the full range of the definition of the measure $\rho$.
}
\end{rema}
\begin{rema}
{\rm
The previous discussion has shown that we have two methods to globalise the solutions in the context of random data well-posedness for the nonlinear wave equation.
The one of the previous chapter is based on energy estimates while the method of this chapter is based on invariant measures considerations. 
It is worth mentioning that these two methods are also employed in the context of singular stochastic PDE's.  More precisely in \cite{MW} the globalisation is done via the (more flexible) method of energy estimates while 
in \cite{HM} one globalises by exploiting invariant measures considerations.  
}
\end{rema}
\chapter{Quasi-invariant measures }
\section{Introduction}
\subsection{Motivation}
In Chapter~\ref{Chapter2},  for each $s\in (0,1)$ we introduced a family of measures ${\mathcal M}^s$ on the Sobolev space
${\mathcal H}^s(\T^3)=H^s(\T^3)\times H^{s-1}(\T^3)$. 
Then for each $\mu\in {\mathcal M}^s$, we succeeded to define a unique global flow $\Phi(t)$ of the cubic defocusing wave equation a.s. with respect to $\mu$. 
This result is of interest for the solvability of the Cauchy problem associated with the  cubic defocusing wave equation for data in ${\mathcal H}^s(\T^3)$, especially for $s<1/2$ because for these regularities this Cauchy problem is ill-posed in the Hadamard sense in ${\mathcal H}^s(\T^3)$. 
On the other hand the methods of Chapter~\ref{Chapter2} give no information about the transport by $\Phi(t)$ of the measures in ${\mathcal M}^s$, even for large $s$. Of course, ${\mathcal M}^s$ can be defined for any $s\in\R$ and for $s\geq 1$ the global existence a.s. with respect to an element of  ${\mathcal M}^s$  follows from Theorem~\ref{prop.global}.
The question of the transport of the measures of  ${\mathcal M}^s$ under  $\Phi(t)$  is of interest in the context of the macroscopic description of the flow of the cubic defocusing wave equation. 
It is no longer only a low regularity issue and the answer of this question is a priori  not clear at all for regular solutions either. 

On the other hand, in Chapter~\ref{Chapter3}, we constructed a very particular (Gaussian) measure $\mu$  on the Sobolev spaces of radial functions on the unit disc of $\R^3$ such that a.s. with respect to this measure the nonlinear defocusing wave equation with nonlinear term  $|u|^\alpha u$, $\alpha\in (2,3)$ has a well defined dynamics. The typical  Sobolev regularity on the support of this measure is supercritical and thus  again this result is of interest concerning the individual behaviour of the trajectories. This result is also of interest concerning the macroscopic description of the flow because, we can also prove by the methods of Chapter~\ref{Chapter3} that the transported measure by the flow is absolutely continuous with respect to $\mu$.  Unfortunately, the method of Chapter~\ref{Chapter3} is only restricted to a very particular initial distribution with data of low regularity.

Motivated by the previous discussion, a natural question to ask is what can be said for the transport of the measures of ${\mathcal M}^s$ under the flow of the cubic defocusing wave equation.  In this chapter we discuss some recent progress on this question.
\subsection{Statement of the result}
Consider the cubic defocusing wave equation
\begin{equation}\label{NLW}
(\partial_t^2 -\Delta) u+u^3=0,
\end{equation}
where $u :\R\times \T^d\rightarrow \R$. 
We rewrite \eqref{NLW} as the first order system
\begin{equation}\label{NLW-sys}
\partial_t u=v,\quad \partial_t v=\Delta u-u^3.
\end{equation}
As we already know, if $(u,v)$ is a smooth solution of \eqref{NLW-sys} then
$$
\frac{d}{dt}H(u(t),v(t))=0,
$$
where 
\begin{equation}\label{H}
H(u,v)=\frac{1}{2}\int_{\T^d}\big(v^2+|\nabla u|^2\big)+\frac{1}{4}\int_{\T^d}u^4\,.
\end{equation}
Thanks to Theorem~\ref{prop.global}, for $d\leq 3$ the Cauchy problem associated with \eqref{NLW-sys} is globally well-posed in  ${\mathcal H}^s(\T^d)= H^s(\T^d)\times H^{s-1}(\T^d),$ $s\geq 1$.
Denote by $\Phi(t):{\mathcal H}^s(\T^d)\rightarrow {\mathcal H}^s(\T^d)$ the resulting flow.  As we already mentioned, we are interested in the statistical description of $\Phi(t)$.
Let  $\mu_{s,d}$  be the measure  formally defined by
$$
 d \mu_{s,d} 
  = Z_{s,d}^{-1} e^{-\frac 12 \| (u,v)\|_{{\mathcal H}^{s+1}}^2} du dv
 $$ 
 or
 $$
 d \mu_{s,d} 
  =
 Z_{s,d}^{-1} \prod_{n \in \Z^2} 
 e^{-\frac 12 \langle n\rangle ^{2(s+1)} |\ft u_n|^2}   
  e^{-\frac 12 \langle n\rangle^{2s} |\ft v_n|^2}   
 d\ft u_n d\ft v_n\,,
$$
where $\ft u_n$ and $\ft v_n$ denote the Fourier transforms of $u$ and $v$ respectively. Recall that $\langle n\rangle=(1+|n|^2)^{\frac{1}{2}}$.

Rigorously one can define the Gaussian measure  $\mu_{s,d}$  as  the induced probability measure under the map
$$
\omega  \longmapsto (u^\omega(x), v^\omega(x))
$$
with 
\begin{equation}\label{series}
u^\omega(x) = \sum_{n \in \Z^d} \frac{g_n(\omega)}{\langle n\rangle^{s+1}}e^{in\cdot x},\quad
v^\omega(x) = \sum_{n \in \Z^d} \frac{h_n(\omega)}{\langle n\rangle^{s}}e^{in\cdot x} \,\,.
\end{equation}
In \eqref{series},  $( g_n )_{n \in \Z^d}$,  $( h_n )_{n \in \Z^d}$
are two sequences of "standard" complex Gaussian random variables, such that $g_n=\overline{g_{-n}}$,  $h_n=\overline{h_{-n}}$  and such that 
$
\{g_n, h_n\}
$
are independent, modulo the central symmetry.  The measures $\mu_{s,d}$ can be seen as special cases of the measures in ${\mathcal M}^s$ considered in Chapter~\ref{Chapter2}.
The partial sums of the series in \eqref{series} are a Cauchy sequence in 
$L^2(\Omega; \mathcal{H}^{\sigma}(\T^d))$ for every $\sigma<s+1-\frac{d}{2}$ and therefore one can see $\mu_{s,d}$ as a probability measure on ${\mathcal H}^\sigma$ for a fixed $\sigma<s+1-\frac{d}{2}$.
Therefore, thanks to the results of   Chapter~\ref{Chapter2},  for $d\leq 3$, the flow $\Phi(t)$ can be extended $\mu_{s,d}$ almost surely, provided $s>\frac{d}{2}-1$. 
We have the following result.
\begin{theo}\label{dim1}
Let $s\geq 0$ be an integer. 
Then the measure $\mu_{s,1}$ is quasi-invariant under the flow of \eqref{NLW-sys}.
\end{theo}
We recall that given a measure space $(X, \mu)$,  we say that $\mu$ is {\it quasi-invariant} under a transformation $T:X \to X$ 
if  the transported measure $T_*\mu = \mu\circ T^{-1}$ and $\mu$ are equivalent, i.e.~mutually absolutely continuous with respect to each other.
The proof  of Theorem~\ref{dim1} is essentially contained in the analysis of \cite{TzBBM}. 

For $d=2$  the situation is much more complicated. Recently in \cite{OhTz}, we were able to prove the following statement.
\begin{theo}\label{dim2}
Let $s\geq 2$ be an even integer. 
Then the measure $\mu_{s,2}$ is quasi-invariant under the flow of \eqref{NLW-sys}.
\end{theo}
We expect that by using the methods of Chapter~\ref{Chapter2}, one can extend the result of Theorem~\ref{dim2} to all $s>0$, not necessarily an integer.
\\

It would be interesting to decide whether one can extend the result of Theorem~\ref{dim2} to the three dimensional case. It could be that the type of renormalisations employed in the context of singular stochastic PDE's or the QFT become useful in this context.  
\\

From now on we consider $d=2$ and we denote $\mu_{s,2}$ simply by $\mu_s$.
\subsection{Relation to Cameron-Martin type results}
In probability theory, there is an extensive literature on the transport property  of Gaussian measures  under linear and nonlinear transformations.
The statements of Theorem~\ref{dim1} and Theorem~\ref{dim2}  can be seen as such kind of results for the nonlinear transformation defined by the flow map of the cubic defocusing wave equation. 
The most classical result concerning the transport property  of Gaussian measures is the result of Cameron-Martin \cite{CM} giving an optimal answer concerning the shifts.  
The Cameron-Martin theorem  in the context of the measures $\mu_s$ is saying that for a fixed $(h_1,h_2)\in {\mathcal H}^{\sigma}$, $\sigma<s$,  the transport of $\mu_s$ under the shift
$$
(u,v)\longmapsto (u,v)+(h_1,h_2),
$$ 
is absolutely continuous with respect to $\mu_s$ if and only if $(h_1,h_2)\in {\mathcal H}^{s+1}$.

If we denote by $S(t)$ the free evolution  associated with \eqref{NLW-sys}  then for $(u,v)\in {\mathcal H}^\sigma$, we classically have that the flow of the nonlinear wave equation can be decomposed as 
\begin{equation}\label{wave_dec}
\Phi(t)(u,v)=S(t)\big((u,v)+(h_1,h_2)\big),
\end{equation}
where $(h_1,h_2)= (h_1(u,v),h_2(u,v)) \in {\mathcal H}^{\sigma+1}$. In other word there is  one derivative smoothing and no more.
Of course,  if $\sigma<s$ then $\sigma+1<s+1$ and therefore the result of Theorem~\ref{dim2} represents  a statement displaying fine properties of the vector field generating $\Phi(t)$. 
More precisely if in \eqref{wave_dec} $(h_1,h_2)\in  {\mathcal H}^{\sigma+1}$ were fixed (independent of $(u,v)$) then the transported measures would be singular with respect to $\mu_s$ !
\\

Let us next compare the result of Theorem~\ref{dim2} with a result of Ramer \cite{Ra}.
For $\sigma<s$, let us consider an invertible map $\Psi$ on ${\mathcal H}^{\sigma}(\T^2)$ of the form 
$$
\Psi(u,v)= (u,v)+F(u,v),
$$
where 
$F:{\mathcal H}^\sigma(\T^2)\rightarrow {\mathcal H}^{s+1}(\T^2)$.
Under some more assumptions, the most important being that
$$
DF(u,v):{\mathcal H}^{s+1}(\T^2)\rightarrow {\mathcal H}^{s+1}(\T^2)
$$ 
is a  Hilbert-Schmidt map, the analysis of \cite{Ra} implies that  $\mu_s$ is quasi-invariant under $\Psi$.
A  typical example for the $F$ is
$$
F(u,v)=\varepsilon (1-\Delta)^{-1-\delta}(u^2,v^2),\quad \delta>0,\,\, |\varepsilon|\ll 1,
$$
i.e. $2$-smoothing is needed in order to ensure the Hilbert-Schmidt assumption.  
Therefore the approach of Ramer is far from being applicable in the context of the flow map of the  nonlinear wave equation because for the nonlinear wave equation there is only $1$-smoothing. 
\\

Let us finally discuss the Cruzeiro generalisation of the Cameron-Martin theorem. In \cite{Cru}, Ana Bela Cruzeiro considered a general equation of the form
\begin{equation}\label{ABC}
\partial_t u=X(u),
\end{equation}
where $X$ is an infinite dimensional vector field. 
She proved that $\mu_{s}$ would be quasi-invariant under the flow of \eqref{ABC} if we suppose a number of assumptions, the most important being of type :
\begin{equation}\label{exp}
\int_{H^{\sigma}(\T^2)}e^{{\rm div} (X(u))}d\mu_s(u)<\infty\,.
\end{equation}
The problem is how to check the abstract assumption \eqref{exp} for concrete examples.
Very roughly speaking the result of Theorem~\ref{dim2} aims to verify assumptions of type \eqref{exp}  "in practice".
\section{Elements of the proof}
In this section, we present some of the key steps in the proof of Theorem~\ref{dim2}. 
\subsection{An equivalent Gaussian measure}
Since the quadratic part of \eqref{H} does not control the $L^2$ norm of $u$, we will prove the quasi-invariance for the equivalent measure 
$\widetilde{ \mu}_{s} $  defined as  the induced probability measure under the map
\begin{equation*}
\omega \in \Omega \longmapsto (u^\omega(x), v^\omega(x))
 \end{equation*}
with 
\begin{equation*}
u^\omega(x) = \sum_{n \in \Z^2} \frac{g_n(\omega)}
{
(1+|n|^2+|n|^{2s+2})^{\frac{1}{2}}
}e^{in\cdot x},\quad
v^\omega(x) = \sum_{n \in \Z^2} 
\frac{h_n(\omega)}{
(1+|n|^{2s})^{\frac{1}{2}}
}e^{in\cdot x}\,.
\end{equation*}
Formally $\widetilde{ \mu}_{s} $ can be seen as 
$$
Z^{-1} e^{-\frac 12 \int v^2- \frac 12  \int (D^s v)^2 
-  \frac 12 \int  u^2 
-  \frac 12 \int  |\nabla u|^2 
-  \frac 12 \int (D^{s+1} u)^2}du dv,
$$
where 
$$
D\equiv \sqrt{-\Delta}\,.
$$
As we shall see below, the expression
\begin{equation}\label{lili}
\frac 12 \int_{\T^2} v^2+\frac 12  \int_{\T^2} (D^s v)^2 +  \frac 12 \int_{\T^2}  u^2 + \frac 12 \int_{\T^2}  |\nabla u|^2 + \frac 12 \int_{\T^2} (D^{s+1} u)^2
\end{equation}
is the main part of the quadratic part of the renormalised energy in the context of the  nonlinear wave equation  \eqref{NLW-sys}.
Using the result of Kakutani \cite{Kakutani}, we can show that for $s>1/2$ the Gaussian measures $\mu_s$ and $\widetilde{ \mu}_{s} $ are equivalent. 
\subsection{The renormalised energies}
Consider the truncated wave equation 
\begin{equation}\label{NLW-sys_N}
\partial_t u=v,\quad \partial_t v=\Delta u-\pi_N((\pi_N u)^3),
\end{equation}
where $\pi_N$ is the Dirichlet projector on frequencies $n\in\Z^2$ such that $|n|\leq N$.
If $(u,v)$ is a solution of  \eqref{NLW-sys_N} then
\begin{equation*}
\partial_t \Big[\frac 12  \int_{\T^2} (D^s v_N)^2 + \frac 12 \int_{\T^2} (D^{s+1}  u_N)^2
\Big]  =  \int_{\T^2} (D^{2s} v_N)(-u_N^3)\,,
\end{equation*}
where 
$
(u_N,v_N)=(\pi_N u,\pi_N v).
$
Clearly $\partial_t u_N=v_N$. 
Observe that for $s=0$, we recover the conservation of the truncated energy $H_N(u,v)$, defined by
\begin{equation*}
H_N(u,v)\equiv H(\pi_N u,\pi_N v)\,.
\end{equation*}
For $s\geq 2$, an even integer, using the Leibniz rule, we get 
\begin{multline*}
\int_{\T^2} (D^{2s} v_N)(-u_N^3)
=
\\
-3\int_{\T^2} D^sv_N\, D^s u_N\, u_N^2+
\sum_{\substack{ |\alpha|+|\beta|+|\gamma|=s\\
|\alpha|,|\beta|,|\gamma|<s
}}
c_{\alpha,\beta,\gamma}
\int_{\T^2}
D^sv_N\,\partial^\alpha u_N \partial^\beta u_N \partial^\gamma u_N,
\end{multline*}
for some unessential constants $c_{\alpha,\beta,\gamma}$.

It will be convenient in the sequel to suppose that the integration on $\T^2$ is done with respect to a probability measure. Therefore the integrations will be done with respect to 
 the Lebesgue measure multiplied by $(2\pi)^{-2}$.  We can write 
\begin{multline}\label{psg}
-3\int_{\T^2} D^sv_N\, D^s u_N\, u_N^2= - \frac 32 \partial_t \bigg[\int_{\T^2} (D^s u_N)^2u_N^2 \bigg]
+ 3 \int_{\T^2} (D^s u_N)^2\, v_N \,u_N 
\\
=
  - \frac 32 \partial_t\bigg[ \int_{\T^2} \Pi_{0}^\perp [(D^su_N)^2] \,  \Pi_{0}^\perp[u_N^2] \bigg]
+ 3 \int_{\T^2} \Pi_{0}^\perp[(D^su_N)^2]\, \Pi_{0}^\perp[v_N \, u_N]  
\\
-   \frac 32 \partial_t \bigg[\int_{\T^2} (D^s u_N)^2 \int_{\T^2}  u_N^2\bigg]
+   3  \int_{\T^2} (D^s u_N)^2\int_{\T^2} v_N\,  u_N .
\end{multline}
where $\Pi_{0}^\perp$ is  again the projector on the nonzero frequencies, i.e.
$$
(\Pi_{0}^\perp(f))(x)=f(x)-\int_{\T^2} f(y)dy\,.
$$
The last two terms on the right-hand side of \eqref{psg} are problematic because 
$$
\lim_{N\rightarrow\infty} \E_{\widetilde{\mu}_s}
\bigg[\int_{\T^2} (D^s \pi_{N}u)^2\bigg]=+\infty\,.
$$
Therefore, we need to use a renormalisation in the definitions of the energies. 
Define $ \sigma_N$ by
\[ 
\sigma_N = \E_{\widetilde{\mu}_s}
\bigg[\int_{\T^2} (D^s \pi_{N}u)^2\bigg] = \sum_{\substack{n \in \Z^2\\|n|\leq N}} \frac{|n|^{2s}}{1+|n|^2+|n|^{2s+2}}\sim \log N \,.
\]
Then, we have
\begin{align*}
 -   \frac 32 & \partial_t \bigg[\int_{\T^2} (D^s u_N)^2 \int_{\T^2}  u_N^2\bigg]
+   3  \int_{\T^2} (D^s u_N)^2\int  v_N\, u_N \notag\\
& = 
 -   \frac 32 \partial_t \bigg[\bigg(\int_{\T^2} (D^s u_N)^2- \sigma_N\bigg)\int_{\T^2}  u_N^2\bigg]
+   3 \bigg( \int_{\T^2} (D^s u_N)^2 - \sigma_N\bigg)\int_{\T^2}  v_N\,  u_N.
\end{align*}
Now, the term
$$
\int_{\T^2} (D^s u_N)^2-  \sigma_N
$$
is a good term because thanks to Wiener chaos estimates, we have the bound 
$$
\Big\|
\int_{\T^2} (D^s \pi_N u)^2- \sigma_N
\Big\|_{L^p(d \widetilde{\mu}_s(u,v))}\leq C p,
$$
where the constant $C$ is independent of $p$ and $N$.
We define $\tilde{H}_{s, N}(u,v)$ by 
\begin{equation*}
\tilde{H}_{s, N}(u,v)  = \frac 12  \int_{\T^2} (D^s v)^2 + \frac 12 \int_{\T^2} (D^{s+1}  u)^2+ \frac 32  \int_{\T^2} (D^s u)^2  u^2 - \frac 32 \sigma_N  \int_{\T^2} u^2\notag\,.
\end{equation*}
We can summarise the previous analysis as follows :  if $(u,v)$ is a solution of \eqref{NLW-sys_N}  then
\begin{multline}\label{H8_pak}
\partial_t \tilde{H}_{s, N}(u_N,v_N) 
 =  
 3 \int_{\T^2}  \Pi_{0}^\perp[ (D^s u_N)^2  ] \, \Pi_{0}^\perp [v_N \, u_N ]
+ 
\\
\sum_{\substack{ |\alpha|+|\beta|+|\gamma|=s\\
|\alpha|,|\beta|,|\gamma|<s
}}
c_{\alpha,\beta,\gamma}
\int_{\T^2}
D^sv_N\,\partial^\alpha u_N \partial^\beta u_N \partial^\gamma u_N+
\\
  3 \bigg( \int_{\T^2} (D^s u_N)^2 - \sigma_N\bigg)\int_{\T^2}  v_N\,  u_N. 
\end{multline}
All terms in the right hand-side of \eqref{H8_pak} are suitable for a perturbative analysis.
We finally define  the full modified energy $H_{s, N}(u,v)$ as
$$
H_{s, N}(u,v)=\tilde{H}_{s, N}(u,v)+H(u,v)+ \frac{1}{2} \int_{\T^2} u^2,
$$
where $H$ is defined by \eqref{H}.
The quadratic part of $H_{s,N}$ (except the renormalisation term which is morally quartic) is now given by \eqref{lili}.
Therefore in order to prove the quasi-invariance it will be of crucial importance to study the variation in time of $H_{s,N}$.
Here is the main quantitative bound used in the proof of Theorem~\ref{dim2}.
\begin{theo}\label{key_key}
Let $s\geq 2$ be an even integer and let us denote by $\Phi_N(t)$ the flow of 
$$
\partial_t u=v,\quad \partial_t v=\Delta u-\pi_N((\pi_N u)^3)\,.
$$
Then  for every $r>0$ there is a constant $C$ such that for every $p\geq 2$ and every $N\geq 1$,
$$
\Big(\int_{H_N(u,v)\leq r} 
\Big|\partial_t H_{s, N}(\pi_N\Phi_N(t)(u,v))\vert_{t=0} \Big|^p d\widetilde{\mu}_s(u,v)\Big)^{\frac{1}{p}}\leq Cp.
$$
\end{theo}
\subsection{On the proof of Theorem~\ref{key_key}}
Using  the equation \eqref{NLW-sys_N}, we have that
$$
\partial_t H_{s, N}(u_N,v_N) =\partial_t\tilde{H}_{s, N}(u_N,v_N) +\int_{\T^2} u_N v_N.
$$
Therefore, coming back to \eqref{H8_pak}, we obtain 
$$
\partial_t  \tilde{H}_{s, N}(\pi_N\Phi_N(t)(u,v))\vert_{t=0}= \int_{\T^2}\pi_N u\pi_N v+  Q_1(u,v)+Q_2(u,v)+Q_3(u,v),
$$
where 
\begin{multline*}
Q_1(u,v)=
3 \int_{\T^2}  \Pi_{0}^\perp[ (D^s \pi_N u)^2  ] \, \Pi_{0}^\perp [\pi_N v \, \pi_N u],
\\
Q_2(u,v)=
\sum_{\substack{ |\alpha|+|\beta|+|\gamma|=s\\
|\alpha|,|\beta|,|\gamma|<s
}}
c_{\alpha,\beta,\gamma}
\int_{\T^2}
D^s \pi_N v\,\partial^\alpha \pi_N u \partial^\beta \pi_N u \partial^\gamma \pi_N u,
\\
Q_3(u,v)=
 3 \bigg( \int_{\T^2} (D^s u_N)^2 - \sigma_N\bigg)\int_{\T^2}  \pi_N v\, \pi_N u.
 \end{multline*}
Let us first consider 
\begin{equation}\label{NEW}
\int_{\T^2} \pi_N  u \pi_N v\,.
\end{equation}
We need to estimate \eqref{NEW} under the restriction
\begin{equation}\label{restriction}
\int_{\T^2}(|\nabla\pi_N u|^2+(\pi_N v)^2+\frac{1}{2}(\pi_N u)^4)\leq 2r.
\end{equation}
Using the compactness of $\T^2$, one can see that under the restriction \eqref{restriction},
$$
\big|\int_{\T^2} \pi_N  u\pi_N v\big|\leq 
\| \pi_N u\|_{L^2(\T^2)}
\|\pi_N v\|_{L^2(\T^2)}
\leq  C\|\pi_N u\|_{L^4(\T^2)}\|\pi_N v\|_{L^2(\T^2)}\leq Cr^{\frac{3}{4}}\,.
$$
Let us next consider $Q_3(u,v)$. For $r>0$, we define $\mu_{s,r,N}$ as 
$$
d \mu_{s,r,N}(u,v)= \chi_{H_N(u,v) \leq r} \, d \widetilde{\mu}_s(u,v)\,,
$$
where $ \chi_{H_N(u,v) \leq r} $ stays for the characteristic function of the set 
$$
\{(u,v)\,:\, H_N(u,v)\leq r\}.
$$
The goal is to show that
$$
\|Q_3(u,v)\|_{L^p(d\mu_{s,r,N}(u,v))}\leq Cp,
$$
with a constant $C$ independent of $N$ and $p$. Since  we already checked that under \eqref{restriction},
$$
\Big|
\int_{\T^2}  \pi_N v\,  \pi_N u
\Big|
 \leq C_r,
$$
we obtain that
\begin{multline*}
\|Q_3(u,v)\|_{L^p(d\mu_{s,r,N}(u,v))}\leq C_r\, \Big\| \int_{\T^2} (D^s \pi_N u)^2 - \sigma_N\Big\|_{L^p(d\mu_{s,r,N}(u,v))}
\\
\leq C_r\, \Big\| \int_{\T^2} (D^s \pi_N u)^2 - \sigma_N\Big\|_{L^p(d\widetilde{\mu}_{s}(u,v))}\,.
\end{multline*}
On the other hand
$$
 \Big\| \int_{\T^2} (D^s \pi_N u)^2 - \sigma_N\Big\|_{L^p(d\widetilde{\mu}_{s}(u,v))}=
 \Big\|\sum_{\substack{n \in \Z^2\\|n|\leq N}} 
 \frac{(|g_n(\omega)|^2 - 1)|n|^{2s}}{1+|n|^2+|n|^{2s+2}}
  \Big\|_{L^p(\Omega)}
 $$
and by using Wiener chaos estimates, we have 
$$
 \Big\|\sum_{\substack{n \in \Z^2\\|n|\leq N}} 
 \frac{(|g_n(\omega)|^2 - 1)|n|^{2s}}{1+|n|^2+|n|^{2s+2}}
  \Big\|_{L^p(\Omega)}
  \leq Cp
 \Big\|\sum_{\substack{n \in \Z^2\\|n|\leq N}} 
 \frac{(|g_n(\omega)|^2 - 1)|n|^{2s}}{1+|n|^2+|n|^{2s+2}}
  \Big\|_{L^2(\Omega)}
  \leq Cp
$$
which provides the needed bound for $Q_3(u,v)$.

The analysis of 
$$
Q_1(u,v)=
3 \int_{\T^2}  \Pi_{0}^\perp[ (D^s \pi_N u)^2  ] \, \Pi_{0}^\perp [\pi_N v \, \pi_N u]
$$
is the most delicate part of the analysis and relies on subtle multi-linear arguments.  The analysis of $Q_2(u,v)$ follows similar lines. 

Basically, we are allowed to have outputs as
$$
\| D^\sigma u\|_{L^\infty(\T^2)},\quad \sigma<s
$$
with a loss $\sqrt{p}$ and $H_N(u,v)$ with no loss in $p$.
The outputs $H_N(u,v)$ follow from deterministic analysis and thus have no loss in $p$ but they are regularity consuming.
 
We observe that a naive H\"older inequality approach clearly fails.  
A purely probabilistic argument based on Wiener chaos estimates fails because the output power of $p$ is too large.
The basic strategy is  therefore to perform a multi-scale analysis redistributing properly the derivative losses by never having more then quadratic weight of the contribution of the Wiener chaos estimate. 

When analysing the $4$-linear expression defining $Q_1(u,v)$, we suppose that 
$$
D^s  \pi_N u,\,\, D^s  \pi_N u,\,\, \pi_N v,\,\, \pi_N u
$$ 
are localised at dyadic frequencies $N_1$, $N_2$, $N_3$, $N_4$ respectively. 

We first consider the case when $N_4\gtrsim (\max(N_1,N_2))^{\frac{1}{100}}$. In this case we exchange some regularity of $D^s  \pi_N u$ with this of $\pi_N u$ and we perform the naive linear analysis.

Therefore, in the analysis of $Q_1(u,v)$ we can suppose that  
$$
N_4\ll (\max(N_1,N_2))^{\frac{1}{100}}.
$$ 
In this case, we have that
$$
\max(N_1,N_2)\sim \max(N_j,\,  j=1,2,3,4).
$$
By symmetry, we can suppose that $N_1=\max(N_1,N_2)$. 
We next consider the case 
$$
N_3\ll N_1^{1-a},\quad a=a(s)\ll 1\,.
$$
In this case, we perform a bi-linear Wiener chaos estimate and we have some gain of regularity in the localisation of  $ \Pi_{0}^\perp[ (D^s \pi_N u)^2  ]$. 
Finally, we consider the case 
$$
N_1\sim \max(N_j,\,  j=1,2,3,4),\,\, N_4\ll (\max(N_1,N_2))^{\frac{1}{100}},\,\,
N_3\gtrsim N_1^{1-a}
$$
In this case, we perform a tri-linear Wiener chaos estimate and we have enough gain of regularity in the localisation of  
$$ 
\Pi_{0}^\perp[ (J^s \pi_N u)^2  ]\pi_N v\,.
$$
This essentially explains the argument leading to the key estimate of  Theorem~\ref{key_key}.
We refer to \cite{OhTz} for the details. 
\subsection{On the soft analysis}
We can observe that 
\begin{equation*}
H_{s, N}(u,v)  =\eqref{lili}
+ \frac 32  \int_{\T^2} (D^s u)^2  u^2 - \frac 32 \sigma_N  \int_{\T^2} u^2
+\int_{\T^2}u^4\,.
\end{equation*}
By classical arguments from QFT, we can define 
$$
\lim_{
N\rightarrow\infty
}
\Big(
\frac 32  \int (D^s \pi_N u)^2 (\pi_N u)^2 - \frac 32 \sigma_N  \int (\pi_N u)^2
\Big)
$$
in
$L^p(d\widetilde{\mu}_s(u,v))$, $p<\infty$. 
Denote this limit by $R(u)$. 
Essentially speaking, once we have the key estimate, we study the quasi-invariance of
\begin{equation}\label{renren}
\chi_{H(u,v) \leq r}\,  e^{-  R(u)-\int u^4}\,d\widetilde{\mu}_s(u,v)
\end{equation}
by soft analysis techniques. 

Let us  finally explain the importance of the loss $p$ in the key estimate of Theorem~\ref{key_key}. 
Denote by $x(t)$ the measure evolution of a set having zero measure with respect to \eqref{renren}.
Essentially speaking, using the key estimate and the arguments introduced in \cite{TV,TVjmpa}, we obtain that $x(t)$ satisfy the estimate 
\begin{equation}\label{yuyu}
\dot{x}(t)\leq Cp(x(t))^{1-\frac{1}{p}},\quad x(0)=0\,.
\end{equation}
Integrating the last estimate leads to $x(t)\leq (Ct)^p$. Taking the limit $p\rightarrow \infty$, we infer that $x(t)=0$ for $0\leq t<1/C$. Since $C$ is an absolute constant, we can iterate the argument and show that $x(t)$ is vanishing. 
Observe that this argument would not work if in \eqref{yuyu}, we have $p^{\alpha}$, $\alpha>1$ instead of $p$.
In order to make the previous reasoning rigorous, we need to use some more or less standard approximation arguments. We refer to  \cite{TzBBM} and \cite{OhTz} for the details of such type of reasoning. 
\subsection{Acknowledgement} 
I am grateful to Leonardo Tolomeo, Tadahiro Oh and Yuzhao Wang for their remarks on the manuscript.  
I am very grateful to Chenmin Sun for pointing our an error in a previous version of Lemma~\ref{compar}.
I am particularly indebted to Nicolas Burq and Tadahiro Oh since this text benefitted from the discussions we had on the topics discussed in the lectures. 
I am grateful to  Franco Flandoli  and  Massimiliano Gubinelli for inviting me to give these lectures.

\end{document}